\documentclass[11pt,letterpaper]{amsart}

\usepackage{amsmath, amsthm, amssymb}
\usepackage{amsfonts}
\usepackage{mathtools}

\usepackage[cal=boondox]{mathalfa}
\usepackage{bbm}
\usepackage{mathrsfs}
\usepackage{stmaryrd}
\usepackage{esint}
\usepackage[T1]{fontenc}
\usepackage{lmodern}

\usepackage{graphicx}
\usepackage{tikz}
\usepackage{tikz-cd}
\usetikzlibrary{arrows}
\usepackage[all,cmtip]{xy}
\usepackage{amscd}
\usepackage{picinpar}

\usepackage{enumerate}
\usepackage{enumitem}
\usepackage{comment}
\usepackage{array,booktabs,longtable,ragged2e}
\usepackage{extarrows}
\usepackage{verbatim}
\usepackage{calc}
\usepackage{xcolor}

\usepackage{hyperref}

\numberwithin{equation}{section}

\theoremstyle{plain}
\newtheorem{thm}{Theorem}
\numberwithin{thm}{subsection}
\newtheorem{cor}[thm]{Corollary}
\newtheorem{lem}[thm]{Lemma}
\newtheorem{prop}[thm]{Proposition}

\newtheorem{mainthm}{Theorem}

\theoremstyle{definition}

\theoremstyle{remark}
\newtheorem{rem}[thm]{Remark}

\newcolumntype{L}[1]{>{\RaggedRight\arraybackslash}p{#1}}
\DeclareMathOperator{\GL}{GL}
\DeclareMathOperator{\SL}{SL}

\DeclareMathOperator{\PU}{PU}
\DeclareMathOperator{\PGL}{PGL}
\DeclareMathOperator{\PSL}{PSL}
\DeclareMathOperator{\PN}{PN}

\DeclareMathOperator{\Proj}{P}

\DeclareMathOperator{\tr}{tr}

\DeclareMathOperator{\Stab}{Stab}

\DeclareMathOperator{\Ann}{Ann}
\DeclareMathOperator{\Fix}{Fix}
\DeclareMathOperator{\Cl}{Cl}

\def\bC{{\mathbb{C}}}

\def\bR{{\mathbb{R}}}

\def\P{{\mathbb{P}}}

\def\bF{{\mathbb{F}}}
\def\bH{{\mathbb{H}}}

\def\bZ{{\mathbb{Z}}}

\def\bP{{\mathbb{P}}}
\def\bN{{\mathbb{N}}}
\def\bQ{{\mathbb{Q}}}
\def\bK{{\mathbb{K}}}

\def\cO{{\mathcal{O}}}
\def\cG{{\mathcal{G}}}
\def\cH{{\mathcal{H}}}

\def\cA{{\mathscr{A}}}

\def\cC{{\mathscr{C}}}

\def\cE{{\mathscr{E}}}

\def\cM{{\mathscr{M}}}

\def\cT{{\mathscr{T}}}

\def\fn{{\mathfrak{n}}}

\def\fd{{\mathfrak{d}}}

\def\fp{{\mathfrak{p}}}
\def\fb{{\mathfrak{b}}}
\def\fa{{\mathfrak{a}}}
\def\fq{{\mathfrak{q}}}

\begin{document}

\title[Decay for cusps of congruence subgroups]{Decay of weighted cusp counts for congruence subgroups of \(\SL_2\) over number fields}

\author{Shengyuan Zhao}
\address{Universit\'e Paul Sabatier, Institut de Math\'ematiques de Toulouse, 118, route de Narbonne, F-31062 Toulouse, France}
\email{shengyuan.zhao@math.univ-toulouse.fr}

\begin{abstract}
For congruence subgroups commensurable with $\operatorname{SL}_2$
over number fields, we study cusp counts with arithmetic multiplicities. We prove that the ratio of the total weighted cusp
count to the group index is bounded by a negative power of the norm of the congruence level, with an exponent that can be chosen explicitly in terms of the degree of the number field. This is a generalization of a theorem of Cox--Parry over rational numbers. The main input is an explicit local orbit estimate for arbitrary subgroups of exact level in
$\operatorname{SL}_2(\mathcal O_K/\mathfrak p^e)$, uniform over all number fields of fixed degree. This estimate is covered by a general result of Finis--Lapid~\cite{FinisLapid}. We give a new explicit proof in our setting based on an analysis reminiscent of additive combinatorics. 
\end{abstract}

\maketitle

\section{Introduction}

We study non-cocompact irreducible arithmetic lattices in Lie groups of the form $\PSL_2(\bR)^r\times \PSL_2(\bC)^s$, acting on a product of two or three dimensional real hyperbolic spaces $(\bH^2)^r\times (\bH^3)^s$. This amounts to considering arithmetic subgroups of $\SL_2(\bK)$ where $\bK$ is a number field. Serre's theorem~\cite{SerreSL2CSP} asserts that all such arithmetic groups are congruence when $\bK$ is neither $\bQ$ nor an imaginary quadratic field. We are interested in the asymptotic number of cusps when the levels of the congruence subgroups grow. 

For \(\bK=\bQ\) Thompson~\cite{Thompson} proved the finiteness of conjugacy classes of congruence subgroups of fixed genus, while Cox--Parry~\cite{CoxParry1984} obtained effective bounds for the level by delicate $p$-adic arguments. Duque--Rosero and Voight~\cite{DuqueRoseroVoight} later asked whether the methods of Cox--Parry could be extended to other settings.

In this paper we generalize Cox--Parry's result to \(\SL_2\) over arbitrary number fields. We also add one layer of complexity: to each cusp $\sigma$ we associate a multiplicity $a_\sigma\in \bN^*$ (see \S\ref{sec:cusps-hirzebruch-multiplicity} for definition), and we count the sum of the multiplicities. This sum is in general a higher order quantity than the bare number of cusps because multiplicities may grow when the congruence level grows (see \S\ref{sec:examples-cusp-multiplicity} for examples). For $\bK=\bQ$ all multiplicities are one, while for a real quadratic $\bK$ it is the multiplicity that Hirzebruch~\cite{Hirzebruch1973} introduced in his study of resolution of singularities of Hilbert modular surfaces. Our main result is the following:

\begin{mainthm}[Theorem~\ref{thm:quantitative-global-cusp-decay}]\label{mainthmA}
Let $\bK$ be a number field. There are explicit constants \(C,c>0\), depending only on \(\bK\), such that, for every congruence subgroup \(\Gamma\) of $\SL(2,\cO_\bK)$ with level ideal $\fn$, we have
\[
 \sum_{\sigma \;\text{cusp of}\; \Gamma} a_\sigma(\Gamma)\le C\,N(\fn)^{-c}[\SL(2,\cO_\bK):\Gamma],
\]
where the sum runs over all cusps of $\Gamma$. In particular, the number of cusps of $\Gamma$ satisfies the same upper bound.
\end{mainthm}

For principal congruence subgroups or for special classical congruence subgroups such as $\Gamma_0(\fn),\Gamma_1(\fn)$, the ratio \((\sum a_\sigma)/[\cG:\Gamma]\) can sometimes be computed explicitly, and the asymptotic estimate follows, see \S\ref{sec:examples-cusp-multiplicity}. The difficulty addressed here is that
an arbitrary congruence subgroup of a fixed level may have a much less rigid
image in the finite congruence quotient. The main work is therefore to obtain
a uniform estimate for all subgroups of the relevant finite quotients with prescribed level. Theorem~\ref{mainthmA} over \(\bK=\bQ\) is obtained by Cox--Parry~\cite{CoxParry1984}. The central parts of our method deal with difficulties not present over $\bQ$.

Let us mention some qualitative or quantitative forms of cusp sublinearity known in other settings. For two or three dimensional real hyperbolic orbifolds Frączyk and Raimbault~\cite{Raimbault,Fraczyk,FraczykRaimbault}  obtained Benjamini--Schramm convergence results which imply that the cusp number is dominated by the volume. Note also that the Benjamini--Schramm convergence is done for general arithmetic groups in a very recent preprint \cite{FraczykHurtadoRaimbault} of Fr\k{a}czyk--Hurtado--Raimbault. Di Cerbo--Stern~\cite{DicerboStern} obtained bounds for cusp numbers for both real- and complex-hyperbolic manifolds. In contrast, Theorem~\ref{mainthmA} is uniform over all congruence subgroups of a prescribed level and controls also the weights. 

In the complex hyperbolic \(\PU(1,n)\)-setting, universal bounds for the number of cusps in terms of volume were obtained by Parker~\cite{Parker}, Hwang~\cite{Hwang},
Di Cerbo--Di Cerbo~\cite{DiCerboDiCerbo-sharp,DiCerboDiCerbo-effective},
and Bakker--Tsimerman~\cite{BakkerTsimerman}. These results give linear bounds in the volume,
rather than decay of the ratio in congruence towers. Actually without congruence hypothesis the cusp counting decay may even fail in the complex hyperbolic setting: Stover~\cite{Stover} constructed towers of finite covers of complex ball quotients in which the number of cusps grows linearly with the covering degree.

Let us explain ideas of the proof of Theorem~\ref{mainthmA}. We first reduce the problem to a local theorem over a prime ideal where the real difficulties lie:

\begin{mainthm}[Theorem~\ref{thm:uniform-exponential-local-chi}]\label{mainthmB}
Let $\bK$ be a number field with degree \(n=[\bK:\mathbb Q]\). There exist constants $B,\alpha>0$ depending only on $n$ such that, for every prime ideal $\fp$ of $\cO_\bK$, for every integer $e\ge 1$, for every subgroup \(H\subset \SL_2(\cO_\bK/\fp^e)\) of exact level $e$, we have
\[
\frac{\big\lvert H\backslash\SL_2(\cO_\bK/\fp^e)/U_e \big\rvert}{[\SL_2(\cO_\bK/\fp^e): H]}
\le 
\lvert\cO_\bK/\fp \rvert^{\displaystyle B-\alpha e},
\]
where $U_e$ denotes the unipotent upper triangular subgroup of $\SL_2(\cO_\bK/\fp^e)$.
\end{mainthm}

After a first version of our paper is put online, the author discovered that a qualitative version of Theorem~\ref{mainthmB} with constants depending on $\bK$ is covered by a paper of Finis-Lapid~\cite{FinisLapid}. Finis-Lapid's results work for more general algebraic groups. We give a new explicit  rank-one proof of the orbital-integral decay phenomenon of Finis–Lapid; see Appendix \S\ref{sec:finis-lapid}.

The global argument expresses the weighted cusp contribution in
terms of double cosets and decomposes it prime by prime, thereby
reducing Theorem~A to the local estimate of Theorem~B. The cusps with multiplicities are expressed as double cosets, and the weighted cusp contribution is decomposed prime by prime. One is led to show that, in the finite quotient
\(\SL_2(\cO_\bK/\fp^e)\), a subgroup of exact level \(e\) cannot have too many orbits on
\(\SL_2(\cO_\bK/\fp^e)/U_e\) relative to its index. 


The technical bulk of the paper is Section~\ref{sec:local-estimates-II}. Though the underlying philosophy is similar to Finis--Lapid's work, our proof is different and more explicit. Let us explain the ideas of our proof. One fixes a nilpotent direction and a first coefficient, and then counts the possible coefficients. If too many slopes occur, then many pairs of slopes have controlled differences. The corresponding four-point configurations (with three consecutive differences, which we call parallelograms) force new infinitesimal directions at deeper congruence levels. This is the mechanism which prevents the fixed-point contribution from being too large. This use of differences and parallelograms is inspired by the energy argument in additive combinatorics, see \cite[\S~2.3]{TaoVuAdditiveCombinatorics}. The present argument, however, takes place inside the congruence filtration of a finite principal local ring, and the output is a propagation statement across levels rather than a growth statement inside a single abelian group. A useful way to summarize the proof is that all sources of asymptotic decay are eventually compressed into four numerical quantities. The final fixed-point estimate has, up to constants depending only on \([\bK:\bQ]\), the schematic form
\[
   p^{\jmath_1+\jmath_2+\widehat c+\lvert\Pi\rvert}.
\]
The notations are introduced later; in this overview it
is enough to regard these four terms as measuring four possible obstructions to
having many cusp orbits. Two of them measure how late the relevant nilpotent
directions appear, one measures a transverse codimension, and one counts levels at which a nilpotent direction is absent. The main point is that exact level \(e\) forces at least one of these four quantities to be comparable to \(e\). This alternative is what turns the local analysis into a uniform exponential decay in the level.

We now give a straightforward application of Theorem~\ref{mainthmA}. Hirzebruch, Van den Ven and Zagier~\cite{HirzebruchVandeVen,HirzebruchZagier} classified birational types of Hilbert modular surfaces, and raised a conjecture concerning their minimal models (see \cite[Conjecture~VII.4.4]{VDG}). The following theorem is due to van der Geer~\cite{VDG79} in the case of principal congruence subgroups. We denote by $Y_\Gamma$ the smooth projective surface obtained by Hirzebruch's resolution of $\bH\times \bH/\Gamma$. 
\begin{mainthm}\label{mainthm:large-level-normal-minimal}
Let \(\bK\) be a real quadratic field and \(\Omega\) be a generalized Hilbert modular group over \(\bK\). There exists a constant
\(N_0=N_0(\bK)\) with the following property. If \(\Gamma\triangleleft\Omega\) is a normal congruence subgroup of congruence level \(\fn\) with \(N(\fn)>N_0\) such that the smooth projective surface \(Y_\Gamma\) is not rational, then it is minimal, i.e.\ has no $(-1)$-curves.
\end{mainthm}
The condition that \(Y_\Gamma\) is not rational is automatically satisfied for all but finitely many real quadratic $\bK$, see \cite{HirzebruchVandeVen,HirzebruchZagier} and \cite[\S~VII.3]{VDG}. Note that normal congruence subgroups do not outnumber principal congruence subgroups by much, see \cite{Vaserstein, Abe}. Theorem~\ref{mainthm:large-level-normal-minimal} is not true if we drop the normality condition; this will be studied in a subsequent paper. 

We refer to \cite{BGLS,ABBGNRS,FL18,Matz} an the references therein for other problems about asymptotic growth of invariants of arithmetic groups when the level grows.

\addtocontents{toc}{\protect\setcounter{tocdepth}{-1}}
\subsubsection*{Plan}
\addtocontents{toc}{\protect\setcounter{tocdepth}{2}}
Sections~\ref{sec:local-estimates-I} and \ref{sec:local-estimates-II} are local over a fixed prime while Section~\ref{sec:global-field-section} is over a global field. 

Section~\ref{sec:local-estimates-I} sets up the basic objects, notations and computations for Section~\ref{sec:local-estimates-II}. Section~\ref{sec:local-estimates-II} is technically the hardest part of the paper. Once the framework based on Burnside's lemma is settled, the key steps of our method are the four subsections \ref{sec:parallelogram}, \ref{sec:semisimple-obstruction}, \ref{sec:double-counting} and \ref{sec:semisimple-bounds}. Theorem~\ref{mainthmB} is proved at the end of Section~\ref{sec:local-estimates-II}. Note that the case of residue characteristic two requires often a separate, more complicated treatment; the computations in \S\ref{sec:commutator-constraints-two} illustrate the kind of troubles. 

In Section~\ref{sec:global-field-section} we define and explain the central objects of the paper. After the setup, the rest of Section~\ref{sec:global-field-section} is mostly a chain of formal arguments connecting Theorem~\ref{mainthmA} to Theorem~\ref{mainthmB}. 

Theorem~\ref{mainthm:large-level-normal-minimal} is a quick consequence of Theorem~\ref{mainthmA}, it is proved in the appendice.

\addtocontents{toc}{\protect\setcounter{tocdepth}{-1}}
\subsubsection*{Notations}
\addtocontents{toc}{\protect\setcounter{tocdepth}{2}}
Most notations used in the two local sections are not used in Section~\ref{sec:global-field-section}. In this paper we use both $\lvert \cdot \rvert$ and $\#$ to denote the cardinal of a finite set.

\begingroup
\small
\renewcommand{\arraystretch}{1.05}

\setlength{\tabcolsep}{3pt}

\begin{longtable}{@{}L{0.39\textwidth}L{0.08\textwidth}L{0.39\textwidth}L{0.08\textwidth}@{}}
\toprule
\textbf{Notation/Notion}   & \textbf{Ref.} &
\textbf{Notation/Notion}   & \textbf{Ref.} \\
\midrule
\endfirsthead

\endhead

\midrule
\multicolumn{4}{r@{}}{\emph{continued on the next page}}\\
\endfoot

\bottomrule
\endlastfoot

\(R_e=\cO_\bK/\fp^e\)
& \S\ref{sec:setup}
&
\(\bF_q=\cO_\bK/\fp\)
& \S\ref{sec:setup} \\

\(q=p^f\)
& \S\ref{sec:setup}
&
\(\varpi\) uniformizer
& \S\ref{sec:setup} \\

\(e_0=v_\fp(p)\)
& \S\ref{sec:setup}
&
\(p=u_0\varpi^{e_0}\)
& \S\ref{sec:setup} \\

\(G_e=\SL_2(R_e)\)
& \S\ref{sec:setup}
&
\(U_e\) upper unipotent group
& \S\ref{sec:setup} \\

\(X_e=G_e/U_e\)
& \S\ref{sec:setup}
&
\(\rho_{k,j}:G_k\to G_j\)
& \S\ref{sec:setup} \\

\(\chi_e(H)=|H\backslash X_e|/[G_e:H]\)
& \S\ref{sec:setup}
&
exact level \(e\)
& \S\ref{sec:setup} \\

\(|G_e|=q^{3e-2}(q^2-1)\)
& \S\ref{sec:setup}
&
primitive columns
& \S\ref{sec:primitive-column} \\

\(K_j=\ker(G_j\to G_{j-1})\)
& \S\ref{sec:filtration-lie-algebra}
&
\(H_j=\rho_{e,j}(H)\)
& \S\ref{sec:filtration-lie-algebra} \\

\(N_j=N_j(H)\)
& \S\ref{sec:filtration-lie-algebra}
&
\(W_j=W_j(H)\subset\mathfrak{sl}_2(\bF_q)\)
& \S\ref{sec:filtration-lie-algebra} \\

\(\psi_j:N_j\to W_j\)
& \S\ref{sec:filtration-lie-algebra}
&
\(d_j=\dim_{\bF_p}W_j\)
& \S\ref{sec:filtration-lie-algebra} \\

\(v_\fp\) valuation
& \S\ref{sec:filtration-lie-algebra}
&
\(M_2(R_e)\) matrix ring
& \S\ref{sec:filtration-lie-algebra} \\

\((x,y)=xyx^{-1}y^{-1}\)
& \S\ref{sec:commutator-constraints-one}
&
\([P,Q]=PQ-QP\)
& \S\ref{sec:commutator-constraints-one} \\

\(D=\mathrm{diag}(1,-1)\)
& \S\ref{sec:commutator-constraints-two}
&
\(E=e_{12}\), \(F=e_{21}\)
& \S\ref{sec:commutator-constraints-two} \\

\(\nu=v_\fp(2)\)
& \S\ref{sec:commutator-constraints-two}
&
\(\widehat{(x,y)}=((((x,y),x),x),x)\)
& \S\ref{sec:commutator-constraints-two} \\

\(\Lambda_j\) bad lines
& \S\ref{sec:bad-lines}
&
\(X_e^w\) fixed-point set
& \S\ref{sec:fixed-points-trace-two} \\

\(\Sigma_j=\varpi R_{e-j+1}\)
& \S\ref{sec:slope-decomposition}
&
\(N(\eta)\)
& \S\ref{sec:slope-decomposition} \\

\(M_j(\eta)\), \(m_j(\eta)\)
& \S\ref{sec:slope-decomposition}
&
\(S_j(\eta)\)
& \S\ref{sec:filtered-counting} \\

\(\pi_{j,r}:\Sigma_j\to\Sigma_{j+r}\)
& \S\ref{sec:filtered-counting}
&
\(\mathcal A_j(\ell)\), \(A_j(\ell)\)
& \S\ref{sec:filtered-counting} \\

\(\mathcal E_j(b)\)
& \S\ref{sec:filtered-counting}
&
\(T_j(\ell)\), \(T_j(\lambda)\)
& \S\ref{sec:conjugating-other-lines} \\

\(T_j=\sum_\lambda T_j(\lambda)\)
& \S\ref{sec:counting-formula}
&
exact formula for \(\chi_e(N_{j_0})\)
& \S\ref{sec:counting-formula} \\

\(E_\lambda,F_\mu,D_{\lambda,\mu}\)
& \S\ref{sec:choices-of-basis}
&
adapted \(\mathfrak{sl}_2\)-basis
& \S\ref{sec:choices-of-basis} \\

\(Q_{\eta',\eta}(b)\)
& \S\ref{prop:membership-transversality-new}
&
\(\mathcal P_{j,b}(\delta)\)
& \S\ref{sec:double-counting} \\

\(\jmath_1,\jmath_2,\jmath\)
& \S\ref{sec:double-counting}
&
\(\Pi\subset[\jmath_2,e]\)
& \S\ref{sec:double-counting} \\

\(m=e-j\)
& \S\ref{sec:double-counting}
&
\(\gamma=e+\nu+1-\jmath_2\)
& \S\ref{sec:double-counting} \\

\(c_*(j,s,d)\)
& \S\ref{sec:double-counting}
&
\(c^*(j,s,d)\)
& \S\ref{sec:double-counting} \\

\(\hat c(j,s)\)
& \S\ref{sec:double-counting}
&
\(V_k=\{a\in\bF_q:aD\in W_k\}\)
& \S\ref{sec:semisimple-bounds} \\

\(c_0(K)\)
& \S\ref{sec:semisimple-bounds}
&
\(\jmath_3,\jmath^*,J\)
& \S\ref{sec:semisimple-bounds} \\

\(c_0^*(K)\)
& \S\ref{sec:semisimple-bounds}
&
final \(A_j(\ell)\)-bounds
& \S\ref{sec:final-estimates-local} \\

\(G_\infty=\prod_{v\mid\infty}\PGL_2(\bK_v)\)
& \S\ref{sec:lattices-arithmetic-groups}
&
\(\iota_\infty\), diagonal embedding
& \S\ref{sec:lattices-arithmetic-groups} \\

\(\Gamma_{\bK}^{\mathrm{rat}}=\PSL_2(\cO_\bK)\)
& \S\ref{sec:lattices-arithmetic-groups}
&
\(\Delta_\bK=\iota_\infty(\Gamma_{\bK}^{\mathrm{rat}})\)
& \S\ref{sec:lattices-arithmetic-groups} \\

\(\cC_{\bK}^{\mathrm{rat}}\), rational commensurability class
& \S\ref{sec:lattices-arithmetic-groups}
&
\(\cC_{\bK}^{\infty}\), archimedean commensurability class
& \S\ref{sec:lattices-arithmetic-groups} \\

\(\cE\subset M_2(\bK)\), Eichler order
& \S\ref{sec:global-arithmetic-setting}
&
\(\cM\supset\cE\), maximal order
& \S\ref{sec:global-arithmetic-setting} \\

\(\cG=\SL(\cM)\)
& \S\ref{sec:global-arithmetic-setting}
&
\(\P\cG\), projective image of \(\cG\)
& \S\ref{sec:global-arithmetic-setting} \\

\(\Lambda\subset\PN(\cE)\)
& \S\ref{sec:global-arithmetic-setting}
&
\(\Lambda_\cM=\Lambda\cap\P\cG\)
& \S\ref{sec:global-arithmetic-setting} \\

\(\Gamma\), inverse image of \(\Lambda_\cM\) in \(\cG\)
& \S\ref{sec:global-arithmetic-setting}
&
\(\Proj g\), projective class of \(g\)
& \S\ref{sec:global-arithmetic-setting} \\

\(\cG(\fn)\), principal congruence
& \S\ref{sec:global-arithmetic-setting}
&
\(\fn\), congruence level
& \S\ref{sec:global-arithmetic-setting} \\

\(L\) with \(\cM=\operatorname{End}_{\cO_\bK}(L)\)
& \S\ref{sec:global-arithmetic-setting}
&
\(h_\bK\), class number
& \S\ref{sec:cusp-multiplicities-and-relative-multiplier-indices} \\

\(\omega(\fa)\), number of divisors
& \S\ref{sec:quantitative-global-decay}
&
\(N(\fa)=|\cO_\bK/\fa|\), norm
& \S\ref{sec:quantitative-global-decay} \\

\(\bK_+^*\), totally positive elements
& \S\ref{sec:cusps-hirzebruch-multiplicity}
&
\(U_M^+=\{\varepsilon\in\bK_+^*: \varepsilon M=M\}\)
& \S\ref{sec:cusps-hirzebruch-multiplicity} \\

\(B(\bK)\), upper triangular
& \S\ref{sec:cusps-hirzebruch-multiplicity}
&
\(u(x)=\begin{psmallmatrix}1&x\\0&1\end{psmallmatrix}\)
& \S\ref{sec:cusps-hirzebruch-multiplicity} \\

\(b(\Gamma,g)\), quasi-amplitude
& \S\ref{sec:cusps-hirzebruch-multiplicity}
&
\(w_\sigma=[b(\cG,g):b(\Gamma,g)]\)
& \S\ref{sec:cusps-hirzebruch-multiplicity} \\

\(V(\Gamma,g)\), multiplier group
& \S\ref{sec:cusps-hirzebruch-multiplicity}
&
\((M_\sigma,V_\sigma)\), cusp type
& \S\ref{sec:cusps-hirzebruch-multiplicity} \\

\(a_\sigma(\Gamma)=[U_{M_\sigma(\Gamma)}^+:V_\sigma(\Gamma)]\)
& \S\ref{sec:cusps-hirzebruch-multiplicity}
&
\(\cC_\tau(\Gamma)=\Gamma\backslash\cG\cdot\tau\)
& \S\ref{sec:stabilizer-index-decomposition} \\

\(g_\tau\), \(g_\sigma=r g_\tau\), cusp charts
& \S\ref{sec:stabilizer-index-decomposition}
&
\(P_\cH(g_\sigma)\)
& \S\ref{sec:stabilizer-index-decomposition} \\

\(w_{\sigma|\tau}\), \(u_{\sigma|\tau}\)
& \S\ref{sec:stabilizer-index-decomposition}
&
\(\cT\), cusp representatives
& \S\ref{sec:cusp-multiplicities-and-relative-multiplier-indices} \\

\(C_{\rm UV}\)
& \S\ref{sec:cusp-multiplicities-and-relative-multiplier-indices}
&
\(\overline\cG=\cG/\cG(\fn)\), \(\overline\Gamma=\Gamma/\cG(\fn)\)
& \S\ref{sec:finite-quotient-weighted-width} \\

\(P_\tau=\Stab_\cG(\tau)\)
& \S\ref{sec:finite-quotient-weighted-width}
&
\(U_\tau=g_\tau\{u(x)\mid x\in\bK\}g_\tau^{-1}\cap\cG\)
& \S\ref{sec:finite-quotient-weighted-width} \\

\(S(r,\tau)=P_\tau\cap r^{-1}\Gamma r\)
& \S\ref{sec:finite-quotient-weighted-width}
&
\(\overline P_\tau,\overline U_\tau,\overline S(r,\tau)\)
& \S\ref{sec:finite-quotient-weighted-width} \\

\(\overline X=\overline\Gamma\backslash\overline\cG\)
& \S\ref{sec:finite-quotient-weighted-width}
&
\(\chi_{Q,U}(H)=|H\backslash Q/U|/[Q:H]\)
& \S\ref{sec:product-bounds-localization} \\

\(G_e(\fp)=\SL(L/\fp^eL)\)
& \S\ref{sec:product-bounds-localization}
&
\(\fb_\tau=b(\cG,g_\tau)\)
& \S\ref{sec:product-bounds-localization} \\

\(\Gamma_{e_\fp}(\fp)\)
& \S\ref{sec:product-bounds-localization}
&
\(U_{\tau,e_\fp}(\fp)\)
& \S\ref{sec:product-bounds-localization} \\

\(C_{\mathrm{AL}}(\bK)\),
& \S\ref{sec:lambda-gamma-cusp-multiplicity-comparison}
&
\(d_\Lambda=[\Lambda:\Lambda_\cM]\)
& \S\ref{sec:lambda-gamma-cusp-multiplicity-comparison} \\

\(S_\cE=\{\fp:\cE_\fp\ne\cM_\fp\}\)
& \S\ref{sec:lambda-gamma-cusp-multiplicity-comparison}
&
\(\mathcal R(\Lambda)\)
& \S\ref{sec:quantitative-global-decay} \\

\end{longtable}

\endgroup

\section{Local asymptotic estimates I: preparation}\label{sec:local-estimates-I}

\subsection{Setup and introduction to the local problem}\label{sec:setup}

Let $\bK$ be a number field, $\fp$ a prime ideal of $\cO_\bK$
lying above a prime number $p\in \bN$, with residue field $\bF_q$ where $q=p^f$, $f\in \bN^*$, ramification index $e_0$, and uniformiser $\varpi\in\cO_\bK$. 
For $e\in \bN^*$ set
\[
R_e := \cO_\bK/\fp^e, \quad
G_e := \SL_2(R_e), \quad
U_e :=
\left\{
\begin{pmatrix}1&x\\0&1\end{pmatrix} \mid x \in R_e
\right\}.
\]
Write \(p=u_0\varpi^{e_0}\) for some unit $u_0\in R_e^*$.

Define the coset space $X_e := G_e/U_e$ and, for a subgroup $H \subset G_e$, set
\[
\chi_e(H) := \frac{|H\backslash X_e|}{[G_e:H]}.
\]

Write $\rho_{k,j} : G_k \to G_j$ for the reduction map, for \(j\le k\). We say that a subgroup $H\subset G_e$ has \emph{exact level $e$} if 
\[
H\neq \rho_{e,e-1}^{-1}\big(\rho_{e,e-1}(H)\big),\quad \text{when} \; e>1,
\]
and if $H\neq G_1$ when $e=1$.

The proof of the following three theorems will occupy Sections~\ref{sec:local-estimates-I} and \ref{sec:local-estimates-II}. 

\begin{thm}\label{thm:main_local_chi_B}
Assume that \(q\ge 59\). For every \(e\ge 1\) and every subgroup
\(H\subset G_e\) such that \(\rho_{e,1}(H)\neq \SL_2(\bF_q)\), one has
\[
\chi_e(H)<\frac{2}{q-1}.
\]
\end{thm}
Note that Theorem~\ref{thm:main_local_chi_B} works also for a subgroup not of exact level. Moreover the condition on $H$ is often automatically satisfied:
\begin{rem}\label{rem:the-H1G1-condition-unramified}
Assume that \(\fp\) is unramified and that \(p\ge 5\).  If \(H\) has exact
level \(e\), then \(H_1\neq \SL_2(\bF_q)\). This remark will be proved in \S\ref{sec:proof-of-main-local-thm-B}.
\end{rem}

\begin{thm}\label{thm:main_local_chi}
Let \(H\subset G_e\) be a subgroup of exact level \(e\).

If \(p\ge 3\) and \(e\ge 2e_0^2+3e_0+4\), then
\[
\chi_e(H)
\le
q(q^2-1)
\left(
q^{-e}+2q^{1-e}+q^2 e^{3/2}p^{-(e-3)/(6+4e_0)}
\right).
\]

If \(p=2\) and \(e>36e_0+13\), then
\[
\chi_e(H)
\le
q^{6e_0+1}(q^2-1)
\left(
q^{-e}+2q^{1-e}
+4q^{2+e_0/2} e^{3/2}p^{-(e-37e_0)/(18+34e_0)}
\right).
\]
\end{thm}

In practice, Theorem~\ref{thm:main_local_chi} is used in the following simpler form:
\begin{thm}\label{thm:uniform-exponential-local-chi}
Let \(n=[\bK:\mathbb Q]\).  For every \(e\ge1\), and
every subgroup \(H\subset G_e\) of exact level \(e\), the following estimates
hold.

If \(p\ge3\), then
\[
  \chi_e(H)\le q^{B_{1}-\alpha_1 e},
  \quad
  \alpha_{1}=\frac{1}{2n(6+4n)},\;
  B_{1}=7n+4.
\]

If \(p=2\), then
\[
  \chi_e(H)\le q^{B_2-\alpha_2e},
  \;
  \alpha_2=\frac{1}{2n(18+34n)},\,
  B_2=\frac{13}{2}n^2+\frac{27}{2}n+\frac{17}{2}.
\]
\end{thm}

\subsection{Preparatory material}
Most of the time we will work with \(H\) an arbitrary subgroup of $G_e$ not necessarily of exact level $e$.

\subsubsection{}\label{sec:primitive-column}
Recall that $X_e = G_e/U_e$ is the coset space. There is a concrete description:
\begin{prop}\label{prop:primitive-cols}
The map $gU_e \mapsto g\begin{pmatrix}1 \\ 0\end{pmatrix}$ is a
$G_e$-equivariant bijection from $X_e=G_e/U_e$ to the set of primitive column vectors
\[
\Bigl\{\begin{pmatrix}a \\ b\end{pmatrix}\in R_e^2
\mid \begin{pmatrix}a \\ b\end{pmatrix}\notin \fp R_e^2\Bigr\}.
\]
\end{prop}

\begin{proof}
The stabilizer of $\begin{pmatrix}1 \\ 0\end{pmatrix}$ is exactly $U_e$. Therefore the map is well defined and the injectivity follows. If $\begin{pmatrix}a \\ b\end{pmatrix}\notin \fp R_e^2$ then $aR_e+bR_e=R_e$, so the column can be completed to a matrix with determinant $1$, thus surjectivity. 
\end{proof}

\subsubsection{}\label{sec:filtration-lie-algebra}
For $j \ge 2$, let
\[
K_j := \ker(G_j \to G_{j-1}).
\]
Every element of $K_j$ has the form
$I + \varpi^{j-1}A$ where
$A\in M_2(\bF_q)$ is a matrix with coefficients in the finite field $\bF_q$. The condition $\det (I + \varpi^{j-1}A)=1$ is equivalent to $\tr(A)=0$, so $K_j$ is an abelian group of order $q^3$ isomorphic to
$(\mathfrak{sl}_2(\bF_q),+)$ as an $\bF_{p}$-vector space of $\bF_{p}$-dimension $3f$, where $q=p^f$. For any $M\in M_2(R_e)$, we denote $v_\fp(M)$ the minimal valuation of the coefficients of $M$.

\begin{lem}\label{lem:order-Ge}
For every \(e\ge 1\), the cardinal of $G_e=\SL_2(R_e)$ is
\[
|G_e|=q^{3e-2}(q^2-1).
\]
\end{lem}

\begin{proof}
For \(e=1\), this is the standard formula
\[
|\SL_2(\bF_q)|=q(q^2-1).
\]
For \(j\ge2\), the kernel of the reduction map $\SL_2(R_j)\to \SL_2(R_{j-1})$ is in bijection with $\mathfrak{sl}_2(\bF_q)$ and has cardinal \(q^3\). The conclusion is then obtained by induction.
\end{proof}

Let \(H\subset G_e\) be a subgroup. For \(1\le j\le e\), define
\[
H_j := \rho_{e,j}(H).
\]
For \(2\le j\le e\), define
\[
N_j=N_j(H) := \ker(\rho_{e,j-1}|_{H}).
\]
Then \(\rho_{e,j}(N_j)=H_j\cap K_j\). If $x$ belongs to the complement $N_j\setminus N_{j+1}$, then $x=I+M$ with $v_\fp(M)=j-1$.

Define the \emph{Lie image at level $j$} of $H$ as
\[
W_j=W_j(H) := \{A\in\mathfrak{sl}_2(\bF_q)\mid
I+\varpi^{j-1}A\in\rho_{e,j}(N_j)\}.
\]

For $j\ge 2$ the map 
\[
\psi_j \colon
\left\{
\begin{array}{rcl}
N_j & \longrightarrow & W_j\subset \mathfrak{sl}_2(\bF_q)\\
n & \longmapsto & A
\end{array}
\right.
\quad\text{with}\;
n \equiv I+\varpi^{j-1}A \pmod{\varpi^j}.
\]
is a group homomorphism. In particular, for $j\ge 2$ the group $W_j=\mathrm{Im}(\psi_j)$ is an
$\bF_{p}$-subspace of $\mathfrak{sl}_2(\bF_q)$. In general \(W_j\) need not be an \(\bF_q\)-subspace; this is one of the major sources of technical difficulties.

We use the convention $N_j:=K_j, W_j:=\mathfrak{sl}_2(\bF_q)$ for all $j\ge e+1$.

Define, for $2\le j\le e$, 
\[
d_j\;:=\;\dim_{\bF_{p}} W_j,
\]
so that $|W_j|=p^{d_j}$ and
$d_j\in\{0,1,\ldots,3f\}$.

\subsubsection{}\label{sec:commutator-constraints-one}

We explain now how commutators behave in \(G_e\) with respect to the filtration $\{ N_j\}$.  

\begin{lem}\label{lem:trace-val}
Let $M\in M_2(R_e)$ with $\det(I+M)=1$. Then 
\[
\tr(M) = M_{12}M_{21}-M_{11}M_{22},
\]
and in particular
\[
v_\fp(\tr(M))\ge\min\bigl(v_\fp(M_{12})+v_\fp(M_{21}),\; v_\fp(M_{11})+v_\fp(M_{22})\bigr).
\]
\end{lem}

\begin{proof}
Expanding 
\[
\det(I+M) = (1+M_{11})(1+M_{22})-M_{12}M_{21} = 1+\tr(M)+M_{11}M_{22}-M_{12}M_{21}
\]
gives $\tr(M) = M_{12}M_{21}-M_{11}M_{22}$.
\end{proof}

\begin{lem}\label{lem:lie-bracket-ring}
Let $x=I+P\in N_j$ and $y=I+Q\in N_k$ with $j, k\ge 2$. Then the commutator $(x,y)=xyx^{-1}y^{-1}$ satisfies 
\begin{equation}\label{eq:comm-exp-formula}
(x,y) - I =[P,Q] + \tr(P)\tr(Q)I + \Theta,
\end{equation}
where $\Theta\in M_2(R_e)$ satisfies
\begin{equation}\label{eq:Theta-bound}
v_\fp(\Theta)\ge j+k+\min(j,k)-3.
\end{equation}
In particular, if $j+k-1\le e$, then $(x,y)\in N_{j+k-1}$ and
\begin{equation}\label{eq:comm-exp-cong}
(x,y)\equiv I+[P,Q]\pmod{\varpi^{j+k-1}}.
\end{equation}
\end{lem}

\begin{proof}
Since $\det(x)=1$, the Cayley--Hamilton theorem over a commutative ring gives
\[
x^{-1}=I-P+\tr(P)I,
\]
and similarly for $y$. Thus
\[
(x,y)=(I+P)(I+Q)(I-P+\tr(P)I)(I-Q+\tr(Q)I).
\]
Set \(X_1=P, X_2=Q,X_3=-P+\tr(P)I,X_4=-Q+\tr(Q)I\).
Expanding the product gives
\[
(x,y)-I=\sum_{\emptyset\ne S\subseteq\{1,2,3,4\}}X_S,
\]
where \(X_S=X_{s_1}\cdots X_{s_l}\) for
\(S=\{s_1<\cdots<s_l\}\).

The inverse relation \((I+X_1)(I+X_3)=I\) gives
\[
\sum_{\emptyset\ne S\subseteq\{1,3\}}X_S=0,\quad \text{and similarly} \;\sum_{\emptyset\ne S\subseteq\{2,4\}}X_S=0.
\]
The four mixed two-element terms are
\begin{align*}
& X_1X_2 = PQ,\quad X_1X_4 = -PQ+\tr(Q)P,\quad X_2X_3 = -QP+\tr(P)Q,\\
& X_3X_4 = PQ-\tr(Q)P-\tr(P)Q+\tr(P)\tr(Q)I.
\end{align*}
Their sum is \([P,Q]+\tr(P)\tr(Q)I\). Therefore \eqref{eq:comm-exp-formula} holds with
\[
\Theta:=\sum_{|S|\ge 3}X_S.
\]

It remains to estimate \(\Theta\). By Lemma~\ref{lem:trace-val},
\[
v_\fp(X_1),v_\fp(X_3)\ge j-1,
\qquad
v_\fp(X_2),v_\fp(X_4)\ge k-1.
\]
Every subset \(S\) with \(|S|\ge3\) meets both \(\{1,3\}\) and \(\{2,4\}\). Hence, if \(a=|S\cap\{1,3\}|\) and \(b=|S\cap\{2,4\}|\), then \(a,b\ge1\) and \(a+b\ge3\), so
\[
v_\fp(X_S)\ge a(j-1)+b(k-1)
   \ge j+k+\min(j,k)-3.
\]
This proves \eqref{eq:Theta-bound}. Finally,
\[
v_\fp([P,Q])\ge j+k-2,\quad
v_\fp(\tr(P)\tr(Q)I)\ge 2j+2k-4\ge j+k-1,
\]
and \eqref{eq:Theta-bound} give
\(v_\fp(\Theta)\ge j+k-1\). Therefore, if \(j+k-1\le e\), then
\[
(x,y)\equiv I+[P,Q]\pmod{\varpi^{j+k-1}},
\]
and in particular \((x,y)\in N_{j+k-1}\).
\end{proof}

\begin{lem}\label{lem:lie-bracket}
For all $j, k\ge 2$ with $j+k-1\le e$\textup{:}
\[
[W_j,\,W_k]\subset W_{j+k-1}.
\]
\end{lem}
\begin{proof}
This is a direct corollary of Lemma~\ref{lem:lie-bracket-ring}.
\end{proof}

\subsection{Commutators when \texorpdfstring{$p=2$}{p=2}}\label{sec:commutator-constraints-two}

Consider the standard basis:
\[
D=\begin{pmatrix}1&0\\0&-1\end{pmatrix},\; E=\begin{pmatrix}0&1\\0&0\end{pmatrix},\;F=\begin{pmatrix}0&0\\1&0\end{pmatrix}\in \mathfrak{sl}_2(\bF_q).
\]
We have \([D,E]=2E,[D,F]=-2F,[E,F]=D\). 

Note that when $p=2$, we have $D=I$ and \([D,E]=[D,F]=0\), so Lemma~\ref{lem:lie-bracket} does not give much information. We will need finer information in caracteristic $2$. Recall that when $p=2$ the ramification index satisfies $e_0=\nu=v_\fp(2)$.

\begin{lem}\label{lem:two-word-one}
Assume $p=2$. Let $a,b$ be integers with
\[
a\ge \nu+2,\quad b\ge \nu+2, \quad 2a+b+\nu-2\le e.
\]
Let $x\in N_a$ and $y\in N_b$ satisfy
\[
\psi_a(x)=\alpha D+\mu E,\qquad
\psi_b(y)=\beta F,
\]
with $\alpha\in\bF_q$ and $\mu,\beta\in\bF_q^*$. Denote by
$\bar u_0$ the image in $\bF_q$ of the unit $u_0$.
We have $\big((x,y),x\big)\in N_{2a+b+\nu-2}$ and
\[
\psi_{2a+b+\nu-2}\Big(\big((x,y),x\big)\Big) = \bar u_0 \mu \beta(\alpha D + \mu E).
\]
\end{lem}
\begin{proof}
We denote $z=(x,y)$. Then by Lemma~\ref{lem:lie-bracket-ring} $z\in N_{a+b-1}$ and
\[
\psi_{a+b-1}(z)=\mu\beta D, \qquad \mu\beta\in \bF_q^*.
\]
We write $x=I+P,y=I+Q,z=I+R$. Then by Lemma~\ref{lem:lie-bracket-ring}
\[
R=PQ-QP+\tr(P)\tr(Q)I+\Theta.
\]
We have
\begin{align*}
R_{11}-R_{22} &  = (PQ-QP)_{11}-(PQ-QP)_{22} + (\Theta)_{11}-(\Theta)_{22}\\
& = 2(P_{12}Q_{21}-Q_{12}P_{21})+(\Theta)_{11}-(\Theta)_{22}.
\end{align*}
Recall that $\nu=e_0=v_\fp(2)$ and $2=u_0\varpi^\nu$. Note that the hypothesis on $\psi_a(x),\psi(y)$ implies that
\[
v_\fp(2P_{12}Q_{21}) = \nu+a+b-2,\; v_\fp(2Q_{12}P_{21})\ge \nu+a+b.
\]
The term $\Theta$ satisfies, by Lemma~\ref{lem:lie-bracket-ring},
\[
v_\fp((\Theta)_{11}-(\Theta)_{22})
\ge v_\fp(\Theta)
\ge a+b+\min(a,b)-3
\ge \nu+a+b-1,
\]
because \(\min(a,b)\ge\nu+2\).
Therefore we can write
\begin{align}\label{eq:two-commutator-formula-11}
R_{11}-R_{22} &  = 2P_{12}Q_{21}+S, \quad v_\fp(S)\ge a+b+\nu-1 \\
R_{11}-R_{22} & \equiv \bar u_0\mu\beta \varpi^{a+b+\nu-2} \pmod {\varpi^{a+b+\nu-1}}. \notag
\end{align}
We have \(R_{21}  = (PQ-QP)_{21}+ \Theta_{21}\), so
\begin{align*}
R_{21}  = Q_{21}(P_{22}-P_{11})+P_{21}(Q_{11}-Q_{22})+\Theta_{21}.
\end{align*}
Combining $2P_{22}-\tr(P)=P_{22}-P_{11}$ with Lemma~\ref{lem:trace-val} gives
\begin{equation}\label{eq:two-commutator-formula-12}
P_{22}-P_{11}
\equiv
\bar u_0\alpha\,\varpi^{a+\nu-1}
\pmod{\varpi^{a+\nu}}.
\end{equation}
Combining \eqref{eq:two-commutator-formula-12} with the hypotheses on
\(\psi_a(x)\) and \(\psi_b(y)\) gives
\begin{align}\label{eq:two-commutator-formula-13}
R_{21} & \equiv Q_{21}(P_{22}-P_{11}) \equiv 2Q_{21}P_{22} \pmod {\varpi^{a+b+\nu-1}},
\notag \\
R_{21} & \equiv \bar u_0\alpha\beta \varpi^{a+b+\nu-2} \pmod {\varpi^{a+b+\nu-1}}. 
\end{align}
Similarly we obtain
\begin{equation}\label{eq:two-commutator-formula-14}
v_\fp(R_{12})\ge a+b+\nu-1.
\end{equation}

Write $w=(z,x)=I+T$. Using the same commutator expansion as in
Lemma~\ref{lem:lie-bracket-ring}, we have
\begin{equation}\label{eq:two-commutator-formula-15}
T=RP - PR + \tr(R)\tr(P)I + \Omega,\qquad
v_\fp(\Omega)\ge 3a+b+\nu-4.
\end{equation}
Indeed, put \(r:=a+b-2\).  By
\eqref{eq:two-commutator-formula-11},
\eqref{eq:two-commutator-formula-13}, and
\eqref{eq:two-commutator-formula-14}, one may write
\begin{align*}
& R=sI+R',\quad -R+\tr(R)I=sI+R'', \qquad
\text{with} \\
& v_\fp(s)\ge r,\qquad v_\fp(R'),v_\fp(R'')\ge r+\nu.
\end{align*}
The degree at least three terms with two \(R\)-side factors have valuation at least \(2r+a-1\ge 3a+b+\nu-4\).  The remaining degree three terms are
\[
RPP^*+P(-R+\tr(R)I)P^*,
\qquad P^*:=-P+\tr(P)I.
\]
Using \(PP^*=-\tr(P)I\), their sum is
\[
-2s\tr(P)I+R'PP^*+PR''P^*,
\]
which has valuation at least
\[
\nu+r+2a-2=3a+b+\nu-4.
\]
This proves the asserted bound for \(\Omega\).

By Lemma~\ref{lem:trace-val}
\begin{align}
& v_\fp\big(\tr(P)\big)\ge 2a-2,\; v_\fp\big(\tr(R)\big)\ge 2a+2b-4, \notag \\
& v_\fp\big(\tr(P)\tr(R)I\big)\ge 4a+2b-6. \label{eq:two-commutator-formula-16}
\end{align}

Compute
\begin{equation}\label{eq:two-commutator-formula-17}
T_{12} = P_{12}(R_{11}-R_{22}) + R_{12}(P_{22}-P_{11}) + \Omega_{12}.    
\end{equation}
By \eqref{eq:two-commutator-formula-11} we have 
\begin{align}
 & v_\fp\big(P_{12}(R_{11}-R_{22})\big)= a-1+a+b+\nu-2=2a+b+\nu-3 \notag \\
& P_{12}(R_{11}-R_{22})\equiv u_0\mu^2\beta \varpi^{2a+b+\nu-3} \pmod \varpi^{2a+b+\nu-2}\label{eq:two-commutator-formula-18}
\end{align}
By \eqref{eq:two-commutator-formula-12} and \eqref{eq:two-commutator-formula-14} we get
\begin{equation}\label{eq:two-commutator-formula-19}
v_\fp\big(R_{12}(P_{22}-P_{11})\big)\ge a-1+\nu+a+b+\nu-1>2a+b+\nu-3.
\end{equation}
Combining \eqref{eq:two-commutator-formula-15}, \eqref{eq:two-commutator-formula-18}, \eqref{eq:two-commutator-formula-19} gives
\begin{equation}\label{eq:two-commutator-formula-110}
T_{12} \equiv u_0\mu^2\beta \varpi^{2a+b+\nu-3} \pmod \varpi^{2a+b+\nu-2}.
\end{equation}

Now compute
\begin{equation}\label{eq:two-commutator-formula-111}
T_{21} = R_{21}(P_{11}-P_{22}) + P_{21}(R_{22}-R_{11}) + \Omega_{21}.
\end{equation}
Combining \eqref{eq:two-commutator-formula-13}, \eqref{eq:two-commutator-formula-12}, \eqref{eq:two-commutator-formula-11} gives
\begin{align}\label{eq:two-commutator-formula-112}
 & v_\fp\big(R_{21}(P_{11}-P_{22})\big)\ge a+b+\nu-2+a-1+\nu\ge 2a+b+\nu-2,\\
 & v_\fp\big(P_{21}(R_{22}-R_{11})\big)\ge a+a+b+\nu-2\ge 2a+b+\nu-2.
\end{align}
Since $v_\fp(\Omega_{21})\ge 2a+b+\nu-2$, we get
\begin{equation}\label{eq:two-commutator-formula-113}
v_\fp(T_{21})\ge 2a+b+\nu-2.
\end{equation}

Now compute
\begin{equation}\label{eq:two-commutator-formula-114}
T_{11} = R_{12}P_{21} - P_{12}R_{21} + \tr(R)\tr(P) + \Omega_{11}.
\end{equation}
By \eqref{eq:two-commutator-formula-14} we have
\begin{equation}\label{eq:two-commutator-formula-115}
v_\fp(R_{12}P_{21})\ge a+b+\nu-1+a>2a+b+\nu-2,
\end{equation}
while by \eqref{eq:two-commutator-formula-13} we have
\begin{equation}\label{eq:two-commutator-formula-116}
P_{12}R_{21}  \equiv \bar u_0\alpha\beta\mu \varpi^{2a+b+\nu-3} \pmod {\varpi^{2a+b+\nu-2}}.
\end{equation}
Combining \eqref{eq:two-commutator-formula-114}, \eqref{eq:two-commutator-formula-115}, \eqref{eq:two-commutator-formula-116} with \eqref{eq:two-commutator-formula-16} gives
\begin{equation}\label{eq:two-commutator-formula-117}
T_{11}\equiv \bar u_0\alpha\beta\mu \varpi^{2a+b+\nu-3} \pmod {\varpi^{2a+b+\nu-2}}.
\end{equation}
Since \(T_{22} = R_{21}P_{12}  -P_{21}R_{12}+ \tr(R)\tr(P) + \Omega_{11}\), same arguments give
\begin{equation}\label{eq:two-commutator-formula-118}
T_{22}\equiv \bar u_0\alpha\beta\mu \varpi^{2a+b+\nu-3} \pmod {\varpi^{2a+b+\nu-2}}.
\end{equation}
Formulas \eqref{eq:two-commutator-formula-110}, \eqref{eq:two-commutator-formula-113}, \eqref{eq:two-commutator-formula-117}, and \eqref{eq:two-commutator-formula-118} provide the required information on all four entries of the matrix \(T\). This yields the conclusion.
\end{proof}

\begin{lem}\label{lem:two-word-two}
Assume the hypotheses of Lemma~\ref{lem:two-word-one}, and assume moreover that
$\alpha\neq 0, 3a+b+2\nu-3\le e$.
Set \(u:=(((x,y),x),x)\). Then
\[
u\in N_{3a+b+2\nu-3},
\qquad
\psi_{3a+b+2\nu-3}(u)
=
\bar u_0^{\,2}\alpha^2\mu\beta\,D.
\]
\end{lem}

\begin{proof}
Set
\(
z:=(x,y), w:=(z,x), u:=(w,x),
\)
and write
\(
x=I+P, z=I+R, w=I+T, u=I+S.
\)
Also denote
\[
m:=2a+b+\nu-2,\qquad c:=3a+b+2\nu-3.
\]
The hypothesis $\psi_a(x)=\alpha D+\mu E$ gives
\[
P_{12}\equiv \mu\,\varpi^{a-1}\pmod{\varpi^a},\qquad v_\fp(P_{21})\ge a,
\]
and, combined with \eqref{eq:two-commutator-formula-12},
\begin{equation}\label{eq:two-word-two-6}
P_{22}-P_{11}\equiv \bar u_0\alpha\,\varpi^{a+\nu-1}\pmod{\varpi^{a+\nu}}.
\end{equation}

We shall freely use several formulas from the proof of Lemma~\ref{lem:two-word-one}. 
We first sharpen the information on $T_{11}-T_{22}$.
We have
\[
T_{11}-T_{22}
=
2(R_{12}P_{21}-P_{12}R_{21})+(\Omega_{11}-\Omega_{22}).
\]
By \eqref{eq:two-commutator-formula-15}, we have
\[
v_\fp(\Omega_{11}-\Omega_{22})\ge 3a+b+\nu-4.
\]
Now, by \eqref{eq:two-commutator-formula-11} and by the inequality $v_\fp(P_{21})\ge a$,
\[
v_\fp\bigl(2R_{12}P_{21}\bigr)\ge \nu+(a+b+\nu-1)+a
=2a+b+2\nu-1>m+\nu-1,
\]
while by \eqref{eq:two-commutator-formula-117}, \(2P_{12}R_{21}\equiv
\bar u_0^2\alpha\mu\beta\,\varpi^{m+\nu-1}
\pmod{\varpi^{m+\nu}}\). Since
\[
3a+b+\nu-4\ge 2a+b+2\nu-2=m+\nu
\]
because $a\ge \nu+2$, we obtain
\begin{equation}\label{eq:two-word-two-8}
T_{11}-T_{22}
\equiv
-\bar u_0^{\,2}\alpha\mu\beta\,\varpi^{m+\nu-1}
\pmod{\varpi^{m+\nu}}.
\end{equation}

Now applying Lemma~\ref{lem:lie-bracket-ring} to $(w,x)$ we get
\begin{align}
S & =TP-PT+\tr(T)\tr(P)I+\Xi,
\notag\\
v_\fp(\Xi) & \ge m+2a-3=4a+b+\nu-5\ge c.\label{eq:two-commutator-formula-20}
\end{align}
Note that $\Xi$ will be negligible for the purpose of our computation.
We now inspect the four entries of $S$. We have
\begin{equation}\label{eq:two-word-two-D}
S_{12}=T_{12}(P_{22}-P_{11})+P_{12}(T_{11}-T_{22})+\Xi_{12}.
\end{equation}
Using \eqref{eq:two-commutator-formula-110}, \eqref{eq:two-word-two-6}, and \eqref{eq:two-word-two-8}, we obtain
\begin{align*}
T_{12}(P_{22}-P_{11})
&\equiv \bar u_0^{\,2}\alpha\mu^2\beta\,\varpi^{c-1}
\pmod{\varpi^c},\\
P_{12}(T_{11}-T_{22})
&\equiv -\bar u_0^{\,2}\alpha\mu^2\beta\,\varpi^{c-1}
\pmod{\varpi^c}.
\end{align*}
Therefore
\begin{equation}\label{eq:two-word-two-9}
v_\fp(S_{12})\ge c.
\end{equation}
For the $(2,1)$-entry,
\begin{equation}\label{eq:two-word-two-E}
S_{21}=T_{21}(P_{11}-P_{22})+P_{21}(T_{22}-T_{11})+\Xi_{21}.
\end{equation}
Using \eqref{eq:two-commutator-formula-113}, \eqref{eq:two-word-two-6}, \eqref{eq:two-word-two-8}, and $v_\fp(P_{21})\ge a$, we obtain
\begin{align*}
& v_\fp\bigl(T_{21}(P_{11}-P_{22})\bigr)\ge m+(a+\nu-1)=c,\\
 & v_\fp\bigl(P_{21}(T_{22}-T_{11})\bigr)\ge a+(m+\nu-1)=c.
\end{align*}
Hence
\begin{equation}\label{eq:two-word-two-10}
v_\fp(S_{21})\ge c.
\end{equation}

We have \(S_{11}=T_{12}P_{21}-P_{12}T_{21}+\tr(T)\tr(P)+\Xi_{11}\). Using \eqref{eq:two-commutator-formula-17} and
\eqref{eq:two-commutator-formula-111}, we obtain
\begin{align*}
T_{12}P_{21}-P_{12}T_{21}
&=
\bigl(P_{12}(R_{11}-R_{22})+R_{12}(P_{22}-P_{11})+\Omega_{12}\bigr)P_{21}\\
&\qquad
-P_{12}\bigl(R_{21}(P_{11}-P_{22})+P_{21}(R_{22}-R_{11})+\Omega_{21}\bigr)\\
&=
2P_{12}P_{21}(R_{11}-R_{22})
+R_{12}(P_{22}-P_{11})P_{21}\\
&\qquad
-P_{12}R_{21}(P_{11}-P_{22})
+\Omega_{12}P_{21}-P_{12}\Omega_{21}.
\end{align*}
By \eqref{eq:two-commutator-formula-11} and $v_\fp(P_{21})\ge a$,
\begin{equation}\label{eq:two-word-two-A}
v_\fp\!\bigl(2P_{12}P_{21}(R_{11}-R_{22})\bigr)
\ge \nu+(a-1)+a+(a+b+\nu-2)=c.
\end{equation}
By \eqref{eq:two-commutator-formula-14} and \eqref{eq:two-word-two-6},
\begin{equation}\label{eq:two-word-two-B}
v_\fp\!\bigl(R_{12}(P_{22}-P_{11})P_{21}\bigr)
\ge (a+b+\nu-1)+(a+\nu-1)+a=c+1.
\end{equation}
By \eqref{eq:two-commutator-formula-15},
\begin{equation}\label{eq:two-word-two-C}
v_\fp(\Omega_{12}P_{21})\ge c,\qquad
v_\fp(P_{12}\Omega_{21})\ge c.
\end{equation}
Therefore, by \eqref{eq:two-commutator-formula-13}, \eqref{eq:two-word-two-6}, and $P_{12}\equiv \mu\varpi^{a-1}\pmod{\varpi^a}$,
\begin{equation}\label{eq:two-word-two-11}
T_{12}P_{21}-P_{12}T_{21}
\equiv
-P_{12}R_{21}(P_{11}-P_{22})
\equiv
\bar u_0^{\,2}\alpha^2\mu\beta\,\varpi^{c-1}
\pmod{\varpi^c}.
\end{equation}
Hence
\[
S_{11}\equiv
\bar u_0^{\,2}\alpha^2\mu\beta\,\varpi^{c-1}
\pmod{\varpi^c}.
\]
We have
\[
S_{22}=T_{21}P_{12}-P_{21}T_{12}+\tr(T)\tr(P)+\Xi_{22}.
\]
Expanding as above gives
\begin{align*}
T_{21}P_{12}-P_{21}T_{12}
&=
-2P_{12}P_{21}(R_{11}-R_{22})
+R_{21}(P_{11}-P_{22})P_{12}\\
&\qquad
-P_{21}R_{12}(P_{22}-P_{11})
+\Omega_{21}P_{12}-P_{21}\Omega_{12}.
\end{align*}
The first term has valuation at least \(c\), the third and last two terms are
\(>c-1\), and by \eqref{eq:two-commutator-formula-13}, \eqref{eq:two-word-two-6}, and $P_{12}\equiv \mu\varpi^{a-1}\pmod{\varpi^a}$,
\[
R_{21}(P_{11}-P_{22})P_{12}
\equiv
-\bar u_0^{\,2}\alpha^2\mu\beta\,\varpi^{c-1}
\pmod{\varpi^c}.
\]
Therefore \(S_{22}\equiv
-\bar u_0^{\,2}\alpha^2\mu\beta\,\varpi^{c-1}
\pmod{\varpi^c}\).
\end{proof}

\begin{lem}\label{lem:two-word-three}
Assume the hypotheses of Lemma~\ref{lem:two-word-one}, and assume moreover that
$\alpha\neq 0,a\ge 2\nu+2, 4a+b+3\nu-4\le e$.
Set
\[
\widehat{(x,y)}:=((((x,y),x),x),x).
\]
Then
\[
\widehat{(x,y)}\in N_{4a+b+3\nu-4},
\quad
\psi_{4a+b+3\nu-4}\big(\widehat{(x,y)}\big)
=
\bar u_0^{\,3}\alpha^2\mu\beta\,(\alpha D+\mu E).
\]
\end{lem}

\begin{proof}
Set $z:=(x,y)$, $w:=(z,x)$, $u:=(w,x)$, $v:=(u,x)=\widehat{(x,y)}$, and write $x=I+P$, $z=I+R$, $w=I+T$, $u=I+S$, $v=I+V$. Also denote
\[
m:=2a+b+\nu-2,\quad
c:=3a+b+2\nu-3,\quad
d:=4a+b+3\nu-4.
\]

By Lemma~\ref{lem:two-word-two}, $u\in N_c$ with $\psi_c(u)=\bar u_0^{\,2}\alpha^2\mu\beta\,D$. 
Set $\Delta_P:=P_{22}-P_{11}$. 
We first sharpen the information on \(S_{11}-S_{22}\). By
Lemma~\ref{lem:two-word-two},
\[
S\in N_c,\qquad
\psi_c(u)=\bar u_0^{\,2}\alpha^2\mu\beta\,D.
\]
Hence
\[
S_{11}\equiv
\bar u_0^{\,2}\alpha^2\mu\beta\,\varpi^{c-1}
\pmod{\varpi^c}.
\]
Since \(u=I+S\in \SL_2(R_e)\), Lemma~\ref{lem:trace-val} gives
\[
v_\fp(\tr(S))\ge 2c-2\ge c+\nu.
\]
Therefore \(S_{11}-S_{22}=2S_{11}-\tr(S)\), so
\begin{equation}\label{eq:two-word-three-12}
S_{11}-S_{22}
\equiv
\bar u_0^{\,3}\alpha^2\mu\beta\,\varpi^{c+\nu-1}
\pmod{\varpi^{c+\nu}}.
\end{equation}

Now applying Lemma~\ref{lem:lie-bracket-ring} to the pair $(u,x)$ we get
\begin{align*}
V & =SP-PS+\tr(S)\tr(P)I+\Lambda,\\
v_\fp(\Lambda) & \ge c+2a-3=5a+b+2\nu-6\ge d.
\end{align*}
We begin with the $(1,2)$-entry:
\[
V_{12}=S_{12}(P_{22}-P_{11})+P_{12}(S_{11}-S_{22})+\Lambda_{12}.
\]
By \eqref{eq:two-word-two-9} and \eqref{eq:two-word-two-6},
\[
v_\fp\bigl(S_{12}(P_{22}-P_{11})\bigr)\ge c+(a+\nu-1)=d.
\]
Using $P_{12}\equiv \mu\varpi^{a-1}\pmod{\varpi^a}$ and \eqref{eq:two-word-three-12}, we obtain
\begin{equation}\label{eq:two-word-three-13}
V_{12}\equiv
\bar u_0^{\,3}\alpha^2\mu^2\beta\,\varpi^{d-1}
\pmod{\varpi^d}.
\end{equation}
For the $(2,1)$-entry,
\[
V_{21}=S_{21}(P_{11}-P_{22})+P_{21}(S_{22}-S_{11})+\Lambda_{21}.
\]
By \eqref{eq:two-word-two-10}, \eqref{eq:two-word-two-6}, \eqref{eq:two-word-three-12}, and $v_\fp(P_{21})\ge a$,
\begin{align*}
& v_\fp\bigl(S_{21}(P_{11}-P_{22})\bigr)\ge c+(a+\nu-1)=d,\\
& v_\fp\bigl(P_{21}(S_{22}-S_{11})\bigr)\ge a+(c+\nu-1)=d.
\end{align*}
Hence
\begin{equation}\label{eq:two-word-three-14}
v_\fp(V_{21})\ge d.
\end{equation}

We now compute the diagonal entries. We have
\[
V_{11}=S_{12}P_{21}-P_{12}S_{21}+\tr(S)\tr(P)+\Lambda_{11}.
\]
Substituting \eqref{eq:two-word-two-D} and \eqref{eq:two-word-two-E} gives
\begin{align}\label{eq:two-word-three-15}
V_{11}
&= (T_{12}P_{21}+P_{12}T_{21})\Delta_P
+2P_{12}P_{21}(T_{11}-T_{22}) \notag\\
&\quad +\Xi_{12}P_{21}-P_{12}\Xi_{21}
+\tr(S)\tr(P)+\Lambda_{11}.
\end{align}
Using \eqref{eq:two-commutator-formula-17} and \eqref{eq:two-commutator-formula-111}, the same expansion as in the proof of Lemma~\ref{lem:two-word-two} gives
\[
T_{12}P_{21}+P_{12}T_{21}
=
R_{12}\Delta_P P_{21}
-P_{12}R_{21}\Delta_P
+\Omega_{12}P_{21}+P_{12}\Omega_{21},
\]
the $2P_{12}P_{21}(R_{11}-R_{22})$ term now cancelling because of the sign change. Substituting into \eqref{eq:two-word-three-15} yields
\begin{align*}
V_{11}
&=
R_{12}\Delta_P^2 P_{21}
-P_{12}R_{21}\Delta_P^2
+(\Omega_{12}P_{21}+P_{12}\Omega_{21})\Delta_P\\
&\quad
+2P_{12}P_{21}(T_{11}-T_{22})
+\Xi_{12}P_{21}-P_{12}\Xi_{21}
+\tr(S)\tr(P)+\Lambda_{11}.
\end{align*}
By \eqref{eq:two-word-two-B}, \eqref{eq:two-word-two-C}, and \eqref{eq:two-word-two-6}, after multiplication by $\Delta_P$,
\[
v_\fp(R_{12}\Delta_P^2 P_{21})\ge d+1,\qquad
v_\fp\bigl((\Omega_{12}P_{21}+P_{12}\Omega_{21})\Delta_P\bigr)\ge d.
\]
By \eqref{eq:two-word-two-11} and \eqref{eq:two-word-two-6},
\[
-P_{12}R_{21}\Delta_P^2
\equiv
-\bar u_0^{\,3}\alpha^3\mu\beta\,\varpi^{d-1}
\pmod{\varpi^d}.
\]
By \eqref{eq:two-word-two-8} and $v_\fp(P_{21})\ge a$,
\[
v_\fp\bigl(2P_{12}P_{21}(T_{11}-T_{22})\bigr)
\ge \nu+(a-1)+a+(m+\nu-1)=d.
\]
Since by \eqref{eq:two-commutator-formula-20} 
\[
v_\fp(\Xi)
\ge 4a+b+\nu-5
\ge c+\nu
\]
because $a\ge 2\nu+2$, we have $v_\fp(\Xi_{12}P_{21}), v_\fp(P_{12}\Xi_{21})\ge d$. By Lemma~\ref{lem:trace-val}, $v_\fp(\tr(S))\ge 2c-2$ and $v_\fp(\tr(P))\ge 2a-2$, so $v_\fp(\tr(S)\tr(P))\ge 2c+2a-4>d-1$. Therefore
\begin{equation}\label{eq:two-word-three-17}
V_{11}\equiv
-\bar u_0^{\,3}\alpha^3\mu\beta\,\varpi^{d-1}
\pmod{\varpi^d}.
\end{equation}

For the lower-right entry,
\[
V_{22}=S_{21}P_{12}-P_{21}S_{12}+\tr(S)\tr(P)+\Lambda_{22}.
\]
The same expansion, with signs reversed, gives
\begin{align*}
V_{22}
&=
-(R_{12}P_{21}-P_{12}R_{21})\Delta_P^2
-(\Omega_{21}P_{12}+P_{21}\Omega_{12})\Delta_P\\
&\quad
-2P_{12}P_{21}(T_{11}-T_{22})
+\Xi_{21}P_{12}-P_{21}\Xi_{12}
+\tr(S)\tr(P)+\Lambda_{22}.
\end{align*}
Exactly as above, all terms except $+P_{12}R_{21}\Delta_P^2$ are in
$\varpi^d$, and therefore
\begin{equation}\label{eq:two-word-three-20}
V_{22}\equiv
\bar u_0^{\,3}\alpha^3\mu\beta\,\varpi^{d-1}
\pmod{\varpi^d}.
\end{equation}

Now the conclusion follows from \eqref{eq:two-word-three-13}, \eqref{eq:two-word-three-14}, \eqref{eq:two-word-three-17}, and \eqref{eq:two-word-three-20}. Note that the residue characteristic is $2$, so $\alpha=-\alpha$.
\end{proof}

\subsection{Level progressions}

\subsubsection{}\label{sec:bad-lines}
For $2\le j\le e$, define
\[
\Lambda_j := \{\ell\in\bP^1(\bF_q) :
W_j\cap\Ann(\ell)\ne\{0\}\},
\]
where $\Ann(\ell)=\{A\in\mathfrak{sl}_2(\bF_q)\mid
Av=0\;\forall v\in\ell\}$
is a one-dimensional $\bF_q$-subspace of
$\mathfrak{sl}_2(\bF_q)$
spanned by a nonzero nilpotent element.
Equivalently, $\ell\in\Lambda_j$ if and only if
$W_j$ contains a nonzero nilpotent whose kernel
is~$\ell$. Remark that each $\Lambda_j$ is $H_1$-invariant.

\subsubsection{}
The following simple observation will be used repeatedly:
\begin{lem}\label{lem:propagation}
Let $A\in W_j$ for some $j\ge 2$.
Assume that either $A$ is nilpotent and $j\ge e_0+1$, or $j\ge e_0+2$. Then there exists $u\in\bF_q^*$ such that
\(
uA\in W_{j+e_0}
\).
\end{lem}

\begin{proof}
Choose $h\in H$ with \(h\equiv I+\varpi^{j-1}A \pmod{\varpi^j}\). Write
\(B:=h-I, B=\varpi^{j-1}A+C, v_\fp(C)\ge j\).
Then
\[
h^p=(I+B)^p=\sum_{k=0}^p \binom{p}{k} B^k.
\]
Recall that
\(p=u_0\varpi^{e_0}\) with \(u_0\in R_e^\times\), and \(\overline{u_0}\in\bF_q^\times\) denotes the reduction of \(u_0\).

\noindent
\emph{Case 1: $A$ is nilpotent.}
Then $A^2=0$.
For $k\ge 2$, expand
\[
B^k=(\varpi^{j-1}A+C)^k.
\]
The pure leading term is
\[
\varpi^{k(j-1)}A^k=0
\qquad (k\ge 2),
\]
and every other term contains at least one factor of $C$, hence has $v_\fp$ valuation
\[
 \ge (k-1)(j-1)+j = k(j-1)+1
\]
for all $k\ge 2$, $j\ge 2$.
For $2\le k\le p-1$, we obtain
\begin{align*}
& v_\fp\!\left(\binom{p}{k}B^k\right)
\ge e_0+k(j-1)+1
\ge e_0+2(j-1)+1\ge j+e_0, \\
& v_\fp(B^p)\ge p(j-1)+1. 
\end{align*}
The condition $p(j-1)+1\ge j+e_0$,
equivalently $(p-1)(j-1)\ge e_0-1$,
holds since $j\ge e_0+1$.
Therefore
\begin{equation}\label{eq:propagation-one}
v_\fp\!\left(\binom{p}{k}B^k\right)\ge j+e_0
\qquad (k\ge 2).
\end{equation}

\noindent
\emph{Case 2: $j\ge e_0+2$.}
When $j\ge e_0+2$ the bound \eqref{eq:propagation-one} holds too without the need to use nilpotence to deal with the leading $A$ power.

Thus in either case, we have
\begin{align*}
& pB=u_0\varpi^{e_0}(\varpi^{j-1}A+C)
=u_0\varpi^{j-1+e_0}A+O(\varpi^{j+e_0}),\\
& h^p\equiv I+u_0\varpi^{j-1+e_0}A \pmod{\varpi^{j+e_0}}.
\end{align*}
Therefore
\(\bar u_0 A\in W_{j+e_0}\).
\end{proof}

Here are some immediate consequences of Lemma~\ref{lem:propagation}:
\begin{cor}
\label{cor:bad-lines-persist}
For any $j\ge e_0+1$ we have
$\Lambda_j \subset \Lambda_{j+e_0}$.
In particular,
$\Lambda_j\subset \Lambda_k$
for all $k\ge j$
with $k\equiv j\pmod{e_0}$.
\end{cor}

\begin{cor}
\label{cor:dim-nondecreasing}
For all $j\ge e_0+2$, \(d_{j+e_0}\ge d_j\).
\end{cor}

\begin{cor}\label{cor:all-active}
Assume
$H\subset G_e$ is of exact local level~$e$, i.e.\ $d_e<3f$. Then every level $j\in\{e_0+2,\ldots,e\}$ with $j\equiv e\pmod{e_0}$ satisfies $d_j<3f$.
\end{cor}

\subsection{Fixed points and slope decomposition}

\subsubsection{}\label{sec:fixed-points-trace-two}
We now study the set of fixed points for the action of an arbitrary element $w\in G_e$ on $X_e$; this set is denoted $X_e^w$.

\begin{lem}\label{lem:ss-fp}
Let $w\in G_e$ satisfy
$w\equiv I+\varpi^{j-1}A\pmod{\varpi^j}$
for some $j\in\{2,\ldots,e\}$ and some
$A\in\mathfrak{sl}_2(\bF_q)$ with $\det(A)\ne 0$.
Then $X_e^w=\emptyset$.
\end{lem}

\begin{proof}
Assume for contradiction that there exists \(v=\binom{a}{c}\in X_e^w\).
Then \((w-I)v=0\) in \(R_e^2\).
Since \(w\equiv I+\varpi^{j-1}A\pmod{\varpi^j}\), we can write
\(w=I+\varpi^{j-1}M\) with \(M\in M_2(R_e)\) and
\(M\equiv A\pmod{\varpi}\).
Hence \(\varpi^{j-1}Mv=0\) in \(R_e^2\), so \(A\binom{\bar a}{\bar c}=0\) in \(\bF_q^2\).
But \(A\) is invertible over \(\bF_q\), and therefore
\((\bar a,\bar c)=(0,0)\), contradicting the primitivity of \(v\in X_e\).
\end{proof}

\begin{lem}\label{lem:trace-vanishing}
Let $w\in\SL_2(R_e)$.
If $\tr(w)\ne 2$ in $R_e$,
then $X_e^w=\emptyset$.
\end{lem}

\begin{proof}
By Cayley--Hamilton Theorem $(w-I)^2=(\tr(w)-2)\,w$.
If $v\in X_e^w$ then $(w-I)v=0$,
so $(\tr(w)-2)\,v=0$.
Writing $t:=\tr(w)-2$ and $v=\binom{a}{c}$, we get
$ta=tc=0$.
If $v_\fp(t)<e$:
$v_\fp(a),v_\fp(c)\ge 1$,
contradicting primitivity.
\end{proof}

\subsubsection{}\label{sec:slope-decomposition}
Let \(\ell=[1\!:\!0]\in\bP^1(\bF_q)\) denote the standard line, and recall that \(E=e_{12}=
\begin{pmatrix}0&1\\0&0\end{pmatrix}\), so that \(\ker(E)=\ell\).  For \(j\in\{2,\dots,e\}\), consider the
\(\bF_q\)-line \(\bF_qE\subset\mathfrak{sl}_2(\bF_q)\). Define
\begin{align*}
& T_j(\ell):=
\#\Bigl\{
w\in N_j\setminus N_{j+1}\mid\ \tr(w)=2,\ X_e^w\neq \emptyset,\ 
\psi_j(w)\in \bF_q^* E
\Bigr\},\\
& \Sigma_j:=\varpi R_{e-j+1}\subset R_{e-j+1}.
\end{align*}

For \(\eta\in\Sigma_j\), choose any lift
\(\widetilde\eta\in\varpi R_e\), and write
\[
N(\widetilde\eta):=
\begin{pmatrix}
-\widetilde\eta & 1\\
-\widetilde\eta^2 & \widetilde\eta
\end{pmatrix}.
\]
We shall write \(N(\eta)\) in expressions such as \(bN(\eta)\) when
\(v_\fp(b)\ge j-1\); Lemma~\ref{lem:borel-slope-well-defined} below shows
that these expressions are independent of the chosen lift. Note that, for any lift \(\widetilde\eta\in\varpi R_e\),
\[
\begin{pmatrix}1&0\\ \widetilde\eta&1\end{pmatrix}E \begin{pmatrix}1&0\\ \widetilde\eta&1\end{pmatrix}^{-1}
=
\begin{pmatrix}
-\widetilde\eta & 1\\
-\widetilde\eta^2 & \widetilde\eta
\end{pmatrix}.
\]
Such conjugation is also used in \cite{CoxParry1984}. Recall that elements of $X_e$ are primitive columns; see Prop.~\ref{prop:primitive-cols}. Define
\begin{align}
M_j(\eta) & :=
\Bigl\{
b\in R_e\mid\ v_\fp(b)=j-1,\ I+bN(\eta)\in N_j
\Bigr\},\\
m_j(\eta) & :=|M_j(\eta)|, \\
\Fix_j(\eta) & :=
\left\{
\binom{x}{y}\in X_e\mid\ x\in R_e^*,\ 
y\equiv \widetilde\eta\,x \pmod{\varpi^{e-j+1}}
\right\}.
\end{align}

We need to check that the expressions \(bN(\eta)\) and the set
\(\Fix_j(\eta)\) do not depend on the choice of \(\widetilde\eta\).
\begin{lem}\label{lem:borel-slope-well-defined}
Let \(\eta\in \Sigma_j\), choose a lift
\(\widetilde\eta\in\varpi R_e\), and let \(b\in R_e\) satisfy
\(v_\fp(b)=j-1\).  Then the matrix \(bN(\widetilde\eta)\in M_2(R_e)\)
is independent of the chosen lift.  Moreover, for every lift
\(\widetilde\eta\),
\[
N(\widetilde\eta)^2=0, \quad
N(\widetilde\eta)\equiv E \pmod{\varpi}, \quad
\det(I+bN(\widetilde\eta))=1.
\]
In particular, $I+bN(\widetilde\eta)\in N_j$, and $\psi_j(I+bN(\widetilde\eta))\in \bF_q^* E$.
\end{lem}

\begin{proof}
If $\widetilde\eta'=\widetilde\eta+\varpi^{e-j+1}\xi$, then every entry of
$N(\widetilde\eta')-N(\widetilde\eta)$ lies in $\varpi^{e-j+1}R_e$.
Multiplying by $b$ with $v_\fp(b)=j-1$ kills this difference in $R_e$, since
$(j-1)+(e-j+1)=e$. Thus $bN(\eta)$ is well-defined. The other claims are straightforward computations.
\end{proof}

\begin{lem}\label{lem:fix-slope-class}
The set $\Fix_j(\eta)$ is independent of the chosen lift
$\widetilde\eta$. Moreover, \(\lvert\Fix_j(\eta)\rvert=(q-1)\,q^{e+j-2}\).
\end{lem}

\begin{proof}
If $\widetilde\eta'=\widetilde\eta+\varpi^{e-j+1}\xi$, then
\[
\widetilde\eta' x\equiv \widetilde\eta x \pmod{\varpi^{e-j+1}}
\]
for every $x\in R_e$, so $\Fix_j(\eta)$ is well-defined. 
For each unit $x\in R_e^*$, the congruence equation
\[
y\equiv \widetilde\eta x \pmod{\varpi^{e-j+1}}
\]
has exactly $q^{j-1}$ solutions $y\in R_e$. Since $|R_e^*|=(q-1)q^{e-1}$, we get
\[
|\Fix_j(\eta)|=(q-1)q^{e-1}\cdot q^{j-1}
=(q-1)\,q^{e+j-2}.
\]
\end{proof}

\begin{prop}\label{prop:borel-slope-parameterization}
Let $2\le j\le e$ and $w\in N_j\setminus N_{j+1}$ satisfy
\[
\tr(w)=2,\quad X_e^w\neq \emptyset,\quad \psi_j(w)\in \bF_q^* E.
\]
Then there exists a unique pair $b\in R_e, \eta\in \Sigma_j$ such that
\[
v_\fp(b)=j-1,\quad w=I+bN(\eta).
\]
Moreover, one has \(X_e^w=\Fix_j(\eta)\).
Conversely, for every $\eta\in \Sigma_j$ and every $b\in M_j(\eta)$, the
matrix $w=I+bN(\eta)$ belongs to the set counted by $T_j(\ell)$ and satisfies
$X_e^w=\Fix_j(\eta)$.
\end{prop}

\begin{proof}
Write
\[
w=I+\begin{pmatrix}u&b\\ c&-u\end{pmatrix}
\in N_j\setminus N_{j+1}
\]
with $v_\fp(b)=j-1,v_\fp(u),v_\fp(c)\ge j$. 
Consider a primitive column $\binom{x}{y}\in R_e^2$ fixed by $w$. It satisfies $ux+by=0$, hence
\[
v_\fp(y)\ge v_\fp(ux)-v_\fp(b)\ge j+v_\fp(x)-(j-1)=v_\fp(x)+1.
\]
Primitivity thus forces $x\in R_e^*$. Set \(\widetilde\eta:=y/x\in R_e\). Then the fixed-point equations \(ux+by=0,cx-uy=0\) become
\[
u=-b\widetilde\eta,\qquad c=-b\widetilde\eta^2.
\]
Since $v_\fp(u)\ge j$ and $v_\fp(b)=j-1$, we obtain $v_\fp(\widetilde\eta)\ge 1$, so the class $\eta:=\widetilde\eta \bmod \varpi^{e-j+1}$ lies in $\Sigma_j=\varpi R_{e-j+1}$.

Now write $b=\varpi^{j-1}\beta$ with some $\beta\in R_e^*$, so multiplication by
$b$ induces a bijection
\[
R_{e-j+1}\xrightarrow{\sim}\varpi^{j-1}R_e,\qquad t\mapsto bt.
\]
Therefore the class $\eta$ in $R_{e-j+1}$ determines $u=-b\widetilde\eta$ uniquely in $R_e$,
and then $c=-b\widetilde\eta^2$ is also uniquely determined. Thus $w=I+bN(\eta)$.

The fixed-point description above shows that the primitive columns fixed by
$w$ are exactly those with
\[
x\in R_e^*,\qquad
y\equiv \widetilde\eta\,x \pmod{\varpi^{e-j+1}},
\]
namely $X_e^w=\Fix_j(\eta)$. 
Conversely, let $\eta\in\Sigma_j$ and $b\in M_j(\eta)$, and set $z=I+bN(\eta)$. By Lemma~\ref{lem:borel-slope-well-defined} $z$ is well defined. Moreover $\tr(z)=2$, $\det(z)=1$, and $X_e^z=\Fix_j(\eta)\neq \emptyset$.
\end{proof}

\begin{cor}\label{cor:Tj-slope-sum}
One has the identity \(T_j(\ell)=\sum_{\eta\in \Sigma_j} m_j(\eta)\).
\end{cor}

\begin{proof}
By Proposition~\ref{prop:borel-slope-parameterization}, the map $w\to \eta$ from the set counted by $T_j(\ell)$ to $\Sigma_j$ is well-defined, and its fiber
over $\eta$ is exactly the set
\[
\{\,I+bN(\eta)\mid \ b\in M_j(\eta)\,\}.
\]
\end{proof}

\begin{cor}\label{cor:trivial-Tj-bound}
We have $\lvert T_j(\ell) \rvert\le q^{2e-2j+1}$.
\end{cor}
\begin{proof}
The set $\Sigma_j=\varpi R_{e-j+1}$ has cardinal $q^{e-j}$. Each $M_j(\eta)$ is a subset of $\varpi^{j-1}R_e$ which has cardinal $q^{e-j+1}$. The identity of Corollary~\ref{cor:Tj-slope-sum} then implies that the cardinal of $T_j(\ell)$ is $\leq q^{e-j}q^{e-j+1}$.
\end{proof}

\subsection{Filtered counting}\label{sec:filtered-counting}

\subsubsection{}

For $\eta\in \Sigma_j$, define
\[
S_j(\eta):=\{\,b\in R_e\mid\  I+bN(\eta)\in N_j\,\}.
\]
Then every $b\in S_j(\eta)$ satisfies $v_\fp(b)\ge j-1$. We have
\[
M_j(\eta)=\{\,b\in S_j(\eta)\mid\ v_\fp(b)=j-1\,\}.
\]

\begin{lem}\label{lem:borel-Sj-basic}
For every $\eta\in \Sigma_j$, the set $S_j(\eta)$ is an additive subgroup of
$R_e$, $S_j(\eta)\subset \varpi^{j-1}R_e$, and $m_j(\eta)=|S_j(\eta)|-\bigl|S_j(\eta)\cap \varpi^jR_e\bigr|$.
\end{lem}
\begin{proof}
Since $N(\eta)^2=0$, one has
\begin{equation*}
(I+bN(\eta))(I+b'N(\eta))=I+(b+b')N(\eta),
\quad
(I+bN(\eta))^{-1}=I-bN(\eta).
\end{equation*}
As $N_j$ is a subgroup of $H$, it follows that $S_j(\eta)$ is an
additive subgroup of $R_e$. 
\end{proof}

\begin{lem}\label{lem:borel-Sj-recursion}
For $1\le r\le e-j$, let $\pi_{j,r}:\Sigma_j\to \Sigma_{j+r}$ be the natural reduction map. Then for every $\eta\in \Sigma_j$ one has
\[
S_j(\eta)\cap \varpi^{j+r-1}R_e
=
S_{j+r}\bigl(\pi_{j,r}(\eta)\bigr).
\]
\end{lem}

\begin{proof}
Any two lifts in $\varpi R_e$ of $\pi_{j,r}(\eta)$ differ by an element of $\varpi^{e-j-r+1}R_e$, so for $b\in \varpi^{j+r-1}R_e$ we have $bN(\eta)=bN\bigl(\pi_{j,r}(\eta)\bigr)$ in $M_2(R_e)$.
\end{proof}

\subsubsection{}\label{sec:telescopic}
For $j\in\{2,\dots,e\}$, define
\begin{align*}
\cA_j(\ell) & :=\bigcup_{\eta\in\Sigma_j}S_j(\eta)\;\subset\;\varpi^{j-1}R_e,\\
A_j(\ell)& :=\sum_{\eta\in \Sigma_j}|S_j(\eta)|, \quad A_{e+1}(\ell):=0.
\end{align*}
For \(b\in \varpi^{j-1}R_e\), define the fiber 
\[ 
\mathcal E_j(b):= \{\eta\in \Sigma_j\mid\ I+bN(\eta)\in H\}.
\] 
Recall that by Lemma~\ref{lem:borel-slope-well-defined} the matrix $I+bN(\eta)$ does not depend on the lift $\widetilde\eta$. The following formula is immediate.
\begin{lem}\label{lem:Aj-fiber-sum}
\[ 
A_j(\ell)=\sum_{b\in \varpi^{j-1}R_e} |\mathcal E_j(b)|. 
\]     
\end{lem}
\begin{proof}
Both sides count the set of pairs
$\{(\eta,b)\in \Sigma_j\times \varpi^{j-1}R_e\mid\ I+bN(\eta)\in H\}$,
the left by summing over $\eta$ (using $S_j(\eta)\subset \varpi^{j-1}R_e$
from Lemma~\ref{lem:borel-Sj-basic}), the right by summing over $b$.
\end{proof}

\begin{lem}\label{lem:borel-tail-telescoping}
For every $j\in\{2,\dots,e\}$, \(T_j(\ell)=A_j(\ell)-q\,A_{j+1}(\ell)\). 
Consequently, for every interval $[a,b]\subset \{2,\dots,e\}$,
\[
\sum_{r=a}^b q^{\,r-2e}T_r(\ell)
=
q^{\,a-2e}A_a(\ell)-q^{\,b+1-2e}A_{b+1}(\ell),
\]
and in particular $\sum_{r=a}^e q^{\,r-2e}T_r(\ell)=q^{\,a-2e}A_a(\ell)$.
\end{lem}
\begin{proof}
For \(j=e\), the set \(S_e(\eta)\) is contained in \(\varpi^{e-1}R_e\), and \(S_e(\eta)\cap\varpi^eR_e=\{0\}\).
Thus \(m_e(\eta)=|S_e(\eta)|\) while \(A_{e+1}(\ell)=0\).  Now assume \(j<e\). By Lemma~\ref{lem:borel-Sj-basic} and Lemma~\ref{lem:borel-Sj-recursion} applied with $r=1$,
\[
m_j(\eta)=|S_j(\eta)|-\bigl|S_{j+1}\bigl(\pi_{j,1}(\eta)\bigr)\bigr|.
\]
The map $\pi_{j,1}:\Sigma_j\to\Sigma_{j+1}$ is the reduction $\varpi R_{e-j+1}\to \varpi R_{e-j}$, which is surjective with fibers of cardinality $q$. Therefore
\[
\sum_{\eta\in\Sigma_j}\bigl|S_{j+1}\bigl(\pi_{j,1}(\eta)\bigr)\bigr|
=q\sum_{\eta'\in\Sigma_{j+1}}|S_{j+1}(\eta')|=q\,A_{j+1}(\ell).
\]
Combining this with Corollary~\ref{cor:Tj-slope-sum} and the definition of $A_j(\ell)$,
\[
T_j(\ell)=\sum_{\eta\in\Sigma_j}m_j(\eta)=A_j(\ell)-q\,A_{j+1}(\ell),
\]
which is the first identity. Multiplying by $q^{j-2e}$ gives
\[
q^{j-2e}T_j(\ell)=q^{j-2e}A_j(\ell)-q^{(j+1)-2e}A_{j+1}(\ell),
\]
and summing from $j=a$ to $j=b$ gives the result. 

\end{proof}

\section{Local asymptotic estimates II: proofs}\label{sec:local-estimates-II}

We continue with the same setup and notations as in Section~\ref{sec:local-estimates-I}. 

\subsection{Counting with Burnside's Lemma}\label{sec:Counting-with-Burnside}
In this subsection we establish the basic counting formulas that will guide the proof of Theorem~\ref{thm:main_local_chi}.

\subsubsection{Choices of $\mathfrak{sl}_2$-basis}
\label{sec:choices-of-basis}

The slope decomposition of \S\ref{sec:slope-decomposition} uses explicit
matrices.  It is therefore useful to separate the quantities which depend on
auxiliary coordinates from the quantities which depend only on a nilpotent
direction.

Let \(\lambda,\mu\in\bP^1(\bF_q)\) be two distinct lines.  Choose generators
\[
E_\lambda\in\Ann(\lambda)\subset \mathfrak{sl}_2(\bF_q),
\qquad
F_\mu\in\Ann(\mu)\subset \mathfrak{sl}_2(\bF_q),
\]
and set
\[
D_{\lambda,\mu}:=[E_\lambda,F_\mu]\in \mathfrak{sl}_2(\bF_q).
\]
If \(E_\lambda\) and \(F_\mu\) are rescaled, then \(D_{\lambda,\mu}\) is also
rescaled by a non-zero scalar.  Hence the three $\bF_q$-lines in $\mathfrak{sl}_2(\bF_q)$
\[
\bF_qE_\lambda,\qquad
\bF_qD_{\lambda,\mu},\qquad
\bF_qF_\mu
\]
depend only on the ordered pair \((\lambda,\mu)\).  Choosing such a pair and
such generators is what we mean by choosing an adapted
\(\mathfrak{sl}_2(\bF_q)\)-basis.

In \S\ref{sec:slope-decomposition} and \S\ref{sec:filtered-counting}, the
chosen pair is the standard pair
\[
\ell=[1\!:\!0],
\qquad
\ell'=[0\!:\!1],
\]
together with generators \(E\in\Ann(\ell)\) and \(F\in\Ann(\ell')\).  The
explicit normal form \(I+bN(\eta)\) is tied to this adapted basis.  Consequently
the auxiliary objects
\[
M_j(\eta),\qquad m_j(\eta),\qquad S_j(\eta),\qquad \mathcal E_j(b)
\]
depend on the choice of the second line and on the generators.

By contrast, according to its very definition \(T_j(\ell)\) depends only on the first line \(\ell\).  Indeed it
counts elements whose level-\(j\) image lies in the nilpotent line
\(\Ann(\ell)\), and this condition is independent of the choice of a generator
of \(\Ann(\ell)\).  Although \(A_j(\ell)\) is defined using the same auxiliary
slope coordinates, it is also intrinsic because it can be expressed in terms of the $T_j(\ell)$ by Lemma~\ref{lem:borel-tail-telescoping}.

\subsubsection{Conjugating to other lines}
\label{sec:conjugating-other-lines}

The identities of \S\ref{sec:slope-decomposition} and
\S\ref{sec:filtered-counting} were proved for the standard nilpotent line
\(\Ann(\ell)=\bF_qE\), where \(\ell=[1\!:\!0]\).  We explain now how to use them
when the nilpotent direction corresponds to an arbitrary
\(\lambda\in\bP^1(\bF_q)\).

Choose a generator \(E_\lambda\in\Ann(\lambda)\).  For \(2\le j\le e\), set
\[
T_j(\lambda)
:=
\#\Bigl\{
w\in N_j\setminus N_{j+1} \mid
\tr(w)=2,\ X_e^w\neq\emptyset,\ 
\psi_j(w)\in\bF_q^*E_\lambda
\Bigr\}.
\]
This number does not depend on the choice of \(E_\lambda\).  For
\(\lambda=\ell\), it is the previously defined \(T_j(\ell)\). Choose \(g\in\SL_2(\bF_q)\) with \(g\ell=\lambda\), and choose a lift
\(\hat g\in G_e\).  Since \(E\) has kernel \(\ell\), \(gEg^{-1}\) has kernel
\(\lambda\), hence \(gEg^{-1}\in\bF_q^*E_\lambda\).  Put
\(H_{\hat g}:=\hat g^{-1}H\hat g\).  For \(w\in H\), write
\(w_{\hat g}:=\hat g^{-1}w\hat g\).

Conjugation by \(\hat g\) carries the filtration of \(H\) onto the filtration
of \(H_{\hat g}\).  More precisely,
\[
w\in N_j(H)\setminus N_{j+1}(H)
\quad\Longleftrightarrow\quad
w_{\hat g}\in N_j(H_{\hat g})\setminus N_{j+1}(H_{\hat g}).
\]
If \(w\in N_j(H)\), then
\[
\psi_j^{H_{\hat g}}(w_{\hat g})
=
g^{-1}\psi_j^H(w)g.
\]
This follows directly from the congruence
\(w\equiv I+\varpi^{j-1}\psi_j^H(w)\pmod{\varpi^j}\), after conjugating by
\(\hat g\) and reducing the coefficient of \(\varpi^{j-1}\) modulo \(\varpi\).

The remaining conditions in the definition of \(T_j(\lambda)\) are also
preserved.  Trace is invariant under conjugation, and
\(x\mapsto \hat g^{-1}x\) identifies \(X_e^w\) with \(X_e^{w_{\hat g}}\).
Moreover,
\[
\psi_j^H(w)\in\bF_q^*E_\lambda
\quad\Longleftrightarrow\quad
\psi_j^{H_{\hat g}}(w_{\hat g})\in\bF_q^*E,
\]
because \(g^{-1}E_\lambda g\in\bF_q^*E\).  Hence conjugation identifies the
set counted by \(T_j(\lambda)\) for \(H\) with the set counted by the standard
quantity \(T_j(\ell)\) for \(H_{\hat g}\). In symbols,
\[
T_j(\lambda;H)=T_j(\ell;H_{\hat g}),
\]
where the semicolon indicates the subgroup used to form the count.  With
the previous notation, \(T_j(\lambda)=T_j(\ell;H_{\hat g})\).

Applying Lemma~\ref{lem:borel-tail-telescoping} to \(H_{\hat g}\), we obtain
\begin{equation}\label{eq:conjugated-tail-telescoping-pointwise}
T_j(\lambda)
=
A_j(\ell;H_{\hat g})-qA_{j+1}(\ell;H_{\hat g}).
\end{equation}
Consequently, for every interval \([a,b]\subset\{2,\ldots,e\}\),
\begin{equation}\label{eq:conjugated-tail-telescoping-interval}
\sum_{r=a}^b q^{\,r-2e}T_r(\lambda)
=
q^{\,a-2e}A_a(\ell;H_{\hat g})
-
q^{\,b+1-2e}A_{b+1}(\ell;H_{\hat g}).
\end{equation}

\subsubsection{Burnside's Lemma}
\label{sec:counting-formula}
 Recall that
\[
T_r
=
\#\{w\in N_r\setminus N_{r+1}\mid
\tr(w)=2,\ X_e^w\neq\emptyset\}.
\]
If \(w\) is counted by \(T_r\), then \(\psi_r(w)\neq 0\), and
Lemma~\ref{lem:ss-fp} gives \(\det(\psi_r(w))=0\).  Thus \(\psi_r(w)\) is a
non-zero nilpotent element of \(\mathfrak{sl}_2(\bF_q)\), hence has a unique
kernel line.  Therefore
\begin{equation}\label{eq:Tj-sum-over-lines}
T_r=\sum_{\lambda\in\bP^1(\bF_q)}T_r(\lambda).
\end{equation}

\begin{prop}\label{prop:exact-chi}
For every \(j_0\in\{2,\ldots,e\}\), one has
\begin{equation}\label{eq:exact-chi-total}
\chi_e(N_{j_0})
=
q^{-e}
+
\frac{1}{q+1}\sum_{r=j_0}^e T_r q^{\,r-2e}.
\end{equation}
Equivalently,
\begin{equation}\label{eq:exact-chi-lines}
\chi_e(N_{j_0})
=
q^{-e}
+
\frac{1}{q+1}
\sum_{\lambda\in\bP^1(\bF_q)}
\sum_{r=j_0}^e T_r(\lambda)q^{\,r-2e}.
\end{equation}
\end{prop}

\begin{proof}
Burnside's lemma says that the number of orbits is the average number of fixed
points:
\[
|N_{j_0}\backslash X_e|
=
\frac{1}{|N_{j_0}|}\sum_{w\in N_{j_0}}|X_e^w|.
\]
Since
\(N_{j_0}=\{1\}\sqcup\bigsqcup_{r=j_0}^e(N_r\setminus N_{r+1})\), this gives
\[
|N_{j_0}\backslash X_e|
=
\frac{|X_e|}{|N_{j_0}|}
+
\frac{1}{|N_{j_0}|}
\sum_{r=j_0}^e
\sum_{w\in N_r\setminus N_{r+1}}|X_e^w|.
\]
For \(w\in N_r\setminus N_{r+1}\), Lemmas~\ref{lem:trace-vanishing} and
\ref{lem:ss-fp} show that \(w\) contributes only if it is counted by \(T_r\).
For such a \(w\), the conjugation argument of
\S\ref{sec:conjugating-other-lines}, together with
Proposition~\ref{prop:borel-slope-parameterization} and
Lemma~\ref{lem:fix-slope-class}, gives \(|X_e^w|=(q-1)q^{e+r-2}\). Hence
\[
\sum_{w\in N_r\setminus N_{r+1}}|X_e^w|
=
(q-1)q^{e+r-2}T_r.
\]
Dividing by \([G_e:N_{j_0}]=|G_e|/|N_{j_0}|\), we obtain
\[
\chi_e(N_{j_0})
=
\frac{|X_e|}{|G_e|}
+
\sum_{r=j_0}^e
\frac{(q-1)q^{e+r-2}}{|G_e|}T_r.
\]
By Lemma~\ref{lem:order-Ge}, one has \(|G_e|=q^{3e-2}(q^2-1)\). Moreover, Proposition~\ref{prop:primitive-cols} gives \(|X_e|=q^{2e-2}(q^2-1)\). Thus
\[
\chi_e(N_{j_0})
=
q^{-e}
+
\frac{1}{q+1}\sum_{r=j_0}^e T_rq^{\,r-2e}.
\]
This proves \eqref{eq:exact-chi-total}, and \eqref{eq:exact-chi-lines}
follows from \eqref{eq:Tj-sum-over-lines}.
\end{proof}
\begin{lem}\label{lem:normal-subgroup}
Let \(H\subset G_e\) be a subgroup, and let \(\mathcal N\trianglelefteq H\) be a
normal subgroup.  Then
\[
\chi_e(H)\le [H:\mathcal N]\chi_e(\mathcal N).
\]
\end{lem}

\begin{proof}
Since \(\mathcal N\subset H\), every \(H\)-orbit on \(X_e\) is a union of
\(\mathcal N\)-orbits.  Thus
\[
|H\backslash X_e|\le |\mathcal N\backslash X_e|.
\]
Since \([G_e:H]=[G_e:\mathcal N]/[H:\mathcal N]\), we get
\[
\chi_e(H)
=
\frac{|H\backslash X_e|}{[G_e:H]}
\le
[H:\mathcal N]\frac{|\mathcal N\backslash X_e|}{[G_e:\mathcal N]}
=
[H:\mathcal N]\chi_e(\mathcal N).
\]
\end{proof}

\begin{lem}\label{lem:trivial-chi-bound}
For every \(j_0\in\{2,\ldots,e\}\), one has
\[
  \chi_e(N_{j_0})
  \le
  q^{-e}+\sum_{j=j_0}^e q^{1-j}.
\]
\end{lem}

\begin{proof}
By Corollary~\ref{cor:trivial-Tj-bound} and the conjugation argument of
\S\ref{sec:conjugating-other-lines},
\[
T_j(\lambda)\le q^{2e-2j+1},
\qquad
\lambda\in\bP^1(\bF_q).
\]
Inserting this into \eqref{eq:exact-chi-lines} gives
\[
\chi_e(N_{j_0})
\le
q^{-e}
+
\frac{1}{q+1}
\sum_{\lambda\in\bP^1(\bF_q)}
\sum_{j=j_0}^e q^{2e-2j+1}q^{j-2e}
=
q^{-e}+\sum_{j=j_0}^e q^{1-j}.
\]
The formula follows by summing the finite geometric series.
\end{proof}

\begin{rem}\label{rem:trivial-bound-not-enough}
Since \(N_2\trianglelefteq H\),
Lemma~\ref{lem:normal-subgroup} gives
\[
\chi_e(H)\le [H:N_2]\chi_e(N_2).
\]
The trivial estimate of Lemma~\ref{lem:trivial-chi-bound} provides a uniform upper bound for $\chi_e(H)$, which does not yield the decay required in
Theorem~\ref{thm:main_local_chi}. Most of the rest of
the section is devoted to sharper estimates for the weighted sums of the
\(T_j(\lambda)\).
\end{rem}

\subsection{Proof of Theorem~\ref{thm:main_local_chi_B}}\label{sec:proof-of-main-local-thm-B}

\begin{prop}\label{prop:proper-subgroup-level-one-bound}
Let \(\Delta\) be a proper subgroup of \(G_1=\SL_2(\bF_q)\), and assume \(q\ge 59\). Then \(\chi_1(\Delta)\le \frac{2}{q}\).
\end{prop}

\begin{proof}
We first compute the fixed-point contribution on \(X_1\).  Using the
identification \(X_1=G_1/U_1\simeq \bF_q^2\setminus\{0\}\), an element
\(\gamma\in G_1\) fixes a point of \(X_1\) if and only if it fixes a non-zero
vector.  Since \(\det(\gamma)=1\), this is equivalent to \(1\) being an
eigenvalue of \(\gamma\), hence to \(\tr(\gamma)=2\).  The identity fixes all
\(q^2-1\) points of \(X_1\).  If \(\gamma\ne I\) and \(\tr(\gamma)=2\), then
\(\gamma\) has a one-dimensional fixed space, so it fixes exactly \(q-1\)
points of \(X_1\).

Let \(u(\Delta):=
\#\{\gamma\in\Delta\setminus\{I\}\mid\tr(\gamma)=2\}\). Burnside's lemma gives
\[
|\Delta\backslash X_1|
=
\frac{1}{|\Delta|}
\bigl((q^2-1)+u(\Delta)(q-1)\bigr).
\]
Since \(|G_1|=q(q^2-1)\), it follows that
\begin{equation}\label{eq:level-one-chi-u}
\chi_1(\Delta)
=
\frac{|\Delta\backslash X_1|}{[G_1:\Delta]}
=
q^{-1}+\frac{u(\Delta)}{q(q+1)}.
\end{equation}

It remains to prove \(u(\Delta)\le q+1\).  Let
\(\overline{\Delta}\) be the image of \(\Delta\) in \(\PSL_2(\bF_q)\).  This
image is proper.  Indeed, if \(q\) is even, then
\(\SL_2(\bF_q)=\PSL_2(\bF_q)\), so this is immediate.  If \(q\) is odd and
\(\overline{\Delta}=\PSL_2(\bF_q)\), then either \(\Delta=G_1\), or
\(\Delta\) is an index-two subgroup of \(G_1\); the latter is impossible
because \(\SL_2(\bF_q)\) is perfect for \(q\ge4\).

We shall compare \(u(\Delta)\) with the number of non-trivial \(p\)-elements
in \(\overline{\Delta}\).  If \(\gamma\ne I\) and \(\tr(\gamma)=2\), then the
characteristic polynomial of \(\gamma\) is \((X-1)^2\).  Hence the image of
\(\gamma\) in \(\PSL_2(\bF_q)\) has \(p\)-power order.  Moreover, the
projection \(\SL_2(\bF_q)\to\PSL_2(\bF_q)\) is injective on the set counted by
\(u(\Delta)\): if \(\gamma'\) and \(\gamma\) have trace \(2\) and the same
projective image, then \(\gamma'=z\gamma\) for a central scalar \(z\); in even
characteristic \(z=I\), while in odd characteristic \(z=-I\) would change the
trace from \(2\) to \(-2\).  Thus \(z=I\).  Therefore \(u(\Delta)\) is bounded
above by the number of non-trivial \(p\)-elements of \(\overline{\Delta}\).

By Dickson's classification \cite[Theorems~6.25 and~6.26]{SuzukiGroupTheoryI}, the proper subgroups of
\(\PSL_2(\bF_q)\) are of the following types: \(p\)-groups, Borel type, cyclic
or dihedral type, subfield type, or exceptional type \(A_4,S_4,A_5\).

We now bound the number of non-trivial \(p\)-elements in each case.  A
\(p\)-subgroup of \(\PSL_2(\bF_q)\) is contained in the image of the upper
unitriangular subgroup, which has order \(q\).  Hence the \(p\)-group case
contributes at most \(q-1\) non-trivial \(p\)-elements.

In the Borel case, the \(p\)-elements lie in the unipotent radical of the
Borel subgroup.  This radical has order \(q\), so the number of non-trivial
\(p\)-elements is at most \(q-1\).

In the cyclic case, the non-trivial \(p\)-elements form a cyclic \(p\)-subgroup
of \(\PSL_2(\bF_q)\), hence there are at most \(q-1\) of them.  In the dihedral
case, if \(p\) is odd, there are no non-trivial \(p\)-elements in a torus
normalizer.  If \(p=2\), a dihedral subgroup of a torus normalizer has at most
\(q+1\) involutions.  Thus the cyclic and dihedral cases contribute at most
\(q+1\) non-trivial \(p\)-elements.

In the subfield case, the subgroup is contained in a group of type
\(\PSL_2(\bF_{q_0})\) or \(\PGL_2(\bF_{q_0})\), where
\(\bF_{q_0}\subsetneq\bF_q\) is a proper subfield.  Thus \(q=q_0^m\) with
\(m\ge2\).  In these groups, the non-trivial \(p\)-elements are contained in
the conjugates of the upper triangular unipotent subgroup over \(\bF_{q_0}\).  There
are \(q_0+1\) such conjugates, each has \(q_0-1\) non-identity elements, and
distinct such subgroups intersect trivially.  Hence the number of non-trivial
\(p\)-elements is at most
\[
(q_0+1)(q_0-1)=q_0^2-1\le q-1.
\]

Finally, in the exceptional cases, \(\overline{\Delta}\) is isomorphic to one
of \(A_4,S_4,A_5\).  Hence it has at most \(60\) elements in total.  Since \(q\ge59\), this is at most \(q+1\).

In every case, \(\overline{\Delta}\) contains at most \(q+1\) non-trivial
\(p\)-elements.  Hence \(u(\Delta)\le q+1\).  Substituting this into
\eqref{eq:level-one-chi-u}, we obtain
\[
\chi_1(\Delta)
\le
q^{-1}+\frac{q+1}{q(q+1)}
=
\frac{2}{q}.
\]
\end{proof}

\begin{lem}\label{lem:burnside-first-level-comparison}
Assume \(e\ge 2\). One has
\[
\chi_e(H)
\le
\chi_e(N_2)+\chi_1(H_1)-q^{-1}.
\]
\end{lem}

\begin{proof}
Burnside's lemma for the action of \(H\) on \(X_e\) gives
\[
|H\backslash X_e|
=
\frac{1}{|H|}
\sum_{h\in H}|X_e^h|.
\]
After division by \([G_e:H]=|G_e|/|H|\), this becomes
\begin{equation}\label{eq:burnside-H-expanded}
\chi_e(H)
=
\frac{1}{|G_e|}
\sum_{h\in H}|X_e^h|.
\end{equation}
We split the sum according to \(H=N_2\sqcup(H\setminus N_2)\).

The contribution of \(N_2\) is exactly \(\chi_e(N_2)\).  Indeed, Burnside's
lemma applied to \(N_2\) gives
\[
\frac{1}{|G_e|}
\sum_{h\in N_2}|X_e^h|
=
\frac{|N_2\backslash X_e|}{[G_e:N_2]}
=
\chi_e(N_2).
\]

It remains to estimate the contribution of \(H\setminus N_2\).  Fix
\(\bar\gamma\in H_1\setminus\{I\}\).  Since
\(H\cap\rho_{e,1}^{-1}(\bar\gamma)\subset \rho_{e,1}^{-1}(\bar\gamma)\), we
have
\[
\sum_{h\in H\cap\rho_{e,1}^{-1}(\bar\gamma)}|X_e^h|
\le
\sum_{g\in\rho_{e,1}^{-1}(\bar\gamma)}|X_e^g|.
\]

We claim that
\begin{equation}\label{eq:first-level-fiber-fixed-sum}
\sum_{g\in\rho_{e,1}^{-1}(\bar\gamma)}|X_e^g|
=
q^{3e-3}|X_1^{\bar\gamma}|.
\end{equation}
The left-hand side counts pairs \((g,x)\) with
\(g\in\rho_{e,1}^{-1}(\bar\gamma)\), \(x\in X_e\), and \(gx=x\).  Let
\(\bar x\in X_1\) be the reduction of \(x\).  If \(\bar\gamma\bar x\ne\bar x\),
then no such \(g\) exists above \(x\).  Suppose now that
\(\bar\gamma\bar x=\bar x\).

Choose \(a\in G_e\) with \(x=aU_e\).  Then
\[
\Stab_{G_e}(x)=aU_ea^{-1},
\qquad
\Stab_{G_1}(\bar x)=\rho_{e,1}(a)U_1\rho_{e,1}(a)^{-1}.
\]
The reduction map \(\Stab_{G_e}(x)\to \Stab_{G_1}(\bar x)\) is surjective, and its kernel is conjugate to \(\ker(U_e\to U_1)\), hence has cardinality \(q^{e-1}\).  Since \(\bar\gamma\in\Stab_{G_1}(\bar x)\), the set of elements
\(g\in\rho_{e,1}^{-1}(\bar\gamma)\) fixing \(x\) is exactly one fibre of this
surjective map, and therefore has cardinality \(q^{e-1}\).
Finally, each \(\bar x\in X_1\) has \(q^{2e-2}\) lifts to \(X_e\).  Thus each
\(\bar x\in X_1^{\bar\gamma}\) contributes
\(q^{2e-2}q^{e-1}=q^{3e-3}\) pairs \((g,x)\), proving
\eqref{eq:first-level-fiber-fixed-sum}.

By Lemma~\ref{lem:order-Ge}, we have $|G_e|=q^{3e-3}|G_1|$. Therefore
\[
\begin{aligned}
\frac{1}{|G_e|}
\sum_{h\in H\setminus N_2}|X_e^h|
\le
\frac{1}{|G_e|}
\sum_{\bar\gamma\in H_1\setminus\{I\}}
q^{3e-3}|X_1^{\bar\gamma}| 
=
\frac{1}{|G_1|}
\sum_{\bar\gamma\in H_1\setminus\{I\}}
|X_1^{\bar\gamma}|.
\end{aligned}
\]
By Burnside's lemma applied to the action of \(H_1\) on \(X_1\), we have
\[
|H_1\backslash X_1|
=
\frac{1}{|H_1|}
\sum_{\bar\gamma\in H_1}|X_1^{\bar\gamma}|.
\]
Dividing by \([G_1:H_1]=|G_1|/|H_1|\), we get
\[
\chi_1(H_1)
=
\frac{1}{|G_1|}
\sum_{\bar\gamma\in H_1}|X_1^{\bar\gamma}|.
\]
Separating the identity element gives
\[
\frac{1}{|G_1|}
\sum_{\bar\gamma\in H_1\setminus\{I\}}
|X_1^{\bar\gamma}|
=
\chi_1(H_1)-\frac{|X_1|}{|G_1|}
=
\chi_1(H_1)-q^{-1}.
\]
Together with \eqref{eq:burnside-H-expanded} and the preceding estimate for
the contribution of \(H\setminus N_2\), this proves \(\chi_e(H)\le\chi_e(N_2)+\chi_1(H_1)-q^{-1}\).
\end{proof}

\begin{proof}[Proof of Theorem~\ref{thm:main_local_chi_B}]
If \(e=1\), then \(H=H_1\subsetneq G_1\).  Proposition~\ref{prop:proper-subgroup-level-one-bound}
gives $\chi_1(H)\le \frac{2}{q}$. 
Assume now that \(e\ge2\).  By Lemma~\ref{lem:burnside-first-level-comparison},
\begin{equation}\label{eq:chi_H_chi_N2}
\chi_e(H)
\le
\chi_e(N_2)+\chi_1(H_1)-q^{-1}.
\end{equation}
Since \(H_1\subsetneq G_1\), Proposition~\ref{prop:proper-subgroup-level-one-bound}
gives \(\chi_1(H_1)\le 2/q\).  By Lemma~\ref{lem:trivial-chi-bound},
\[
\chi_e(N_2)
\le
\frac{1}{q-1}-\frac{q^{-e}}{q-1}
<
\frac{1}{q-1}.
\]
Then \eqref{eq:chi_H_chi_N2} gives the desired inequality. 
\end{proof}

\begin{proof}[Proof of {Remark~\ref{rem:the-H1G1-condition-unramified}}]
Assume, by contradiction, that \(H_1=\SL_2(\bF_q)\).  Choose a non-zero
nilpotent \(A\in\mathfrak{sl}_2(\bF_q)\), and lift \(I+A\in H_1\) to an
element \(h=I+B\in H\), with \(B\equiv A\pmod{\fp}\). Since \(\fp\) is
unramified, we have \(p=u_0\varpi\).

We claim that
\[
h^p=(I+B)^p\equiv I+pB\equiv I+u_0\varpi A\pmod{\varpi^2}.
\]
Indeed, since \(A^2=0\), we have \(v_\fp(B^2)\ge1\). Hence, for
\(2\le i\le p-1\),
\[
v_\fp\!\left(\binom pi B^i\right)\ge2,
\]
because \(p\vert \binom pi\). It remains to check \(B^p\). Write
\(t=\tr(B)\) and \(d=\det(B)\). Since \(B\equiv A\pmod\fp\) and \(A\) is
nilpotent, one has \(v_\fp(t),v_\fp(d)\ge1\). By Cayley--Hamilton, \(B^2=tB-dI\). It follows that \(v_\fp(B^4)\ge2\), and since \(p\ge5\), we get \(v_\fp(B^p)\ge2\). This proves the congruence.

Thus \(h^p\in N_2\) and \(\psi_2(h^p)=\bar u_0A\). Hence \(W_2\) contains a
non-zero nilpotent element. Since \(H_1=\SL_2(\bF_q)\), the space \(W_2\)
is stable under the full adjoint action of \(\SL_2(\bF_q)\). Therefore it
contains the span of the \(\SL_2(\bF_q)\)-orbit of \(A\), which is
\(\mathfrak{sl}_2(\bF_q)\). Hence \(W_2=\mathfrak{sl}_2(\bF_q)\).

Finally, \(e_0=1\) because \(\fp\) is unramified. By
Lemma~\ref{lem:propagation}, we get
\(W_e=\mathfrak{sl}_2(\bF_q)\), contradicting the exact-level condition.
Therefore \(H_1\neq\SL_2(\bF_q)\).
\end{proof}

\subsection{Parallelogram obstruction}\label{sec:parallelogram}

For $j\ge 2$, $\eta\in \Sigma_j$, and $b\in \varpi^{j-1}R_e$, write
\[
x_\eta(b):=I+bN(\eta)
=\begin{pmatrix} 1-b\widetilde\eta & b\\ -b\widetilde\eta^2 & 1+b\widetilde\eta \end{pmatrix}\in G_e;
\]
this is well-defined by Lemma~\ref{lem:borel-slope-well-defined}, and $N(\eta)^2=0$ gives $x_\eta(b)^{-1}=x_\eta(-b)$.

\begin{lem}\label{lem:quotient-formula-new}
For $b\in \varpi^{j-1}R_e$ and $\eta,\xi\in \Sigma_j$,
\[
Q_{\eta,\xi}(b)
:=x_\eta(b)\,x_\xi(b)^{-1}
=\begin{pmatrix}
1-b(\widetilde\eta-\widetilde\xi)(1+b\widetilde\xi)
& b^2(\widetilde\eta-\widetilde\xi) \\[1mm]
-b(\widetilde\eta-\widetilde\xi)(\widetilde\eta+\widetilde\xi+b\widetilde\eta\widetilde\xi)
& 1+b(\widetilde\eta-\widetilde\xi)(1+b\widetilde\eta)
\end{pmatrix}.
\]
\end{lem}
\begin{proof}
Expand $x_\eta(b)\,x_\xi(b)^{-1}=I+b(N(\eta)-N(\xi))-b^2 N(\eta)N(\xi)$ and substitute the explicit form of $N(\eta),N(\xi)$.
\end{proof}

\begin{prop}\label{prop:parallelogram-new}
Fix $j\geq 2$. Let $b\in\varpi^{j-1}R_e$, and let
$\xi,\delta,t\in\Sigma_j$. Set 
\[
r:=v_\fp(b\widetilde\delta\widetilde t)+\nu+1, \qquad R:=Q_{\xi+t+\delta,\xi+t}(b)\,Q_{\xi+\delta,\xi}(b)^{-1}.
\]
Assume that
\[
\xi,\ \xi{+}\delta,\ \xi{+}t,\
\xi{+}t{+}\delta\in\mathcal E_j(b).
\]
Assume that $j\ge \nu +2$. Then 
\begin{itemize}
\item If \(b\widetilde\delta\widetilde t\ne0\) in \(R_e\) and \(r\le e\), then
\(
\psi_r(R)\in\bF_q^*F\cap W_r.
\)
\item If \(r\le e\) and \(\bF_qF\cap W_r=\{0\}\), then necessarily
\(b\widetilde\delta\widetilde t=0\) in \(R_e\).
\end{itemize}
\end{prop}
\begin{proof}
Since all four points are in \(\mathcal E_j(b)\), the two quotients 
\[ 
Q_{\xi+\delta,\xi}(b),\qquad Q_{\xi+t+\delta,\xi+t}(b) 
\] 
belong to \(H\). Hence their ratio 
\( 
R:=Q_{\xi+t+\delta,\xi+t}(b)\,Q_{\xi+\delta,\xi}(b)^{-1} 
\)
also lies in \(H\). A direct computation gives 
\begin{align*}
R & = I+ \begin{pmatrix} -b^2\widetilde t\,\widetilde\delta\,A & b^4\widetilde t\,\widetilde\delta^2 \\[1mm] -b\,\widetilde t\,\widetilde\delta\,C & b^2\widetilde t\,\widetilde\delta\,D \end{pmatrix}, \quad \text{where}\, 
\\
A & =1+b\widetilde\delta+b^2\widetilde\delta(\widetilde\xi+\widetilde\delta), \quad
D=1+b\widetilde\delta+b^2\widetilde\delta(\widetilde\xi+\widetilde t+\widetilde\delta), 
\\
C & = 2+b(3\widetilde\delta+\widetilde t+2\widetilde\xi) +b^2\widetilde\delta(2\widetilde\delta+\widetilde t+2\widetilde\xi) +b^3\widetilde\delta \bigl( \widetilde\delta^2+\widetilde\delta\widetilde t+2\widetilde\delta\widetilde\xi +\widetilde t\widetilde\xi+\widetilde\xi^2 \bigr).
\end{align*}
Since 
\(b,\widetilde\xi,\widetilde\delta,\widetilde t\in \varpi R_e \), we have 
\[ 
A,D\in 1+\varpi R_e, \qquad C\in 2+\varpi R_e. 
\] 
Assume that \(b\widetilde\delta\widetilde t\ne0\) in \(R_e\) and that
\(r\le e\). Then
\begin{align*}
& v_\fp(R_{21}) =v_\fp(b\widetilde\delta\widetilde t)+\nu
=r-1,
\\
& v_\fp(R_{11}-1),v_\fp(R_{22}-1)
\ge v_\fp(b\widetilde\delta\widetilde t)
+v_\fp(b)\ge (r-\nu-1)+j-1\ge r,
\\
& v_\fp(R_{12})
 \ge v_\fp(b\widetilde\delta\widetilde t)
+3v_\fp(b)+v_\fp(\widetilde\delta)
\ge r+1.
\end{align*}
Hence $R\in N_r\setminus N_{r+1}$
and $\psi_r(R)$ is a nonzero multiple
of $F$.
Since $R\in H$, we deduce that $\psi_r(R)\in W_r$.
\end{proof}

\subsection{Semisimple obstruction}\label{sec:semisimple-obstruction}

Define
\[
V_k
:=\{a\in\bF_q\mid a D\in W_k\}.
\]
This is an $\bF_p$-subspace
of $\bF_q$,
with $\dim_{\bF_p} V_k\le f$
and $\dim_{\bF_p} V_k=f$
if and only if $\bF_qD\subset W_k$.


The following first-variation identity explains why the diagonal
line \(\bF_qD\) appears in the next proposition.  For lifts
\(\widetilde\eta,\widetilde\delta\in R_e\), one has
\[
N(\widetilde\eta+\widetilde\delta)-N(\widetilde\eta)
=
-\widetilde\delta D
-
(2\widetilde\eta\widetilde\delta+\widetilde\delta^2)F,
\quad
D=\begin{pmatrix}1&0\\0&-1\end{pmatrix},
\;
F=\begin{pmatrix}0&0\\1&0\end{pmatrix}.
\]

\begin{prop}
\label{prop:membership-transversality-new}
Let $j\ge 2$, $\eta\in\Sigma_j$,
$b\in R_e$ with
$v_\fp(b)=j-1+s$, $0\le s\le e-j$.
For $k\in\{j{+}s{+}1,\ldots,e\}$,
suppose
$\rho_{e,k-1}(I+bN(\eta))\in H_{k-1}$.
Then the condition
$\rho_{e,k}(I+bN(\eta'))\in H_k$
for $\eta'\in\Sigma_j$ satisfying
$\eta'\equiv\eta\pmod{\varpi^{k-j-s}}$
depends only on the class of $\eta'$
modulo $\varpi^{k-j-s+1}$,
and the number of such classes is either \(0\) or
\(p^{\dim_{\bF_p} V_k}\). 
\end{prop}

\begin{proof}
Write $w:=I+bN(\eta)$. For $\eta'\in\Sigma_j$ satisfying $\eta'\equiv\eta\pmod{\varpi^{k-j-s}}$, write $w':=I+bN(\eta')$.

\emph{Step 1: $\rho_{e,k}(w')$ depends only on $\eta'\bmod\varpi^{k-j-s+1}$.}

Let $\eta''\in\Sigma_j$ satisfy $\eta''\equiv\eta'\pmod{\varpi^{k-j-s+1}}$, and choose lifts $\widetilde\eta',\widetilde\eta''\in\varpi R_e$ with $v_\fp(\widetilde\eta''-\widetilde\eta')\ge k-j-s+1$. Direct computation from the definition of $N(\eta)$ gives the entries of $N(\eta'')-N(\eta')$:
\begin{align*}
\bigl(N(\eta'')-N(\eta')\bigr)_{12}
&=0,\quad
\bigl(N(\eta'')-N(\eta')\bigr)_{21}
=-(\widetilde\eta''-\widetilde\eta')(\widetilde\eta''+\widetilde\eta'),
\\
\bigl(N(\eta'')-N(\eta')\bigr)_{11}
&=-\bigl(N(\eta'')-N(\eta')\bigr)_{22}
=-(\widetilde\eta''-\widetilde\eta').
\end{align*}
Multiplying by $b$ and using $v_\fp(b)=j-1+s$ and $v_\fp(\widetilde\eta''+\widetilde\eta')\ge 1$, we obtain
\[
v_\fp\bigl(b(\widetilde\eta''-\widetilde\eta')\bigr)
\ge(j-1+s)+(k-j-s+1)=k,
\]
for the diagonal entries, and
\[
v_\fp\bigl(b(\widetilde\eta''-\widetilde\eta')(\widetilde\eta''+\widetilde\eta')\bigr)\ge k+1
\]
for the $(2,1)$-entry. Hence $bN(\eta'')\equiv bN(\eta')\pmod{\varpi^k}$, and $\rho_{e,k}(w'')=\rho_{e,k}(w')$.

\emph{Step 2: Reduction to a parameter $t\in\bF_q$.}

Let $\eta'\in\Sigma_j$ with $\eta'\equiv\eta\pmod{\varpi^{k-j-s}}$. Choose lifts $\widetilde\eta,\widetilde\eta'\in\varpi R_e$ of $\eta,\eta'$ and set $\delta:=\widetilde\eta'-\widetilde\eta$, so that $v_\fp(\delta)\ge k-j-s$. We have
\begin{align*}
\bigl(bN(\eta')-bN(\eta)\bigr)_{12}
&=0,\quad
\bigl(bN(\eta')-bN(\eta)\bigr)_{21}
=-b\delta(\widetilde\eta'+\widetilde\eta),
\\
\bigl(bN(\eta')-bN(\eta)\bigr)_{11}
&=-\bigl(bN(\eta')-bN(\eta)\bigr)_{22}
=-b\delta.
\end{align*}
The diagonal entries have $\fp$-valuation $\ge(j-1+s)+(k-j-s)=k-1$, and the $(2,1)$-entry has valuation $\ge k$ since $v_\fp(\widetilde\eta'+\widetilde\eta)\ge 1$. Therefore $bN(\eta')\equiv bN(\eta)\pmod{\varpi^{k-1}}$, whence $\rho_{e,k-1}(w')=\rho_{e,k-1}(w)$.

Modulo $\varpi^k$, only the diagonal entries of $bN(\eta')-bN(\eta)$ contribute. Write $b=\varpi^{j-1+s}u$ with $u\in R_e^*$, and set $\bar b:=u\bmod\varpi\in\bF_q^*$. Let $t\in\bF_q$ denote the image of $\eta'-\eta$ under the projection
\[
\varpi^{k-j-s}R_{e-j+1}\;\longrightarrow\;\varpi^{k-j-s}R_{e-j+1}\big/\varpi^{k-j-s+1}R_{e-j+1}\;\cong\;\bF_q.
\]
Then $\delta\equiv\varpi^{k-j-s}\hat t\pmod{\varpi^{k-j-s+1}}$ for some $\hat t\in R_e$ with $\hat t\bmod\varpi=t$, and
\[
b\delta=\varpi^{j-1+s}u\cdot\delta\equiv\varpi^{k-1}\bar b\,t\pmod{\varpi^k}.
\]
Combining with the entry computation yields
\begin{equation}\label{eq:Vk-bound-Mequation-new}
bN(\eta')\equiv bN(\eta)-\varpi^{k-1}\bar b\,t\begin{pmatrix}1&0\\0&-1\end{pmatrix}\pmod{\varpi^k}.
\end{equation}

\emph{Step 3: The constraint on $t$.}

Since $N(\eta)^2=0$, we have $w^{-1}=I-bN(\eta)$ in $G_e$. Reducing modulo $\varpi^k$,
\[
\rho_{e,k}(w)^{-1}\rho_{e,k}(w')=\bigl(I-bN(\eta)\bigr)\bigl(I+bN(\eta')\bigr)\bmod\varpi^k.
\]
We show that the cross term $bN(\eta)\cdot bN(\eta')$ vanishes modulo $\varpi^k$. Since $N(\eta)^2=0$, we have
\begin{align*}
bN(\eta)\cdot bN(\eta')
&=bN(\eta)\cdot bN(\eta)+bN(\eta)\bigl(bN(\eta')-bN(\eta)\bigr)\\
&=bN(\eta)\bigl(bN(\eta')-bN(\eta)\bigr).
\end{align*}
The right-hand side has $\fp$-valuation at least $(j-1+s)+(k-1)\ge k$, because $v_\fp(bN(\eta))\ge j-1+s\ge 1$ and $v_\fp(bN(\eta')-bN(\eta))\ge k-1$ by Step~2. Therefore, using~\eqref{eq:Vk-bound-Mequation-new},
\begin{equation}\label{eq:Vk-bound-Ak-new}
\rho_{e,k}(w)^{-1}\rho_{e,k}(w')\equiv I-\varpi^{k-1}\bar b\,t\begin{pmatrix}1&0\\0&-1\end{pmatrix}\pmod{\varpi^k},
\end{equation}
which lies in $K_k=\ker(\rho_{k,k-1})$.

Since $\rho_{e,k-1}(w')=\rho_{e,k-1}(w)\in H_{k-1}$ and $\rho_{e,k}(w)^{-1}\rho_{e,k}(w')\in K_k$ by~\eqref{eq:Vk-bound-Ak-new}, the condition $\rho_{e,k}(w')\in H_k$ is equivalent to
\[
\rho_{e,k}(w)^{-1}\rho_{e,k}(w')\in\rho_{e,k}(w)^{-1}H_k\cap K_k.
\]
The intersection $\rho_{e,k}(w)^{-1}H_k\cap K_k$ is non-empty, since the hypothesis $\rho_{e,k-1}(w)\in H_{k-1}$ provides some $h_0\in H_k$ with $\rho_{k,k-1}(h_0)=\rho_{e,k-1}(w)$, and then $\rho_{e,k}(w)^{-1}h_0$ lies in this intersection. Hence the intersection is a coset of $H_k\cap K_k$ in $K_k$. The isomorphism $K_k\cong\mathfrak{sl}_2(\bF_q)$, $I+\varpi^{k-1}A\mapsto A$, sends $H_k\cap K_k$ to $W_k$ and, by~\eqref{eq:Vk-bound-Ak-new}, sends $\rho_{e,k}(w)^{-1}\rho_{e,k}(w')$ to $-\bar b\,t\,D$. Denote by $W_k+A_0$ the image of $\rho_{e,k}(w)^{-1}H_k\cap K_k$ under this isomorphism, which is a coset of $W_k$ in $\mathfrak{sl}_2(\bF_q)$ depending on $\eta$. The condition on $t$ becomes
\[
-\bar b\,t\,D\in(W_k+A_0)\cap\bF_qD.
\]
The right-hand side is a coset of $W_k\cap\bF_qD=V_k\cdot D$ in $\bF_qD$. If this intersection is non-empty, then \(t\) runs through a coset of
\(\bar b^{-1}V_k\) in \(\bF_q\), of cardinality
\(p^{\dim_{\bF_p} V_k}\); otherwise there is no admissible class.
\end{proof}

\subsection{Double counting}\label{sec:double-counting}

Define
\begin{align*}
\jmath_1 & := \min\{k\le e\mid W_k\cap \bF_q^*E\ne \emptyset\}, \; \jmath_1:=e+1\ \text{if the min does not exist},
\\
\jmath_2 & := \min\{k\le e\mid W_k\cap \bF_q^*F\ne \emptyset\}, \; \jmath_2:=e+1 \ \text{if the min does not exist},\\
\jmath & :=\max(\jmath_1,\jmath_2),
\quad
\Pi :=\{k\in [\jmath_2,e]\mid W_k\cap \bF_q^*F= \emptyset\}.
\end{align*}
Recall that the notation \(\mathcal E_j(b)\) is defined in \S\ref{sec:telescopic}.
For fixed $b\in\varpi^{j-1}R_e$ and $\delta \in \Sigma_j$, define 
\[ 
\mathcal P_{j,b}(\delta):= \{\xi\in \Sigma_j\mid \xi,\xi+\delta\in \mathcal E_j(b)\}. 
\]

\begin{prop}\label{prop:difference-fiber-new}
Let $2+\nu\le j\le e$, $m:=e-j$. Let $b\in R_e$ with
$v_\fp(b)=j-1+s$, $0\le s\le m$,
and let $\delta\in\Sigma_j$ with
$v_\fp(\widetilde\delta)=d$, $1\le d\le m$. Set
\begin{align*}
\gamma & :=\nu+e+1-\jmath_2, 
\\
c_*(j,s,d) & :=
\#\bigl\{k\in\bZ \mid
\max(j{+}s{+}1,\jmath_2)\le k\le e,\;\\
& \dim_{\bF_p}V_k\le f-1,\;
k+d+\nu\notin\Pi\bigr\}
\\
c^*(j,s,d) & :=\lvert \Pi\cap [\,j+s+d+\nu+1,e\,]\rvert
\end{align*}
Then, for any \(\xi,\xi'\in \mathcal P_{j,b}(\delta)\), denoting
\(\tau:=\xi'-\xi\), one has
\[
v_\fp(\tau)\ge \jmath_2-j-s-d-\nu,
\]
and, if \(\tau\neq0\), then
\[
v_\fp(\tau)+j+s+d+\nu\notin\Pi.
\]
Moreover,
\[
|\mathcal P_{j,b}(\delta)|
\le
q^{\min(m,\,s+d+\gamma)}q^{-c^*(j,s,d)}\,p^{-c_*(j,s,d)}.
\]
\end{prop}
 
\begin{proof}

Assume $\xi,\xi'\in\mathcal P_{j,b}(\delta)$. Set $t:=\xi'-\xi$. Then $\xi,\xi+\delta,\xi+t,\xi+t+\delta \in \mathcal{E}_j(b)$. By Proposition~\ref{prop:parallelogram-new} we obtain that
\begin{equation}\label{eq:difference-bound-formula-1}
b\widetilde\delta \widetilde t=0\in R_e\quad \text{when}\ W_{r(t)}\cap \bF_q^* F=\emptyset,   
\end{equation}
where $\widetilde\delta, \widetilde t\in R_e$ are lifts of $\delta,t$ and
\begin{align}\label{eq:difference-bound-formula-2}
r(t): & =v_\fp(b\widetilde\delta \widetilde t)+\nu+1= v_\fp(b)+v_\fp(\delta)+v_\fp(t)+\nu+1
\\
& = j+s+d+\nu+v_\fp(t).
\end{align}
Recall that $W_{k}\cap \bF_q^* F=\emptyset$ for all $k<\jmath_2$. Therefore, $t$ has to satisfy (also true when $t=0$)
\[
v_\fp(t) \ge \jmath_2 -j-s-d-\nu. 
\]
Moreover, if \(t\neq0\), then \(r(t)\) cannot lie in \(\Pi\), by the
definition of \(\Pi\). Thus, for every pair
\(\xi,\xi'\in\mathcal P_{j,b}(\delta)\), with \(t=\xi'-\xi\), one has
\[
v_\fp(t)\ge \jmath_2-j-s-d-\nu,
\]
and, if \(t\neq0\), then
\[
v_\fp(t)+j+s+d+\nu\notin\Pi.
\]

Before we continue the proof let us simplify a little bit the notations. 
We set \(g:=\jmath_2-j-s-d-\nu,g_0:=\max(1,g)\). 
Fix $\xi_0\in\mathcal P_{j,b}(\delta)$ and set
\[
P:=\mathcal P_{j,b}(\delta)-\xi_0\subset \Sigma_j.
\]
Then $0\in P$, and the previous argument shows that for all
$\tau,\tau'\in P$ one has
\[
v_\fp(\tau-\tau')\ge g,
\]
and if $\tau\neq\tau'$ then
\begin{equation}\label{eq:double-counting-formula-1}
v_\fp(\tau-\tau')+j+s+d+\nu\notin\Pi.
\end{equation}
In particular, any two elements of $P$ are congruent modulo $\varpi^{g_0}$.

If $g_0\ge m+1$, then \(P=\{0\}\), hence
$|\mathcal P_{j,b}(\delta)|=1$ and there is nothing more to prove.
Assume now that $g_0\le m$.

For each integer $i\in\{g_0-1,\dots,m\}$, let $P_i$ be the image of $P$ in $\Sigma_j/\varpi^{i+1}R_{m+1}$. Then $|P_{g_0-1}|=1, |P_m|=|P|=|\mathcal P_{j,b}(\delta)|$. For each $i=g_0,\dots,m$, let $\pi_i:P_i\to P_{i-1}$ be the natural projection.

We claim that the fibers of $\pi_i$ satisfy the following bounds.

\smallskip\noindent
\emph{Case 1: if $j+s+d+\nu+i\in\Pi$, then every fiber of $\pi_i$ has
cardinality at most $1$.}

Indeed, if two distinct classes in $P_i$ had the same image in $P_{i-1}$, we could choose representatives $\tau,\tau'\in P$ such that
\[
\tau\equiv\tau' \pmod{\varpi^i},
\qquad
\tau\not\equiv\tau' \pmod{\varpi^{i+1}}.
\]
Then $v_\fp(\tau-\tau')=i$, hence
\[
v_\fp(\tau-\tau')+j+s+d+\nu=j+s+d+\nu+i\in\Pi,
\]
contrary to \eqref{eq:double-counting-formula-1}.

\smallskip\noindent
\emph{Case 2.} We claim the following: if
\begin{align*}
& \max(j{+}s{+}1,\jmath_2)\le j+s+i\le e,\\
& \dim_{\bF_p}V_{j+s+i}\le f-1,\quad j+s+i+d+\nu\notin\Pi,   
\end{align*}
then every fiber of $\pi_i$ has cardinality at most
\[
p^{\dim_{\bF_p}V_{j+s+i}}\le p^{f-1}=q/p.
\]

Let us prove the claim now. Fix a class \(\bar\tau_0\in P_{i-1}\) and choose
\(\tau_0\in P\) mapping to it.
Set
\[
\eta_0:=\xi_0+\tau_0\in \mathcal P_{j,b}(\delta)\subset \mathcal E_j(b).
\]
If $\tau\in P$ maps to the same class $\bar\tau_0$, then
$\tau\equiv\tau_0\pmod{\varpi^i}$, so
\[
\xi_0+\tau\equiv \eta_0 \pmod{\varpi^i},
\]
and also $\xi_0+\tau\in\mathcal E_j(b)$.
Since $\eta_0\in\mathcal E_j(b)$, we have
\[
\rho_{e,j+s+i-1}(I+bN(\eta_0))\in H_{j+s+i-1}.
\]
Therefore we can apply Proposition~\ref{prop:membership-transversality-new} at
level $j+s+i$ and show that, once the class modulo
$\varpi^{i}$ is fixed, there are at most
$p^{\dim_{\bF_p}V_{j+s+i}}$ possible classes modulo
$\varpi^{i+1}$ inside $\mathcal E_j(b)$.
Hence the fiber of $\pi_i$ above $\bar\tau$ has size at most
$p^{\dim_{\bF_p}V_{j+s+i}}$.

\smallskip\noindent
\emph{Case 3:} at all remaining levels $i$, we use the trivial bound
\[
|\pi_i^{-1}(\bar\tau)|\le q.
\]
Multiplying the fiber bounds for $i=g_0,\dots,m$, we obtain
\[
|\mathcal P_{j,b}(\delta)|
=
|P_m|
\le
q^{\,m+1-g_0-\lvert \Pi\cap [\,j+s+d+\nu+1,e\,]\rvert}\,p^{-c_*(j,s,d)}.
\]
Since
\[
m+1-g_0=\min(m,\,s+d+\gamma),
\]
this may also be written as
\[
|\mathcal P_{j,b}(\delta)|
\le
q^{\,\min(m,\,s+d+\gamma)-\lvert \Pi\cap [\,j+s+d+\nu+1,e\,]\rvert}\,p^{-c_*(j,s,d)}.
\]
\end{proof}

\begin{cor}\label{cor:square-root-F}
Let $j\ge\nu+2$. 
For every $b\ne 0$ with $v_\fp(b)=j-1+s$,
$0\le s\le m$,
\[
|\mathcal E_j(b)|
\le 1+\sqrt{m(q{-}1)}
q^{(m+s+\gamma-\lvert \Pi\rvert)/2}.
\]
\end{cor}

\begin{proof}
For $b\in \varpi^{j-1}R_e$, recall that
$\mathcal E_j(b)=\{\eta\in\Sigma_j\mid 
I+bN(\eta)\in H\}\subset\Sigma_j$
and
$\mathcal P_{j,b}(\delta)
=\{\xi\in\Sigma_j\mid 
\xi,\xi{+}\delta\in\mathcal E_j(b)\}$.
Since $\Sigma_j=\varpi R_{e-j+1}$
is an additive group,
the substitution
$(\xi,\eta)\mapsto(\xi,\delta:=\eta-\xi)$
is a bijection
$\Sigma_j^2\to\Sigma_j^2$, giving
\begin{align}
|\mathcal E_j(b)|^2
&=\bigl|\{(\xi,\eta)\in
\mathcal E_j(b)^2\}\bigr|
\notag
\\
&=\bigl|\{(\xi,\delta)\in\Sigma_j^2\mid 
\xi\in\mathcal E_j(b),\;
\xi{+}\delta\in\mathcal E_j(b)\}\bigr|=\sum_{\delta\in\Sigma_j}
|\mathcal P_{j,b}(\delta)|.\label{eq:double-summation-1}
\end{align}
The $\delta=0$ term contributes
$|\mathcal P_{j,b}(0)|=|\mathcal E_j(b)|$.
For $\delta\ne 0$ with
$v_\fp(\delta)=d$,
$1\le d\le m$,
there are $(q{-}1)q^{m-d}$
such elements.
By Proposition~\ref{prop:difference-fiber-new}, the contribution to the sum in \eqref{eq:double-summation-1} from such $\delta$ is
\begin{equation}\label{eq:double-summation-2}
\sum_{\delta\in\Sigma_j, v_\fp(\delta)=d}
|\mathcal P_{j,b}(\delta)|\le (q{-}1)q^{m-d}\cdot q^{\min(m,s+d+\gamma)}
q^{-c^*(j,s,d)}
 p^{-c_*(j,s,d)}.  
\end{equation}
Note that 
\begin{equation*}
\min(m,s+d+\gamma)+m-d  \le m+s+\gamma-\max(0,d-m+s+\gamma).
\end{equation*}
So \eqref{eq:double-summation-2} becomes
\begin{align*}
\sum_{\delta\in\Sigma_j, v_\fp(\delta)=d}
|\mathcal P_{j,b}(\delta)|\le (q-1)q^{m+s+\gamma}q^{-c^*(j,s,d)-\max(0,d-m+s+\gamma)}
 p^{-c_*(j,s,d)}.
\end{align*}

Summing over $d=1,\ldots,m$,
\[
|\mathcal E_j(b)|^2
\;\le\;
|\mathcal E_j(b)|
+(q{-}1)\,q^{m+s+\gamma}\Big(\sum_{d=1}^m
q^{-c^*(j,s,d)-\max(0,d-m+s+\gamma)}
 p^{-c_*(j,s,d)}\Big).
\]
The quadratic inequality
$X^2\le X+A$ gives
$X\le 1+\sqrt{A}$,
hence
\begin{equation}\label{eq:double-summation-3}
|\mathcal E_j(b)|
\le 1+\sqrt{q{-}1}
q^{(m+s+\gamma)/2}\Big(\sum_{d=1}^m
q^{-c^*(j,s,d)-\max(0,d-m+s+\gamma)}
 p^{-c_*(j,s,d)}\Big)^{\frac{1}{2}}.
\end{equation}

We claim that, for every \(1\le d\le m\),
\begin{equation}\label{eq:double-summation-4}
c^*(j,s,d)+\max(0,d-m+s+\gamma)\ge |\Pi|.
\end{equation}
Indeed, set
\[
h:=d-m+s+\gamma.
\]
Since \(m=e-j\) and \(\gamma=e+\nu+1-\jmath_2\), we have
\[
j+s+d+\nu+1=\jmath_2+h.
\]
If \(h\le0\), then
\[
c^*(j,s,d)
=
|\Pi\cap[\jmath_2+h,e]|
=
|\Pi|,
\]
because \(\Pi\subset[\jmath_2,e]\), and \(\max(0,h)=0\). If \(h>0\), then
\[
c^*(j,s,d)+h
=
|\Pi\cap[\jmath_2+h,e]|+h
\ge |\Pi|,
\]
because at most \(h\) integers of \(\Pi\subset[\jmath_2,e]\) can lie below
\(\jmath_2+h\). This proves \eqref{eq:double-summation-4}.
Combining \eqref{eq:double-summation-3} and
\eqref{eq:double-summation-4} gives
\begin{equation*}
|\mathcal E_j(b)|
\le 1+\sqrt{q{-}1}
q^{(m+s+\gamma-\lvert \Pi\rvert)/2}\Big(\sum_{d=1}^m
 p^{-c_*(j,s,d)}\Big)^{\frac{1}{2}}.
\end{equation*}
The result follows because 
$\sum_{d=1}^m
 p^{-c_*(j,s,d)}\leq m$.
\end{proof}

For $s\ge 0, j\ge 2$ set
\[
\hat c(j,s) :=
\#\bigl\{k\in\bZ \mid
\max(j{+}s{+}1,\jmath_2)\le k\le e,\;\\
\dim_{\bF_p}V_k\le f-1\bigr\}.
\]

\begin{cor}\label{cor:square-root-D}
Let $j\ge\nu+2$. 
For every $b\ne 0$ with $v_\fp(b)=j-1+s$,
$0\le s\le m$,
\[
|\mathcal E_j(b)|
\le 1+\sqrt{m(q{-}1)}
q^{(m+s+\gamma)/2}
p^{-\hat c (j,s)/2}.
\]
\end{cor}
\begin{proof}
Recall that
\begin{align*}
c_*(j,s,d) & :=
\#\bigl\{k\in\bZ \mid
\max(j{+}s{+}1,\jmath_2)\le k\le e,\;\\
& \dim_{\bF_p}V_k\le f-1,\;
k+d+\nu\notin\Pi\bigr\}
\\
c^*(j,s,d) & :=\lvert \Pi\cap [\,j+s+d+\nu+1,e\,]\rvert.
\end{align*}
When $k\ge j+s+1$, we have 
\[
k+d+\nu\ge j+s+d+\nu+1,
\]
so
\(
c_*(j,s,d)+c^*(j,s,d)\ge \hat c(j,s)
\).
Combining with \eqref{eq:double-summation-3} we get
\begin{align*}
|\mathcal E_j(b)|
 & \le 1+\sqrt{q{-}1}
q^{(m+s+\gamma)/2}\Big(\sum_{d=1}^m
q^{-c^*(j,s,d)-\max(0,d-m+s+\gamma)}
 p^{-c_*(j,s,d)}\Big)^{\frac{1}{2}}
 \\
 & \le 1+\sqrt{q{-}1}
q^{(m+s+\gamma)/2}\Big(\sum_{d=1}^m
p^{-c^*(j,s,d)-c_*(j,s,d)}\Big)^{\frac{1}{2}}\\
& \le 1+\sqrt{m(q{-}1)}
q^{(m+s+\gamma)/2}
p^{-\hat c (j,s)/2}.
\end{align*}

\end{proof}

\subsection{Semisimple dimension bounds}\label{sec:semisimple-bounds}

Recall that $\jmath_1=\min\{j\mid W_j\cap \bF_q E\neq \emptyset\}$,
$\jmath_2=\min\{j\mid W_j\cap \bF_q F\neq \emptyset\}$ and 
$\jmath=\max(\jmath_1,\jmath_2)$ are introduced in Section~\ref{sec:double-counting}.

\begin{lem}
\label{lem:cross-level-Vk}
Assume \(p\ne 2\). Assume that \(K\) is an integer such that
\[
e-\jmath>K>2\max(\jmath,e_0+1)+2e_0.
\]
Then
\[
c_0(K):=\#\bigl\{k\in[\jmath{+}1,K]\mid
\dim_{\bF_p} V_k\le f{-}1
\bigr\}
\;\ge\;
\frac{K-2\max(\jmath,e_0+1)}{2e_0}.
\]
\end{lem}

\begin{proof}
Put \(M:=\max(\jmath,e_0+1)\).  By Lemma~\ref{lem:propagation},
once \(\bF_qD\subset W_{k_0}\) for some \(k_0\ge e_0+2\), then for
every \(k\ge k_0\) with \(k\equiv k_0\pmod {e_0}\), one has
\(\bF_qD\subset W_k\).

Suppose for contradiction that
\begin{equation}\label{eq:contradiction-hypo-c0K}
c_0(K)<\frac{K-2M}{2e_0}.
\end{equation}
For each residue class \(r\in\{0,\ldots,e_0{-}1\}\), let \(k_r^*\) be
the smallest \(k\in[M+1,K]\) such that
\[
k\equiv r\pmod {e_0},
\qquad
\dim_{\bF_p}V_k=f,
\]
if such a \(k\) exists.

We first claim that \(k_r^*\) exists for every residue class \(r\in \bZ/e_0 \bZ\).
Indeed, if no such \(k\) existed in one residue class, then every integer in \([M+1,K]\) belonging to that residue class would contribute to \(c_0(K)\). The number of such integers is at least
\[
\frac{K-M}{e_0}-1.
\]
Since \(K>2M+2e_0\), this is strictly larger than
\(
(K-2M)/(2e_0)
\),
contradicting the assumption \eqref{eq:contradiction-hypo-c0K}.

Set
\(
k^*:=\max_{0\le r<e_0}k_r^*.
\)
Since \(k_r^*\ge M+1\ge e_0+2\), Lemma~\ref{lem:propagation} gives
\begin{equation}\label{eq:three-inclusions-formula-1}
\bF_qD\subset W_k
\qquad\text{for every }k\ge k^*.
\end{equation}

We now produce \(E\) on all sufficiently large levels \(W_k\). By definition of
\(\jmath_1\), there exists \(\mu\in\bF_q^*\) such that
\(\mu E\in W_{\jmath_1}\). For each residue class \(r\in \bZ/e_0 \bZ\), Lemma~\ref{lem:lie-bracket} gives
\[
\bF_qE=[\bF_qD,\mu E]\subset W_{k_r^*+\jmath_1-1}.
\]
As \(r\) varies, the levels \(k_r^*+\jmath_1-1\) run through all residue
classes modulo \(e_0\). Moreover
\[
k_r^*+\jmath_1-1\le K+\jmath-1<e.
\]
Therefore, by Lemma~\ref{lem:propagation},
\begin{equation}\label{eq:three-inclusions-formula-2}
\bF_qE\subset W_k
\qquad\text{for every } k\ge k^*+\jmath_1-1.
\end{equation}
The same argument, using \(\jmath_2\), gives
\begin{equation}\label{eq:three-inclusions-formula-3}
\bF_qF\subset W_k
\qquad\text{for every } k\ge k^*+\jmath_2-1.
\end{equation}

Combining \eqref{eq:three-inclusions-formula-1}, \eqref{eq:three-inclusions-formula-2}, \eqref{eq:three-inclusions-formula-3}, we get
\[
W_k\supset
\bF_qD+\bF_qE+\bF_qF
=
\mathfrak{sl}_2(\bF_q)
\qquad
\text{for every }k\ge k^*+\jmath-1.
\]
Since \(k^*\le K\) and \(K+\jmath-1<e\), this applies in particular to
\(k=e\). Hence $W_e=\mathfrak{sl}_2(\bF_q)$, contradicting the exact level condition.
\end{proof}

Define
\begin{align*}
\jmath_3 & :=\min\{j\mid \forall k\in [j,e],  W_k\cap \bF_q^* F\neq \emptyset\}, \\
\jmath_3 & := e+1 \; \text{if the minimum does not exist},\\
\jmath^* & := \max ({\jmath_1,\jmath_3}),
\qquad
J  :=\max(\jmath^*,2\nu+1).
\end{align*}

\begin{lem}
\label{lem:cross-level-Vk-two}
Assume \(p=2\). Assume that \(K\) is an integer such that
\[
2J+2\nu<K< \frac{1}{4}(e+4-J-10\nu).
\]
Then
\[
c_0^*(K):=\#\bigl\{k\in[J+1,K]\mid
\dim_{\bF_p} V_k\le f{-}1
\bigr\}
\;\ge\;\frac{K-2J-2\nu}{2\nu}.
\]
\end{lem}

\begin{proof}
Recall that \(2=u_0\varpi^\nu\). For each residue class
\(r\in\{0,\ldots,\nu{-}1\}\), let
\[
B_r:=
\#\{k\in[J+1,K]\mid k\equiv r\pmod{\nu},
\dim_{\bF_p}V_k\le f-1\}.
\]
Then
\(
c_0^*(K)=\sum_{r=0}^{\nu-1}B_r.
\)

Suppose for contradiction that
\begin{equation}\label{eq:contradiction-hypo-c0K*}
c_0^*(K)=\sum_rB_r<\frac{K-2J-2\nu}{2\nu}.
\end{equation}
For each residue class \(r\), let \(k_r^*\) be the smallest
\(k\in[J+1,K]\) such that
\[
k\equiv r\pmod{\nu},
\qquad
\dim_{\bF_p}V_k=f.
\]
Such a \(k_r^*\) exists. Indeed, if no such integer existed in one
residue class, then every integer in \([J+1,K]\) belonging to that
residue class would contribute to \(c_0^*(K)\). The number of such
integers is at least
\[
\frac{K-J}{\nu}-1>
\frac{K-2J-2\nu}{2\nu}.
\]
This contradicts \eqref{eq:contradiction-hypo-c0K*}.

Set \(k^*:=\max_{0\le r<\nu}k_r^*\). Then, by Lemma~\ref{lem:propagation},
\[
\bF_qD\subset W_k
\qquad\text{for every }k\in[k^*,e].
\]

\medskip\noindent
\emph{Step~1: Upper bound on \(k^*\).}
For each residue class \(r\), the definition of \(k_r^*\) gives
\[
B_r=\left\lfloor\frac{k_r^*-J-1}{\nu}\right\rfloor.
\]
Thus $k_r^*\le J+\nu+\nu B_r$. Since
\[
B_r\le \sum_sB_s<\frac{K-2J-2\nu}{2\nu},
\]
we get $k_r^*<\frac K2$, thus
\begin{equation}\label{eq:kstar-Vkbound-inequality-two}
k^*<\frac K2.
\end{equation}

\medskip\noindent
\emph{Step~2: \(\bF_qE\) on every sufficiently large level.}
By hypothesis, \(W_{\jmath_1}\cap\bF_q^*E\neq\emptyset\). Hence, by
Lemma~\ref{lem:propagation}, for every \(i\ge\jmath_1\) with
\(i\equiv\jmath_1\pmod\nu\), one has
\[
W_i\cap\bF_q^*E\neq\emptyset.
\]
Set
\[
k_E:=\min\{k\ge k^*\mid k\equiv\jmath_1\pmod\nu\}.
\]
Then
\[
k_E\ge k^*\ge J+1\ge 2\nu+2,
\qquad
k_E\le k^*+\nu-1.
\]
Moreover \(\bF_qD\subset W_{k_E}\) and
\(W_{k_E}\cap\bF_q^*E\neq\emptyset\). Therefore there exists
\(\mu\in\bF_q^*\) such that, for every \(\alpha\in\bF_q\), we can choose
\(x_\alpha\in H\) satisfying
\[
x_\alpha\in N_{k_E},
\qquad
\psi_{k_E}(x_\alpha)=\alpha D+\mu E.
\]

By definition of \(\jmath_3\), for every integer \(l\) with
\(e-4k_E-3\nu+4\ge l\ge J+1\), we can choose \(y_l\in H\) satisfying
\[
y_l\in N_l,
\qquad
\psi_l(y_l)=\beta_lF,
\qquad
\beta_l\in\bF_q^*.
\]
Indeed \(J+1\ge\jmath^*+1\ge\jmath_3\). Also \(l\ge J+1\ge2\nu+2\).
Thus Lemma~\ref{lem:two-word-three} applies to \(x_\alpha\) and \(y_l\).
It gives
\begin{equation}\label{eq:Vkbound-two-key-formula-E}
\widehat{(x_\alpha,y_l)}
\in N_{4k_E+l+3\nu-4},
\quad
\psi_{4k_E+l+3\nu-4}\bigl(\widehat{(x_\alpha,y_l)}\bigr)
=
\bar u_0^3\alpha^2\mu\beta_l(\alpha D+\mu E).
\end{equation}

The space \(W_{4k_E+l+3\nu-4}\) contains \(\bF_qD\), because
\(4k_E+l+3\nu-4\ge k^*\). Since \(p=2\), the map
\(\alpha\mapsto\alpha^2\) is an automorphism of \(\bF_q\). Varying
\(\alpha\) in \eqref{eq:Vkbound-two-key-formula-E}, we obtain
\begin{equation}\label{eq:p2-all-basis-formula-1}
\bF_qE\subset W_{4k_E+l+3\nu-4}.
\end{equation}
Varying \(l\), we get
\begin{equation}\label{eq:p2-all-basis-formula-2}
\bF_qE\subset W_r
\qquad\text{for every }
r\ge 4k_E+J+3\nu-3.
\end{equation}

\medskip\noindent
\emph{Step~3: \(\bF_qF\) on every sufficiently large level.}

Since \(k^*\ge J+1\ge\jmath_3\), we have
\(W_{k^*}\cap\bF_q^*F\neq\emptyset\). Also
\(\bF_qD\subset W_{k^*}\) and \(k^*\ge2\nu+2\). Therefore there exists
\(\beta\in\bF_q^*\) such that, for every \(\alpha\in\bF_q\), we can
choose \(y_\alpha\in H\) satisfying
\[
y_\alpha\in N_{k^*},
\qquad
\psi_{k^*}(y_\alpha)=\alpha D+\beta F.
\]
By Step~2, for every \(l\in [4k_E+J+3\nu-3,e]\) and every
\(\xi\in\bF_q^*\), we can choose \(x_{l,\xi}\in H\) satisfying
\[
x_{l,\xi}\in N_l,
\qquad
\psi_l(x_{l,\xi})=\xi E.
\]
We apply Lemma~\ref{lem:two-word-three} again, exchanging the roles of
\(E\) and \(F\), to \(y_\alpha\) and \(x_{l,\xi}\). We obtain
\begin{equation}\label{eq:Vkbound-two-key-formula-F}
\widehat{(y_\alpha,x_{l,\xi})}
\in N_{4k^*+l+3\nu-4},
\quad
\psi_{4k^*+l+3\nu-4}\bigl(\widehat{(y_\alpha,x_{l,\xi})}\bigr)
=
\bar u_0^3\alpha^2\xi\beta(\alpha D+\beta F).
\end{equation}
Thus, varying \(\alpha\), we obtain
\begin{equation}\label{eq:p2-all-basis-formula-3}
\bF_qF\subset W_r
\qquad\text{for every }
r\ge 4k^*+4k_E+J+6\nu-7.
\end{equation}

\medskip\noindent
\emph{Step~4: full image.}
By \eqref{eq:kstar-Vkbound-inequality-two} and the inequality
\(k_E\le k^*+\nu-1\), we have
\[
4k^*+4k_E+J+6\nu-7
<
4K+J+10\nu-11
< e.
\]
Therefore taking \(r=e\) in \eqref{eq:p2-all-basis-formula-1}, \eqref{eq:p2-all-basis-formula-2} and \eqref{eq:p2-all-basis-formula-3} gives
\[
W_e\supset
\bF_qD+\bF_qE+\bF_qF
=
\mathfrak{sl}_2(\bF_q),
\]
contradicting the exact level condition. 
The lemma follows.
\end{proof}

\subsection{Final estimates}\label{sec:final-estimates-local}
Recall that \(\gamma=e+\nu+1-\jmath_2\), \(m=e-j\), and
\(\hat c(j,s)\) were introduced in Section~\ref{sec:double-counting}.
Recall also that \(A_j(\ell)\) is defined in Section~\ref{sec:telescopic}. 

\begin{lem}\label{lem:Aj-bound-D}
Let $j\in [\nu+2,e]$. We have
\[
q^{j-2e}A_j(\ell) 
\le 2q^{1-e}+(e-j+1)^{3/2}(q-1)^{3/2}
q^{(1{+}\nu-\jmath_2-j)/2}
p^{-\hat c(j,0)/2}.
\]
\end{lem}

\begin{proof}
By Lemma~\ref{lem:Aj-fiber-sum},
$A_j(\ell)=\sum_{b\in\varpi^{j-1}R_e}
|\mathcal E_j(b)|$.
The $b=0$ term contributes $q^m$. There are $(q{-}1)q^{m-s}$ choices for $b\ne 0$ with $v_\fp(b)=j{-}1{+}s$,
$0\le s\le m$. 
Applying Corollary~\ref{cor:square-root-D}, we get
\begin{align*}
A_j(\ell)
&\le q^m
+\sum_{s=0}^m(q{-}1)q^{m-s}
\cdot\bigl(1+\sqrt{m(q{-}1)}\;
q^{(m+s+\gamma)/2}\;p^{-\hat c(j,s)/2}
\bigr)\\
&\le 2q^{m+1}
+\sqrt{m}(q-1)^{3/2}
q^{3m/2+\gamma/2}
\sum_{s=0}^m q^{-s/2}\;p^{-\hat c(j,s)/2}.
\end{align*}
Multiplying by
$q^{j-2e}$, using $m=e-j$ and
\begin{align*}
3m/2+\gamma/2+j-2e
& =3(e-j)/2+(\nu+e+1-\jmath_2)/2
+j-2e\\
& =(1+\nu-\jmath_2-j)/2,
\end{align*}
we get 
\begin{equation}\label{eq:Aj-bound-formula-1}
q^{j-2e}A_j(\ell) 
\le 2q^{1-e}+\sqrt{e-j}(q-1)^{3/2}
q^{(1+\nu-\jmath_2-j)/2}
\sum_{s=0}^{e-j} q^{-s/2}\;p^{-\hat c(j,s)/2}.
\end{equation}
Recall that
\[
\hat c(j,s)=
\#\bigl\{k\in\bZ \mid
\max(j{+}s{+}1,\jmath_2)\le k\le e,\;\\
\dim_{\bF_p}V_k\le f-1\bigr\},
\]
so
\begin{equation}\label{eq:Aj-bound-formula-2}
\hat c(j,s)+s\ge \hat c(j,0).
\end{equation}
Substituting \eqref{eq:Aj-bound-formula-2} and $q\ge p$
in \eqref{eq:Aj-bound-formula-1}, we get the conclusion.
\end{proof}

\begin{lem}\label{lem:Aj-bound-F}
Let $j\in [\nu+2,e]$. We have
\[
q^{j-2e}A_j(\ell) 
\le 2q^{1-e}+4\sqrt{e-j}(q-1)^{3/2}
q^{(1+\nu-\jmath_2-j)/2}
q^{-\lvert \Pi\rvert/2}.
\]
\end{lem}

\begin{proof}
The proof is similar to that of Lemma~\ref{lem:Aj-bound-D}. Writing $A_j(\ell)=\sum_{b\in\varpi^{j-1}R_e}
|\mathcal E_j(b)|$ and applying Corollary~\ref{cor:square-root-F} instead of Corollary~\ref{cor:square-root-D}, we get
\begin{align*}
A_j(\ell)
&\le q^m
+\sum_{s=0}^m(q{-}1)q^{m-s}
\cdot\bigl(1+\sqrt{m(q-1)}\;
q^{(m+s+\gamma-\lvert \Pi\rvert)/2}
\bigr)\\
&\le 2q^{m+1}
+\sqrt{m}(q-1)^{3/2}
q^{3m/2+\gamma/2-\lvert \Pi\rvert/2}
\sum_{s=0}^m q^{-s/2}.
\end{align*}
Note that $q\ge 2$ implies
\[
\sum_{s=0}^m q^{-s/2}\le \frac{1}{1-1/\sqrt{q}}\le 4.
\]
Combining again with $m=e-j$ and $3m/2+\gamma/2+j-2e=(1+\nu-\jmath_2-j)/2$, we get the result.
\end{proof}

\begin{prop}\label{prop:Aj-bound-three}
Assume that \(p\ge 3\), \(e\ge 2e_0^2+3e_0+4\), and assume that
\(H\subset G_e\) has exact level \(e\). For \(j\in [2,e]\), we have
\[
q^{j-2e}A_j(\ell) 
\le 2q^{1-e}+q^2e^{3/2}
p^{-\frac{e-3}{6+4e_0}}.
\]
\end{prop}
\begin{proof}
By definition of $T_j(\ell)$ (see \S\ref{sec:slope-decomposition}) and $\jmath_1$ (see \S\ref{sec:double-counting}), we have 
\begin{equation}\label{eq:Aj-bound-three-formula-1}
T_j(\ell)=0, \quad \text{when}\ j<\jmath_1.
\end{equation}
Combining \eqref{eq:Aj-bound-three-formula-1} with Lemma~\ref{lem:borel-tail-telescoping}, we deduce that
\begin{equation}\label{eq:Aj-bound-three-formula-2}
q^{j-2e}A_j(\ell)=q^{\jmath_1-2e}A_{\jmath_1}(\ell), \quad \text{when}\ 2\le j\le\jmath_1.
\end{equation}

Note that $\nu=v_\fp(2)=0$ when $p\ge 3$. By Lemma~\ref{lem:Aj-bound-D}, we have
\begin{align}
q^{j-2e}A_j(\ell) 
& \le 2q^{1-e}+(e-j+1)^{3/2}(q-1)^{3/2}
q^{(1{-}\jmath_2-j)/2}
p^{-\hat c(j,0)/2}\notag\\
& \le 2q^{1-e}+e^{3/2}q^{2}
q^{-(j+\jmath_2)/2}p^{-\hat c(j,0)/2}\notag\\
& \le 2q^{1-e}+q^{2}e^{3/2}p^{-\frac{1}{2}\big(j+\jmath_2+\hat c(j,0)\big)}.\label{eq:Aj-bound-three-formula-3}
\end{align}
Because of \eqref{eq:Aj-bound-three-formula-2}, we will assume, without loss of generality, that $j\ge \jmath_1$. 

\smallskip\noindent
\emph{Case (1).} Assume that $\max(j,\jmath_2)\ge \frac{e}{3+2e_0}$. Then \eqref{eq:Aj-bound-three-formula-3} implies the conclusion because $\hat c(j,0)\ge 0$. 

\smallskip\noindent
\emph{Case (2).} Assume now that $\max(j,\jmath_2)< \frac{e}{3+2e_0}$. Then, in particular, $\jmath=\max(\jmath_1,\jmath_2)<\frac{e}{3+2e_0}$. 
Set
\[
K:=e-\jmath-1, \quad M:=\max(\jmath,e_0+1).
\]
We claim that Lemma~\ref{lem:cross-level-Vk} applies. We need to prove $K>2M+2e_0$. If \(M=\jmath\), this follows from
\[
e-\jmath-1-2\jmath-2e_0
>
\frac{2e_0e}{3+2e_0}-2e_0-1>0,
\]
where the last inequality follows from
\(e\ge 2e_0^2+3e_0+4\).  If \(M=e_0+1>\jmath\), then
\(\jmath\le e_0\), and
\[
e-\jmath-1-2(e_0+1)-2e_0
\ge e-5e_0-3>0
\]
again because \(e\ge 2e_0^2+3e_0+4\). Thus Lemma~\ref{lem:cross-level-Vk}
applies and gives
\begin{align}
c_0(K)
&=
\#\{k\in \bZ\mid \jmath<k<e-\jmath,\
\dim_{\bF_p}V_k\le f-1\}
\notag\\
&\ge
\frac{K-2\max(\jmath,e_0+1)}{2e_0}.
\label{eq:Aj-bound-three-formula-4}
\end{align}

Recall that
\[
\hat{c}(j,0)=
\# \{k\in \bZ\mid e\ge k\ge \max(j+1,\jmath_2),\
\dim_{\bF_p}V_k\le f-1\}.
\]
Comparing \(\hat c(j,0)\) and \(c_0(K)\), we obtain
\begin{align}
\hat{c}(j,0)
&\ge
\# \{k\in \bZ\mid e\ge k\ge \jmath+1,\
\dim_{\bF_p}V_k\le f-1\}
-(j-\jmath_1)
\notag\\
&\ge
c_0(K)-j.
\label{eq:Aj-bound-three-formula-5}
\end{align}
Therefore
\[
\hat c(j,0)+j
\ge
\frac{K-2\max(\jmath,e_0+1)}{2e_0}.
\]

If \(\max(\jmath,e_0+1)=\jmath\), then
\[
\hat c(j,0)+j
\ge
\frac{e-3\jmath-1}{2e_0}
>
\frac{e}{3+2e_0}-\frac{1}{2e_0}
\ge
\frac{e-3}{3+2e_0},
\]
where the last inequality follows from \(e_0\ge1\).

If \(\max(\jmath,e_0+1)=e_0+1>\jmath\), then \(\jmath\le e_0\), and
\[
\hat c(j,0)+j
\ge
\frac{e-3e_0-3}{2e_0}.
\]
The inequality
\[
\frac{e-3e_0-3}{2e_0}\ge \frac{e-3}{3+2e_0}
\]
is equivalent to $e\ge 2e_0^2+3e_0+3$, which follows from our hypothesis.
Thus in all cases $\hat c(j,0)+j\ge \frac{e-3}{3+2e_0}$. The conclusion follows from \eqref{eq:Aj-bound-three-formula-3}.
\end{proof}

\begin{prop}\label{prop:Aj-bound-two}
Assume that \(p=2\), \(e>36\nu+13\), and assume that
\(H\subset G_e\) has exact level \(e\).  For \(j\in [2\nu+2,e]\), we have
\[
q^{j-2e}A_j(\ell) 
\le 2q^{1-e}+4q^{2+\nu/2}e^{3/2}
p^{-\frac{e-37\nu}{18+34\nu}}.
\]
\end{prop}
\begin{proof}
As in the beginning of the proof of Proposition~\ref{prop:Aj-bound-three}, we have
\begin{equation}\label{eq:Aj-bound-two-formula-1}
q^{j-2e}A_j(\ell)=q^{\jmath_1-2e}A_{\jmath_1}(\ell), \quad \text{when}\ 2\le j\le\jmath_1.
\end{equation}

Note that $\nu=v_\fp(2)=e_0$ when $p=2$. By Lemma~\ref{lem:Aj-bound-D}, we obtain, by the same arguments as in the proof of Proposition~\ref{prop:Aj-bound-three}, that, for $j\ge 2\nu+2$,
\begin{equation}
q^{j-2e}A_j(\ell)\le
2q^{1-e}+q^{2+\nu/2}e^{3/2}
p^{-\frac{1}{2}\big(j+\jmath_2+\hat c(j,0)\big)}.\label{eq:Aj-bound-two-formula-2}
\end{equation}
By Lemma~\ref{lem:Aj-bound-F}, we obtain also
\begin{equation}
q^{j-2e}A_j(\ell)\le
2q^{1-e}+4q^{2+\nu/2}\sqrt{e}
p^{-\frac{1}{2}\big(j+\jmath_2+\lvert \Pi\rvert\big)}.\label{eq:Aj-bound-two-formula-3}
\end{equation}

Because of \eqref{eq:Aj-bound-two-formula-1}, we will assume, without loss of generality, that $j\ge \jmath_1$. 

\smallskip\noindent
\emph{Case (1).} Assume that $\max(j,\jmath_2)\ge \frac{e-10\nu}{9+17\nu}$. Then \eqref{eq:Aj-bound-two-formula-2} implies the conclusion because $\hat c(j,0)\ge 0$. 

\smallskip\noindent
\emph{Case (2).} Assume now that 
\[
\max(j,\jmath_2)< \frac{e-10\nu}{9+17\nu}, \;
\jmath_3-\jmath_2\ge \nu\frac{e-10\nu}{9+17\nu}.
\]
Recall that $\jmath_3=\min\{k\ge \jmath_2\mid \forall l\ge k, W_l\cap \bF_q^*F\neq \emptyset\}$, see \S\ref{sec:semisimple-bounds}. By minimality we deduce that $W_{\jmath_3-1}\cap \bF_q^*F =\emptyset$. By Lemma~\ref{lem:propagation} we deduce that
\begin{equation*}
\forall l\le \jmath_3-1, l\equiv \jmath_3-1 \pmod \nu \implies W_{l}\cap \bF_q^*F =\emptyset.
\end{equation*}
Thus,
\begin{align}
\lvert \Pi \rvert & = \#\{k\in [\jmath_2,e]\mid W_{k}\cap \bF_q^*F =\emptyset\}
\notag\\
& \ge \#\{k\in [\jmath_2,\jmath_3-1]\mid W_{k}\cap \bF_q^*F =\emptyset\}
\notag\\
& \ge \#\{k\in [\jmath_2,\jmath_3-1]\mid k\equiv \jmath_3-1 \pmod \nu \}
\notag\\
& \ge \frac{\jmath_3-\jmath_2}{\nu} 
\ge \frac{e-10\nu}{9+17\nu}.\label{eq:Aj-bound-two-formula-4}
\end{align}
Substituting \eqref{eq:Aj-bound-two-formula-4} into \eqref{eq:Aj-bound-two-formula-3} implies the conclusion.

\smallskip\noindent
\emph{Case (3).} Assume now that
\[
\max(j,\jmath_2)< \frac{e-10\nu}{9+17\nu},\qquad
\jmath_3-\jmath_2< \nu\frac{e-10\nu}{9+17\nu}.
\]
Then
\begin{align}
& \jmath_1\le j < \frac{e-10\nu}{9+17\nu},\qquad
\jmath_3<(\nu+1)\frac{e-10\nu}{9+17\nu},
\notag\\
& \jmath^*=\max(\jmath_1,\jmath_3)
<(\nu+1)\frac{e-10\nu}{9+17\nu}.
\label{eq:Aj-bound-two-formula-7}
\end{align}
Set
\(
J:=\max(\jmath^*,2\nu+1)
\)
and 
\[
K:=\left\lceil \frac{1}{4}(e+4-J-10\nu)\right\rceil-1.
\]
Then
\[
K<\frac{1}{4}(e+4-J-10\nu).
\]

We first check that Lemma~\ref{lem:cross-level-Vk-two} applies.  From
\eqref{eq:Aj-bound-two-formula-7}, the same computation as before gives
\[
9\jmath^*+18\nu<e-4.
\]
On the other hand,
\[
9(2\nu+1)+18\nu=36\nu+9<e-4,
\]
because \(e>36\nu+13\). Hence
\[
9J+18\nu<e-4.
\]
Therefore
\[
K\ge \frac{1}{4}(e+4-J-10\nu)-1
>2J+2\nu.
\]
Thus Lemma~\ref{lem:cross-level-Vk-two} applies and gives
\begin{align}
c_0^*(K)
&=
\#\{k\in\bZ\mid J<k\le K,\ \dim_{\bF_p}V_k\le f-1\}
\notag\\
&\ge \frac{K-2J-2\nu}{2\nu}.
\label{eq:Aj-bound-two-formula-5}
\end{align}

Recall that
\[
\hat{c}(j,0)
=
\#\{k\in\bZ\mid e\ge k\ge \max(j+1,\jmath_2),\
\dim_{\bF_p}V_k\le f-1\}.
\]
Since \(J\ge\jmath^*\ge\jmath_1\), comparing the two counting sets gives
\begin{align}
\hat{c}(j,0)
&\ge
\#\{k\in\bZ\mid e\ge k\ge J+1,\
\dim_{\bF_p}V_k\le f-1\}
-(j-\jmath_1)
\notag\\
&\ge c_0^*(K)-j
\ge \frac{K-2J-2\nu}{2\nu}-j.
\label{eq:Aj-bound-two-formula-6}
\end{align}
Hence
\[
\hat c(j,0)+j\ge \frac{K-2J}{2\nu}-1.
\]
Using the definition of \(K\), we get
\begin{align*}
\hat c(j,0)+j
&\ge
\frac{1}{2\nu}
\left(
\frac{1}{4}(e+4-J-10\nu)-1-2J-2\nu
\right)
\\
&=
\frac{e-9J-18\nu}{8\nu}.
\end{align*}

If \(J=\jmath^*\), then \eqref{eq:Aj-bound-two-formula-7} gives
\[
\hat c(j,0)+j
>
\frac{e-27\nu-9}{17\nu+9}
\ge
\frac{e-37\nu}{17\nu+9}.
\]
If \(J=2\nu+1\), then
\[
\hat c(j,0)+j
\ge
\frac{e-36\nu-9}{8\nu}.
\]
We claim that
\[
\frac{e-36\nu-9}{8\nu}
>
\frac{e-37\nu}{17\nu+9}.
\]
Indeed, after clearing positive denominators, this is equivalent to
\[
(17\nu+9)(e-36\nu-9)>8\nu(e-37\nu),
\]
or equivalently to
\[
9(\nu+1)e>316\nu^2+477\nu+81.
\]
Since \(e\) is an integer and \(e>36\nu+13\), we have
\(e\ge 36\nu+14\). Therefore
\[
9(\nu+1)e
\ge
9(\nu+1)(36\nu+14),
\]
and the difference is
\[
9(\nu+1)(36\nu+14)-(316\nu^2+477\nu+81)
=
8\nu^2-27\nu+45>0.
\]
This proves the claim. Therefore in all cases
\[
\hat c(j,0)+j
\ge
\frac{e-37\nu}{17\nu+9}.
\]
The conclusion follows from \eqref{eq:Aj-bound-two-formula-2}.
\end{proof}

\subsection{Proof of Theorem~\ref{thm:main_local_chi}}

\begin{prop}\label{prop:chi-N2-odd}
Assume \(p\ge 3\), \(e\ge 2e_0^2+3e_0+4\), and assume that
\(H\subset G_e\) has exact level \(e\). Then
\[
\chi_e(N_2)
\le
q^{-e}+2q^{1-e}+q^2e^{3/2}p^{-(e-3)/(6+4e_0)}.
\]
\end{prop}

\begin{proof}
By \eqref{eq:exact-chi-lines}, applied with \(j_0=2\),
\[
\chi_e(N_2)
=
q^{-e}
+
\frac{1}{q+1}
\sum_{\lambda\in\bP^1(\bF_q)}
\sum_{r=2}^e T_r(\lambda)q^{r-2e}.
\]
Fix \(\lambda\in\bP^1(\bF_q)\).  Choose \(g\in\SL_2(\bF_q)\) with
\(g\ell=\lambda\), choose a lift \(\hat g\in G_e\), and set
\(H_{\hat g}:=\hat g^{-1}H\hat g\). By \eqref{eq:conjugated-tail-telescoping-interval}, applied to \(H_{\hat g}\),
\[
\sum_{r=2}^e T_r(\lambda)q^{r-2e}
=
q^{2-2e}A_2(\ell;H_{\hat g}).
\]
Since \(H_{\hat g}\) has exact level \(e\), Proposition~\ref{prop:Aj-bound-three}
gives
\[
q^{2-2e}A_2(\ell;H_{\hat g})
\le
2q^{1-e}+q^2e^{3/2}p^{-(e-3)/(6+4e_0)}.
\]
Substituting this into the preceding formula for \(\chi_e(N_2)\), and using
\(|\bP^1(\bF_q)|=q+1\), gives the result.
\end{proof}

\begin{prop}\label{prop:chi-Nnu-dyadic}
Assume \(p=2\) and \(e>36\nu+13\).  If
\(H\subset G_e\) has exact level \(e\), then
\[
\chi_e(N_{2\nu+2})
\le
q^{-e}+2q^{1-e}
+4q^{2+\nu/2}e^{3/2}p^{-(e-37\nu)/(18+34\nu)}.
\]
\end{prop}

\begin{proof}
Repeat the proof of Proposition~\ref{prop:chi-N2-odd}, with \(N_2\)
replaced with \(N_{2\nu+2}\), and use
Proposition~\ref{prop:Aj-bound-two}.
\end{proof}

\begin{proof}[Proof of Theorem~\ref{thm:main_local_chi}]
Assume first that \(p\ge3\) and \(e\ge 2e_0^2+3e_0+4\). Lemma~\ref{lem:normal-subgroup} gives
\[
\chi_e(H)\le [H:N_2]\chi_e(N_2).
\]
Moreover \([H:N_2]=|H_1|\le |G_1|=q(q^2-1)\).  The first estimate therefore
follows from Proposition~\ref{prop:chi-N2-odd}.

Assume now that \(p=2\). Recall that \(\nu=e_0\) in this case.
Lemma~\ref{lem:normal-subgroup} gives
\[
\chi_e(H)\le [H:N_{2\nu+2}]\chi_e(N_{2\nu+2}).
\]
Moreover
\[
[H:N_{2\nu+2}]
=
|H_{2\nu+1}|
\le
|G_{2\nu+1}|
=
q^{6\nu+1}(q^2-1),
\]
where the last equality follows from Lemma~\ref{lem:order-Ge}. The estimate then follows from
Proposition~\ref{prop:chi-Nnu-dyadic}.
\end{proof}

\begin{proof}[Proof of Theorem~\ref{thm:uniform-exponential-local-chi}]
We have \(e_0\le n\), \(f\le n\), and \(q=p^f\).

Assume first that \(p\ge3\).  Put \(M:=6+4n\). If
\(e\ge 2n^2+3n+4\), then Theorem~\ref{thm:main_local_chi} applies, and the three terms in its formula are bounded by
\(p^{3-e}\), \(p^{5-e}\), and \(e^{3/2}p^{5n+1-e/M}\), respectively. We remark that
\[
  3x^{3/2}\le 3^{M/2+x/(2M)},\qquad M\ge10,\ x\ge1.
\]
Hence
\begin{align*}
\chi_e(H) & \le 3e^{3/2}p^{5n+1-e/M}
  \le
  p^{5n+1+M/2-e/(2M)}
  \\
  & =
  p^{7n+4-e/(2M)}
  \le
  q^{7n+4-e/(2nM)}.
\end{align*}

Assume now that \(p=2\). Put \(M=18+34n\). If \(e>36n+13\), then
Theorem~\ref{thm:main_local_chi} applies, and its three terms are
bounded by
\[
2^{6n^2+3n-e},\qquad
2^{6n^2+4n+1-e},\qquad
e^{3/2}2^{\frac{13}{2}n^2+5n+4-e/M}.
\]
Thus all three terms are bounded by the last expression, and
\[
  \chi_e(H)
  \le
  3e^{3/2}2^{\frac{13}{2}n^2+5n+4-e/M}.
\]
We remark
\[
  3x^{3/2}\le 2^{M/4+x/(2M)},\qquad M\ge52,\ x\ge1.
\]
Hence 
\begin{align*}
\chi_e(H) & \le 2^{\frac{13}{2}n^2+5n+4+M/4-e/(2M)}
  =
  2^{\frac{13}{2}n^2+\frac{27}{2}n+\frac{17}{2}-e/(2M)}.
\end{align*}

For small values of \(e\) in both cases, we take simply the trivial bound \(\chi_e(H)\le1\).
\end{proof}

\section{Growth of cusp numbers with multiplicity}\label{sec:global-field-section}
In this section we will work with subgroups of $\PGL_2(\bK)$ commensurable with $\PSL_2(\cO_\bK)$ where $\bK$ is a number field. This amounts to working with non-cocompact arithmetic irreducible lattices in Lie groups of the form $\PSL_2{(\bR)}^r\times \PSL_2(\bC)^s$ with $r+s\ge 1$. More details will be explained below and we refer to \cite{Vigneras}, \cite{Voight} for the classical theory of these arithmetic groups, and to \cite{MasonSchweizer12} for cusps.

\subsection{Lattices and arithmetic groups}\label{sec:lattices-arithmetic-groups}
We collect in this subsection some classical descriptions of the arithmetic groups central to this paper.
\subsubsection{Lattices}
Let \(\bK\) be a number field, and let \(\cO_\bK\) be its ring of
integers.  Put
\[
G_\infty:=\prod_{v\mid\infty}\PGL_2(\bK_v),
\]
and let
\[
\iota_\infty:\PGL_2(\bK)\hookrightarrow G_\infty
\]
be the diagonal embedding.  Consider
\[
\PSL_2(\cO_\bK)\subset\PGL_2(\bK),
\qquad
\Delta_\bK:=\iota_\infty(\PSL_2(\cO_\bK))\subset G_\infty .
\]
By the Borel--Harish-Chandra theorem, \(\Delta_\bK\) is a lattice in
\(G_\infty\); it is non-cocompact.  If \(\PGL_2(\cO_\bK)\) denotes the
image of \(\GL_2(\cO_\bK)\) in \(\PGL_2(\bK)\), then
\(\PGL_2(\cO_\bK)\) is commensurable with
\(\PSL_2(\cO_\bK)\). Indeed, the determinant induces an injective homomorphism
\[
\PGL_2(\cO_\bK)/\PSL_2(\cO_\bK)
\hookrightarrow
\cO_\bK^*/(\cO_\bK^*)^2,
\]
and the target is finite by Dirichlet's unit theorem.

Let
\[
\cC_{\bK}^{\mathrm{rat}}
:=
\left\{
\Lambda\subset\PGL_2(\bK)\mid
\Lambda \text{ is commensurable with }
\PSL_2(\cO_\bK)
\right\},
\]
and let
\[
\cC_{\bK}^{\infty}
:=
\left\{
\Delta\subset G_\infty\mid
\Delta \text{ is commensurable with } \Delta_\bK
\right\}.
\]
By Borel's description of arithmetic subgroups associated with \(\PGL_2\), the elements of
\(\cC_{\bK}^{\infty}\) are precisely the arithmetically defined subgroups
of \(G_\infty\) belonging to this split \(\bK\)-form.  See
\cite[\S 3.3]{Borel1981}.

More precisely, the map
\[
\cC_{\bK}^{\mathrm{rat}}\longrightarrow \cC_{\bK}^{\infty},
\qquad
\Lambda\longmapsto \iota_\infty(\Lambda),
\]
is a bijection.  
Let \(\Delta\in\cC_{\bK}^{\infty}\).  If \(\delta\in\Delta\), then \(\delta\) commensurates \(\Delta_\bK\). Indeed, \(\Delta\cap\Delta_\bK\) and \(\delta(\Delta\cap\Delta_\bK)\delta^{-1}\) are finite-index subgroups of \(\Delta\), and their intersection is contained in \(\Delta_\bK\cap\delta\Delta_\bK\delta^{-1}\).  Hence \(\Delta\) is contained in the commensurator group \(\operatorname{Comm}_{G_\infty}(\Delta_\bK)\). By \cite[Theorem~3]{Borel1966} we have 
\[
\operatorname{Comm}_{G_\infty}(\Delta_\bK)
=
\iota_\infty(\PGL_2(\bK)).
\]
Thus \(\Delta=\iota_\infty(\Lambda)\), \(\Lambda:=\iota_\infty^{-1}(\Delta)\subset\PGL_2(\bK)\), and \(\Lambda\) is commensurable with
\(\PSL_2(\cO_\bK)\).

\subsubsection{Maximal arithmetic overgroups}
For every finite place \(v\), let \(\mathscr T_v\) be the
Bruhat--Tits tree of \(\PGL_2(\bK_v)\).  We use the standard model in
which the vertices are the homothety classes of
\(\cO_{\bK_v}\)-lattices \(L\subset\bK_v^2\), and two vertices
\([L]\), \([L']\) are adjacent if representatives can be chosen with
\[
\varpi_v L\subsetneq L'\subsetneq L,
\]
where \(\varpi_v\) is a uniformizer of \(\cO_{\bK_v}\).  The action of
\(\PGL_2(\bK_v)\) on \(\mathscr T_v\) is induced by the natural action
of \(\GL_2(\bK_v)\) on lattices.

To a vertex \(x=[L]\) we attach the maximal order
\[
\cM_{v,x}:=\operatorname{End}_{\cO_{\bK_v}}(L)\subset M_2(\bK_v).
\]
This is independent of the representative \(L\), and identifies the vertices of \(\mathscr T_v\) with maximal orders in
\(M_2(\bK_v)\). Under this identification, the action of
\(\PGL_2(\bK_v)\) on vertices is the action by conjugation on maximal
orders.

For a vertex \(x\), one has
\[
\operatorname{Stab}_{\PGL_2(\bK_v)}(x)
=
N_{\GL_2(\bK_v)}(\cM_{v,x})/\bK_v^* .
\]
Indeed, both sides consist of the projective classes of those
\(g\in\GL_2(\bK_v)\) satisfying
\(g\cM_{v,x}g^{-1}=\cM_{v,x}\).

Let \(e=\{x,y\}\) be an unoriented edge of \(\mathscr T_v\), and set
\[
\cE_{v,e}:=\cM_{v,x}\cap\cM_{v,y}.
\]
Then \(\cE_{v,e}\) is an Eichler order.  By
\cite[Ch.~II, \S2, Lemma~2.4]{Vigneras}, an Eichler order in
\(M_2(\bK_v)\) is the intersection of a unique pair of maximal orders.
Therefore
\[
\operatorname{Stab}_{\PGL_2(\bK_v)}(e)
=
N_{\GL_2(\bK_v)}(\cE_{v,e})/\bK_v^* ,
\]
where the stabilizer of \(e\) is the setwise stabilizer of the unordered
pair \(\{x,y\}\).  Indeed, preserving the edge is the same as preserving
the unordered pair \(\{\cM_{v,x},\cM_{v,y}\}\), and by the uniqueness
statement in \cite[Ch.~II, \S2, Lemma~2.4]{Vigneras} this is equivalent
to normalizing their intersection \(\cE_{v,e}\).

By \cite[Proposition~4.4(iii)]{Borel1981}, every arithmetic subgroup of
\(\PGL_2(\bK)\) commensurable with \(\PSL_2(\cO_\bK)\) is, after
conjugacy in \(\PGL_2(\bK)\), contained in a group of the form
\[
\PGL_2(\bK)\cap \prod_{v \, \mathrm{finite}} C_v,
\]
where \(\PGL_2(\bK)\) is embedded diagonally in
\(\prod_{v\,\mathrm{finite}}\PGL_2(\bK_v)\), each \(C_v\) is either the
stabilizer of a vertex of \(\mathscr T_v\) or the setwise stabilizer of
an unoriented edge of \(\mathscr T_v\), and \(C_v\) is the stabilizer of
the vertex represented by \(\cO_{\bK_v}^2\) for all but finitely many
\(v\).

For each finite \(v\), define an order \(\cE_v\subset M_2(\bK_v)\) as
follows.  If \(C_v=\operatorname{Stab}(x)\) for a vertex \(x\), put
\[
\cE_v:=\cM_{v,x}.
\]
If \(C_v=\operatorname{Stab}(e)\) for an unoriented edge \(e\), put
\[
\cE_v:=\cE_{v,e}.
\]
By the preceding stabilizer computations,
\[
C_v=
N_{\GL_2(\bK_v)}(\cE_v)/\bK_v^*
\]
for every finite place \(v\).  Moreover \(\cE_v\) is maximal for all but
finitely many \(v\).

By the local-global correspondence for lattices
\cite[Ch.~III, \S5, Proposition~5.1]{Vigneras}, there is an
\(\cO_\bK\)-lattice \(\cE\subset M_2(\bK)\) whose completion at every
finite place \(v\) is \(\cE_v\); being an order, being a maximal order, and being an
Eichler order are local properties (see \cite{Vigneras} immediately after the cited proposition, see also \cite[Theorem~9.4.9]{Voight}).  Hence \(\cE\) is an Eichler order in
\(M_2(\bK)\). For an order
\(\cA\subset M_2(\bK)\), set
\[
N(\cA):=\{g\in\GL_2(\bK)\mid g\cA g^{-1}=\cA\},
\;
\PN(\cA):=N(\cA)/\bK^*\subset\PGL_2(\bK).
\]
The same local-global theorem gives
\begin{align*}
& N(\cE)
=
\{g\in\GL_2(\bK)\mid g_v\cE_vg_v^{-1}=\cE_v
\text{ for every finite }v\}, 
\\
& \PGL_2(\bK)\cap \prod_{v\,\mathrm{finite}} C_v
=
\PN(\cE).
\end{align*}
Indeed, a projective class \([g]\in\PGL_2(\bK)\) belongs to the left-hand
side if and only if, for every finite \(v\), some scalar multiple of
\(g_v\) normalizes \(\cE_v\).  Since scalar conjugation is trivial, this
is equivalent to \(g_v\cE_vg_v^{-1}=\cE_v\) for every finite \(v\), hence
to \(g\cE g^{-1}=\cE\) by the preceding local-global criterion.

Consequently, if
\(\Lambda\subset\PGL_2(\bK)\) is commensurable with
\(\PSL_2(\cO_\bK)\), then there exist
\(g\in\PGL_2(\bK)\) and an Eichler order
\(\cE\subset M_2(\bK)\) such that
\[
g\Lambda g^{-1}\subset \PN(\cE).
\]
Conversely, every subgroup
\(\Lambda\subset\PN(\cE)\) commensurable with
\(\PSL_2(\cO_\bK)\) has diagonal image commensurable with
\(\Delta_\bK\).  Hence, up to conjugacy, the lattices in
\(\cC_{\bK}^{\infty}\) are exactly the diagonal images of the groups
\[
\Lambda\subset\PN(\cE)\subset\PGL_2(\bK),
\]
where \(\cE\subset M_2(\bK)\) is an Eichler order and
\(\Lambda\) is commensurable with \(\PSL_2(\cO_\bK)\).

\subsection{Global arithmetic setting}
\label{sec:global-arithmetic-setting}
Let $\bK$ be a number field with ring of integers $\cO_\bK$. Let $M_2(\bK)$ be the space of $2 \times 2$ matrices. Let \(\cE\subset M_2(\bK)\) be an Eichler \(\cO_\bK\)-order, and choose a
maximal \(\cO_\bK\)-order \(\cM\subset M_2(\bK)\) containing \(\cE\).  We
use the notation \(N(\cE)\) and \(\PN(\cE)\) introduced in \S\ref{sec:lattices-arithmetic-groups}, and fix a
subgroup
\[
\Lambda\subset\PN(\cE)
\]
commensurable with \(\PSL_2(\cO_\bK)\). Put
\[
\cG := \SL(\cM) = \{g \in \cM : \det g = 1\},
\qquad
\P\cG := \operatorname{im}\!\left(\cG \longrightarrow \PGL_2(\bK)\right).
\]
Since $\cM$ is commensurable with $M_2(\cO_\bK)$ as an $\cO_\bK$-lattice, $\cG$ is commensurable with $\SL_2(\cO_\bK)$. We write $\Proj g$ for the image in $\PGL_2(\bK)$ of an element $g \in \GL_2(\bK)$. We set
\[
\Lambda_\cM := \Lambda \cap \P\cG
\]
and let
\[
\Gamma
:=
\{g \in \cG : \Proj g \in \Lambda_\cM\}
\subset \cG
\]
be the full inverse image of $\Lambda_\cM$ in $\cG$; it has finite index in $\cG$ and contains $-I$, with $\Proj\Gamma = \Lambda_\cM$.

For a non-zero ideal $\fn\subset\cO_\bK$, the product $\fn\cM$ is a two-sided ideal of $\cM$, since $\fn$ is central. Hence $\cM/\fn\cM$ is a finite ring, and reduction modulo $\fn\cM$ gives a homomorphism
\[
\cM^*\longrightarrow(\cM/\fn\cM)^*.
\]
We define the \emph{principal congruence subgroup of level} $\fn$, relative to $\cM$, by
\[
\cG(\fn)
:=
\ker\!\left(
\SL(\cM)\longrightarrow(\cM/\fn\cM)^*
\right).
\]
Equivalently, $\cG(\fn)=\{g\in\SL(\cM):g\equiv I\pmod{\fn\cM}\}$.

\begin{lem}\label{lem:congruence-product-maximal-order}
For any two non-zero ideals $\fa,\fb\subset\cO_\bK$, one has
\[
\cG(\fa)\cG(\fb)=\cG(\fa+\fb).
\]
\end{lem}

\begin{proof}
The case $\cM=M_2(\cO_\bK)$ is \cite[Cor.~(9.3), p.~267]{BassBook}. We include the proof for general $\cM$, since we did not find an exact reference for the required form. Since $\cM$ is a maximal order in $M_2(\bK)$, there exists a projective $\cO_\bK$-module $L$ of rank two such that $\cM=\operatorname{End}_{\cO_\bK}(L)$. Thus $\cG=\SL(L)$, where $\SL(L)$ denotes the group of determinant-one automorphisms of $L$.

We first show that, for every non-zero ideal $\fd\subset\cO_\bK$, the reduction map $\SL(L)\to \SL(L/\fd L)$ is surjective. By the structure theorem for finitely generated torsion-free modules over a
Dedekind domain, we may write $L\simeq \cO_\bK\oplus I$ for some
fractional ideal $I$. Since $\cO_\bK/\fd$ is finite, it is semi-local. Hence the invertible $\cO_\bK/\fd$-module $I/\fd I$ is free of rank one. 

Let $E_2(\cO_\bK)$ denote the subgroup of $\SL_2(\cO_\bK)$ generated by the matrices
\[
\begin{pmatrix}1&a\\0&1\end{pmatrix},
\qquad
\begin{pmatrix}1&0\\a&1\end{pmatrix},
\qquad a\in\cO_\bK.
\]
By \cite[Cor.~(9.3), p.~267]{BassBook}, the reduction map
\[
E_2(\cO_\bK)\longrightarrow \SL_2(\cO_\bK/\fd)
\]
is surjective. After choosing a generator $\bar s$ of the free rank-one $\cO_\bK/\fd$-module $I/\fd I$, we identify
\[
L/\fd L=(\cO_\bK/\fd)\oplus(I/\fd I)
\]
with $(\cO_\bK/\fd)^2$. Under this identification, every element of $\SL(L/\fd L)$ is a product of transformations of the following two forms:
\[
(\bar x,\bar y)\longmapsto(\bar x+\bar a\bar y,\bar y),
\qquad
(\bar x,\bar y)\longmapsto(\bar x,\bar y+\bar a\bar x),
\qquad
\bar a\in\cO_\bK/\fd.
\]

We check that both types lift to determinant-one automorphisms of $L=\cO_\bK\oplus I$. For the transformation $(\bar x,\bar y)\mapsto(\bar x,\bar y+\bar a\bar x)$, choose $c\in I$ lifting $\bar a\bar s\in I/\fd I$. Then
\[
(x,y)\longmapsto(x,y+cx)
\]
is an automorphism of $L$ with determinant one, and its reduction modulo $\fd$ gives the required transformation. For the transformation $(\bar x,\bar y)\mapsto(\bar x+\bar a\bar y,\bar y)$, choose
\[
\varphi\in\operatorname{Hom}_{\cO_\bK}(I,\cO_\bK)
\]
whose reduction sends $\bar s$ to $\bar a$. Such a lift exists because $I$ is projective, hence the natural map
\[
\operatorname{Hom}_{\cO_\bK}(I,\cO_\bK)
\longrightarrow
\operatorname{Hom}_{\cO_\bK/\fd}(I/\fd I,\cO_\bK/\fd)
\]
is surjective. Then
\[
(x,y)\longmapsto(x+\varphi(y),y)
\]
is an automorphism of $L$ with determinant one, and its reduction modulo $\fd$ gives the required transformation. Thus every generator of $\SL(L/\fd L)$ obtained from Bass's theorem lifts to an element of $\SL(L)$, so the reduction map $\SL(L)\to\SL(L/\fd L)$ is surjective.

We now prove the product formula. The inclusion $\cG(\fa)\cG(\fb)\subset \cG(\fa+\fb)$ is immediate. Conversely, let $g\in\cG(\fa+\fb)$ and set $\fd:=\fa\cap\fb$. We use the standard exact sequence
\[
0\to \cO_\bK/\fd
\to \cO_\bK/\fa\oplus \cO_\bK/\fb
\to \cO_\bK/(\fa+\fb)\to 0.
\]
Since $L$ is projective, hence flat, tensoring with $L$ gives
\[
0\to L/\fd L
\to L/\fa L\oplus L/\fb L
\to L/(\fa+\fb)L\to 0.
\]
Now $g\in\cG(\fa+\fb)$ acts trivially on $L/(\fa+\fb)L$. Hence the identity on $L/\fa L$ and the automorphism induced by $g$ on $L/\fb L$ are compatible, and they glue to an automorphism $\bar x$ of $L/\fd L$. By construction, $\bar x\equiv I\pmod{\fa}$ and $\bar x\equiv g\pmod{\fb}$. The compatible pair $(I,g^{-1})$ glues to the inverse of $\bar x$, so $\bar x$ is an automorphism. Its determinant, in the sense of the action
on $\bigwedge^2_{\cO_\bK/\fd}(L/\fd L)$, reduces to $1$ modulo both
$\fa$ and $\fb$; hence it is $1$ modulo $\fd=\fa\cap\fb$. By the surjectivity proved above, choose $x\in\SL(L)=\cG$ lifting $\bar x$. Then $x\in\cG(\fa)$ and $x^{-1}g\in\cG(\fb)$. Hence $g=x(x^{-1}g)\in\cG(\fa)\cG(\fb)$, proving the reverse inclusion.
\end{proof}

A subgroup of $\cG$ is a \emph{congruence subgroup}, relative to $\cM$, if it contains $\cG(\fn)$ for some non-zero ideal $\fn$. Henceforth $\Gamma$ is assumed to be a congruence subgroup relative to $\cM$. By Lemma~\ref{lem:congruence-product-maximal-order}, the set of ideals $\fa$ such that $\cG(\fa)\subset\Gamma$ is closed under sums. It therefore has a largest element with respect to divisibility. We denote it by $\fn$ and call it the \emph{congruence level of $\Gamma$ relative to $\cM$}.

Let $\cH$ be a subgroup of $\GL_2(\bK)$ or of $\PGL_2(\bK)$. A \emph{cusp of $\cH$} is an $\cH$-orbit of parabolic fixed points of $\cH$ on $\bP^1(\bK)$. When $\cH$ is commensurable with $\cG$ or with $\P\cG$, every point of $\bP^1(\bK)$ is a parabolic fixed point of $\cH$, so the set of cusps is naturally identified with $\cH \backslash \bP^1(\bK)$. The inclusion $\Lambda_\cM \subset \Lambda$ induces a finite map
\[
\Gamma \backslash \bP^1(\bK)
=
\Lambda_\cM \backslash \bP^1(\bK)
\longrightarrow
\Lambda \backslash \bP^1(\bK),
\]
so each $\Lambda$-cusp is the image of finitely many $\Gamma$-cusps.

\subsection{Finiteness for fixed congruence level}
\label{sec:fixed-level-finiteness}

Let \(G_\infty=\prod_{v\mid\infty}\PGL_2(\bK_v)\) be the archimedean group
fixed in Subsection~\ref{sec:global-arithmetic-setting}. The following fact is classical:
\begin{prop}\label{prop:fixed-level-finitely-many-lie-classes}
Fix a non-zero ideal \(\fn\subset\cO_\bK\). Consider all subgroups
\(\Lambda\subset\PGL_2(\bK)\) for which there exist an Eichler order
\(\cE\subset M_2(\bK)\) and a maximal order \(\cM\supset\cE\) such that
\(
\Lambda\subset\PN(\cE)
\), and such that the congruence level of \(\Gamma\) relative to \(\cM\) is exactly \(\fn\). Then these groups \(\Lambda\) form only finitely many
\(\PGL_2(\bK)\)-conjugacy classes. Equivalently, their diagonal images in
\(G_\infty\) form only finitely many \(G_\infty\)-conjugacy classes.
\end{prop}

\begin{proof}
Recall the notations \(\cG=\SL(\cM)\), \(\Lambda_\cM=\Lambda\cap\P\cG\), so that
\[
\Gamma=\{g\in\cG\mid \Proj g\in\Lambda_\cM\}.
\]
To prove the proposition it is enough to prove that the covolumes of the groups \(\Lambda\) are
uniformly bounded in terms of \(\bK\) and \(\fn\), because of the finiteness theorem of \cite{BorelPrasad1989}.

Since \(\cG(\fn)\subset\Gamma\), we have \(\P\cG(\fn)\subset\Lambda\), and
therefore
\[
  \operatorname{vol}\bigl(G_\infty/\iota_\infty(\Lambda)\bigr)
  \le
  \operatorname{vol}\bigl(G_\infty/\iota_\infty(\P\cG(\fn))\bigr).
\]
We bound the right-hand side as \(\cM\) varies. Every maximal order in \(M_2(\bK)\) is of the form \(\operatorname{End}_{\cO_\bK}(L)\), with \(L\) a rank-two projective
\(\cO_\bK\)-module. By the Steinitz classification there are only finitely many such \(L\) up to isomorphism. Hence maximal orders in \(M_2(\bK)\) fall into only finitely
many \(\PGL_2(\bK)\)-conjugacy classes.
For \(\cM=\operatorname{End}_{\cO_\bK}(L)\), 
\[
  [\P\cG:\P\cG(\fn)]
  \leq
  |\cG/\cG(\fn)|
  =
  |\SL(L/\fn L)|
  =
  |\SL_2(\cO_\bK/\fn)|.
\]

Finally, in the commensurability class \(G_\infty\)-conjugacy is the
same as \(\PGL_2(\bK)\)-conjugacy. If
\(h\iota_\infty(\Lambda_1)h^{-1}=\iota_\infty(\Lambda_2)\), then \(h\)
commensurates \(\Delta_\bK=\iota_\infty(\PSL_2(\cO_\bK))\). By
\cite{Borel1966} (cf.\ \cite[\S~4.3]{Borel1981}), \(h=\iota_\infty(g)\) for some \(g\in\PGL_2(\bK)\), and injectivity of \(\iota_\infty\) gives \(g\Lambda_1g^{-1}=\Lambda_2\).
\end{proof}

\subsection{Quasi-amplitudes and cusp multiplicity}
\label{sec:cusps-hirzebruch-multiplicity}

Set
\[
\bK_+^*
:=
\{x \in \bK^* \mid \rho(x) > 0 \text{ for every real embedding } \rho : \bK \hookrightarrow \bR\}.
\]
If $\bK$ has no real embeddings, then $\bK_+^* = \bK^*$. Set
\[
\PGL_2^+(\bK):=\{\lambda\in \PGL_2(\bK)\mid \exists A\in \GL_2(\bK), \det A\in \bK_+^*, \operatorname{P}A=\lambda\}.
\]
Then $\iota_\infty$ sends arithmetic subgroups contained in $\PGL_2^+(\bK)$ to lattices contained in the connected component of $G_\infty$ which has the form
\[
\PSL_2(\bR)^r\times \PSL_2(\bC)^s. 
\]
From now on we will always assume that 
\[
\Lambda \subset \PGL_2^+(\bK).
\]

For a full $\bZ$-lattice $M \subset \bK$, set
\[
U_M^+ := \{\varepsilon \in \bK_+^* \mid \varepsilon M = M\}.
\]

Let $\sigma \in \bP^1(\bK)$ be a $\Gamma$-cusp and choose $g \in \SL_2(\bK)$ with $g(\infty) = \sigma$. Set
\[
B(\bK)
:=
\left\{ \begin{pmatrix} a & b \\ 0 & a^{-1} \end{pmatrix} \mid a \in \bK^*,\ b \in \bK \right\},
\]
and write
\[
u(x) := \begin{pmatrix} 1 & x \\ 0 & 1 \end{pmatrix}
\quad\text{for } x \in \bK.
\]
We think of the conjugation element \(g\) as providing a coordinate chart around the cusp. We use the following terminology from \cite{MasonSchweizer12}. The \emph{quasi-amplitude} of $\sigma$ in the chart $g$ is the additive group
\[
b(\Gamma, g)
:=
\{x \in \bK \mid u(x) \in g^{-1}\Gamma g\},
\]
a full $\bZ$-lattice in $\bK$. We have the \emph{ambient quasi-amplitude}
\[
b(\cG, g) := \{x \in \bK \mid u(x) \in g^{-1} \cG g\},
\]
which contains $b(\Gamma, g)$ as a finite-index subgroup, because $\Gamma$ has finite index in $\cG$. Write
\[
N_g := g
\begin{pmatrix}0&1\\0&0\end{pmatrix}
g^{-1}.
\]
Then $N_g^2=0$, and $g u(x) g^{-1}=I+xN_g$. Hence $\det(I+xN_g)=1$ for every $x\in \bK$. Since $\cG=\SL(\cM)=\{h\in \cM\mid\det h=1\}$ and $I\in\cM$, we get
\[
b(\cG,g)=\{x\in \bK \mid xN_g \in \cM\}.
\]

\begin{lem}\label{lem:ambient-quasiamplitude-ideal}
The $\cO_\bK$-submodule $b(\cG,g)$ is a fractional ideal of $\bK$.
\end{lem}
\begin{proof}
The $\cO_\bK$-submodule $b(\cG,g)$ is a full $\bZ$-lattice; indeed, since $\cM$ is a full $\cO_\bK$-lattice in $M_2(\bK)$, there exists a non-zero $d\in\cO_\bK$ such that $dN_g\in\cM$, so $d\cO_\bK\subset b(\cG,g)$. On the other hand, choosing a non-zero entry $\alpha$ of $N_g$ and a non-zero $c\in\cO_\bK$ with $\cM\subset c^{-1}M_2(\cO_\bK)$, the condition $xN_g\in\cM$ implies $x\alpha\in c^{-1}\cO_\bK$, hence
\[
b(\cG,g)\subset \alpha^{-1}c^{-1}\cO_\bK.
\]
Thus $b(\cG,g)$ is a fractional ideal. 
\end{proof}

The \emph{cusp width} of $\sigma$ is the index \(w_\sigma := [b(\cG, g) : b(\Gamma, g)]\). The corresponding \emph{multiplier group} is
\[
V(\Gamma, g)
:=
\left\{
a^2 \mid
\begin{pmatrix} a & t \\ 0 & a^{-1} \end{pmatrix} \in g^{-1}\Gamma g
\text{ for some } t \in \bK
\right\}
\subset \bK_+^*.
\]

Any two charts for $\sigma$ differ by right multiplication by an element of $B(\bK)$, the stabilizer of $\infty$ in $\SL_2(\bK)$. The conjugation identity
\[
\begin{pmatrix} a & t \\ 0 & a^{-1}\end{pmatrix}
u(x)
\begin{pmatrix} a & t \\ 0 & a^{-1}\end{pmatrix}^{-1}
=
u(a^2 x)
\]
has three consequences. First, $V(\Gamma, g) \subset U_{b(\Gamma, g)}^+$. Second, replacing $g$ by $g h$ with $h = \begin{pmatrix} a & t \\ 0 & a^{-1}\end{pmatrix} \in B(\bK)$ scales both $b(\Gamma, g)$ and $b(\cG, g)$ by the same factor $a^{-2} \in \bK_+^*$, so $w_\sigma$ is well defined and independent of $g$. Third, conjugation by $h$ preserves the diagonal entries of every element of $B(\bK)$, so $V(\Gamma, gh) = V(\Gamma, g)$.

We set $V_\sigma := V(\Gamma, g)$, and we denote by $M_\sigma$ the class of the lattice $b(\Gamma, g)$ up to multiplication by an element of $\bK_+^*$. Both depend only on $\sigma$. The pair $(M_\sigma, V_\sigma)$ is the \emph{cusp type} of $\sigma$, and its \emph{cusp multiplicity} is the index
\[
a_\sigma(\Gamma):=[U_{M_\sigma(\Gamma)}^+:V_\sigma(\Gamma)].
\]
We will show in Lemma~\ref{lem:cusp-multiplicity-finite} that $a_\sigma(\Gamma)$ is finite. When there is no confusion, we shall simply write $a_\sigma$ for
$a_\sigma(\Gamma)$.

\begin{lem}\label{lem:ambient-multiplier-square}
We have $V(\cG,g)=(\mathcal O_\bK^*)^2$.
\end{lem}
\begin{proof}
Let $L$ be a rank-two projective $\cO_\bK$-module such that $\cG=\SL(L)$. Choose a \(\bK\)-linear identification \(L_\bK:=L\otimes_{\cO_\bK}\bK\simeq \bK^2\). Under the induced identification \(\bP(L_\bK)=\bP^1(\bK)\), let \(\ell\subset L_\bK\) be the \(\bK\)-line represented by the cusp \(\xi\). Put
\(I:=L\cap \ell\). Then \(I\) is a saturated rank-one
\(\cO_\bK\)-submodule of \(L\). Since over a Dedekind domain every finitely generated torsion-free module is projective, both \(I\) and \(L/I\) are projective \(\cO_\bK\)-modules of rank one. Note that, since \(L/I\) is projective, the
sequence \(0\to I\to L\to L/I\to 0\) splits.

Let \(P_\ell\subset\cG\) be the stabilizer of \(\ell\).
Every \(p\in P_\ell\) preserves \(I=L\cap\ell\), and therefore acts on
\(I\) by multiplication by some \(u\in\cO_\bK^*\). It also acts on
\(L/I\) by multiplication by some \(v\in\cO_\bK^*\). The determinant
condition gives \(v=u^{-1}\). After choosing a splitting
\(L\simeq I\oplus L'\), with \(L'\simeq L/I\), the \(\bK\)-linear extension
of \(p\) has the form
\[
  \begin{pmatrix} u&*\\0&u^{-1}\end{pmatrix}
\]
with respect to the basis adapted to \(\ell\). Thus \(V(\cG,g)\subset(\cO_\bK^*)^2\).

Conversely, let \(u\in\cO_\bK^*\). Choose a splitting
\(L\simeq I\oplus L'\). Multiplying by \(u\) on \(I\) and by
\(u^{-1}\) on \(L'\) gives an automorphism in \(\SL(L)\). Hence \((\cO_\bK^*)^2\subset V(\cG,g)\). 
\end{proof}

We use the same notations for the projective group $\Lambda$ as follows. Let $\xi\in\bP^1(\bK)$ represent a $\Lambda$-cusp, and choose $g\in\SL_2(\bK)$ with $g(\infty)=\xi$. Define
\[
b(\Lambda,g)
:=
\{x\in\bK\mid\Proj(g u(x)g^{-1})\in\Lambda\}
\]
and
\begin{align*}
V(\Lambda,g)
& :=
\left\{
\alpha\delta^{-1}\in\bK^*\mid \exists \beta\in \bK,
\Proj\begin{pmatrix}\alpha&\beta\\0&\delta\end{pmatrix}
\in \Proj(g)^{-1}\Stab_\Lambda(\xi)\Proj(g)
\right\}
\subset \bK_+^*.
\end{align*}
Here $V(\Lambda,g)\subset \bK_+^*$ because we have assumed $\Lambda\subset \PGL_2^+(\bK)$. 

The same conjugation argument as above shows that the class
$M_\xi(\Lambda)$ of $b(\Lambda,g)$ up to multiplication by an element of
$\bK_+^*$, and the group $V(\Lambda,g)$, are independent of
the cusp representative $\xi$ and the choice of $g$. We define the cusp type similarly and the cusp multiplicity
\[
a_\xi(\Lambda):=[U_{M_\xi(\Lambda)}^+:V_\xi(\Lambda)].
\]
For $\Lambda_\cM=\Proj\Gamma$, the projective versions coincide with the previous ones: if $\sigma$ is the $\Lambda_\cM$-orbit of $g(\infty)$, then
\[
b(\Lambda_\cM,g)=b(\Gamma,g),\qquad
V(\Lambda_\cM,g)=V(\Gamma,g),\qquad
a_\sigma(\Lambda_\cM)=a_\sigma .
\]

The cusp width $w_\sigma$ and the cusp multiplicity $a_\sigma$ capture different parts of the cusp stabilizer: $w_\sigma$ its unipotent part, $a_\sigma$ its diagonal part.

\begin{lem}
\label{lem:cusp-multiplicity-finite}
The cusp multiplicity $a_\sigma = [U_{M_\sigma}^+ : V_\sigma]$ is finite.
\end{lem}

\begin{proof}
We show that both $U_{M_\sigma}^+$ and $V_\sigma$ are commensurable with $\cO_\bK^* \cap \bK_+^*$. Combined with $V_\sigma \subset U_{M_\sigma}^+$, this implies $[U_{M_\sigma}^+ : V_\sigma] < \infty$.

We start with $U_{M_\sigma}^+$. Pick a representative $M$ of $M_\sigma$ and set
\[
\cO_M := \{\alpha \in \bK \mid \alpha M \subset M\}.
\]
Then $\cO_M$ is a subring of $\bK$ containing $\bZ$. For any non-zero $m \in M$ one has $\cO_M \subset m^{-1} M$, so $\cO_M$ is a finitely generated $\bZ$-module; in particular, for every $\alpha \in \cO_M$, the subring $\bZ[\alpha] \subset \cO_M$ is finitely generated over $\bZ$, so $\alpha$ is integral over $\bZ$. Hence $\cO_M \subset \cO_\bK$. Moreover $\cO_M$ has full rank $n = [\bK : \bQ]$ as a $\bZ$-module. Indeed, fixing a $\bZ$-basis $m_1, \ldots, m_n$ of $M$ --- which is also a $\bQ$-basis of $\bK$, since $M$ is a full $\bZ$-lattice --- for any $\alpha \in \bK$ we may write $\alpha m_i = \sum_j q_{ij} m_j$ with $q_{ij} \in \bQ$, and clearing denominators by some non-zero $c \in \bZ$ gives $c\alpha \cdot m_i \in M$ for all $i$, hence $c\alpha \in \cO_M$. Thus every element of $\bK$ has a non-zero $\bZ$-multiple in $\cO_M$, so $\cO_M \otimes_\bZ \bQ = \bK$. Hence $\cO_M$ is an order in $\bK$, so by Dirichlet's unit theorem $\cO_M^*$ has finite index in $\cO_\bK^*$. Intersecting with $\bK_+^*$ and using $U_M^+ = \cO_M^* \cap \bK_+^*$ gives the commensurability of $U_{M_\sigma}^+$ with $\cO_\bK^* \cap \bK_+^*$.

We now turn to $V_\sigma$. Put $\cM_g:=g^{-1}\cM g$ and $H_g:=\SL(\cM_g)\cap B(\bK)$. Let
\[
\mu:H_g\to \bK_+^*
\]
be the homomorphism sending $\begin{pmatrix} a & t \\ 0 & a^{-1} \end{pmatrix}$ to $a^2$. We first reduce the claim to the corresponding statement for $H_g$. Since $g^{-1}\Gamma g$ has finite index in $g^{-1}\cG g=\SL(\cM_g)$, the subgroup $(g^{-1}\Gamma g)\cap B(\bK)$ has finite index in $H_g$. Therefore its image under $\mu$ has finite index in $\mu(H_g)$. But this image is exactly $V_\sigma$. It is therefore enough to prove that $\mu(H_g)$ is commensurable with $\cO_\bK^*\cap\bK_+^*$.

We prove this by sandwiching $\mu(H_g)$ between two finite-index subgroups of $\cO_\bK^*\cap\bK_+^*$. Since $\cM_g$ is a full $\cO_\bK$-lattice in $M_2(\bK)$, there exists a non-zero $c\in\cO_\bK$ such that $\cM_g\subset c^{-1}M_2(\cO_\bK)$. Hence, if $\begin{pmatrix} a & t \\ 0 & a^{-1}\end{pmatrix}\in H_g$, then $a,a^{-1}\in c^{-1}\cO_\bK$. Thus the valuations $v_\fp(a)$ are bounded independently of $a$. Since the set of such $a$ is a subgroup of $\bK^*$, each $v_\fp$ vanishes on it. Hence $a\in\cO_\bK^*$, and so $\mu(H_g)\subset(\cO_\bK^*)^2\subset \cO_\bK^*\cap\bK_+^*$.

Conversely, since $\cM_g$ and $M_2(\cO_\bK)$ are commensurable, there exists a non-zero ideal $\fa\subset\cO_\bK$ such that $\fa M_2(\cO_\bK)\subset\cM_g$. If $\varepsilon\in\cO_\bK^*$ satisfies $\varepsilon\equiv1\pmod{\fa}$, then also $\varepsilon^{-1}\equiv1\pmod{\fa}$, hence
\[
\begin{pmatrix} \varepsilon & 0 \\ 0 & \varepsilon^{-1} \end{pmatrix}
\in I+\fa M_2(\cO_\bK)\subset \cM_g.
\]
This matrix belongs to $H_g$, so $\varepsilon^2\in\mu(H_g)$. The group of units congruent to $1$ modulo $\fa$ has finite index in $\cO_\bK^*$, and its square therefore has finite index in $\cO_\bK^*\cap\bK_+^*$. Thus $\mu(H_g)$ contains a finite-index subgroup of $\cO_\bK^*\cap\bK_+^*$, while the previous paragraph showed that it is contained in $\cO_\bK^*\cap\bK_+^*$. Hence $\mu(H_g)$ is commensurable with $\cO_\bK^*\cap\bK_+^*$. Since $V_\sigma$ has finite index in $\mu(H_g)$, the same is true for $V_\sigma$.
\end{proof}

\subsection{Comparison of cusp multiplicities for \texorpdfstring{$\Lambda$}{Lambda} and \texorpdfstring{$\Gamma$}{Gamma}}
\label{sec:lambda-gamma-cusp-multiplicity-comparison}

Recall that $\Lambda$ is the projective arithmetic group fixed in Subsection~\ref{sec:global-arithmetic-setting}. Set
\[
d_\Lambda:=[\Lambda:\Lambda_\cM].
\]
The inclusion $\Lambda_\cM\subset\Lambda$ induces a finite map
\[
\Gamma\backslash\bP^1(\bK)
=
\Lambda_\cM\backslash\bP^1(\bK)
\longrightarrow
\Lambda\backslash\bP^1(\bK).
\]

\begin{prop}\label{prop:lambda-cusp-multiplicity-controlled-by-gamma}
One has
\[
\sum_{\kappa\in\Lambda\backslash\bP^1(\bK)}a_\kappa(\Lambda)
\leq
d_\Lambda^{2([\bK:\bQ]-1)}
\sum_{\sigma\in\Gamma\backslash\bP^1(\bK)}a_\sigma(\Gamma).
\]
\end{prop}
\begin{proof}
Put $n=[\bK:\bQ]$. Let \(\xi\in\Lambda\backslash\bP^1(\bK)\).  Choose a lift
\(\widetilde\xi\in\Lambda_\cM\backslash\bP^1(\bK)\), and choose
\(g\in\SL_2(\bK)\) with \(g(\infty)\) representing this lift. Write
\[
\mathfrak b_\Lambda:=b(\Lambda,g),\quad
\mathfrak b_\Gamma:=b(\Gamma,g),
\quad
V_\Lambda:=V(\Lambda,g),\quad
V_\Gamma:=V(\Gamma,g).
\]
Then $\mathfrak b_\Gamma\subset\mathfrak b_\Lambda$ and
$V_\Gamma\subset V_\Lambda$. If $A_\Lambda$ and $A_\Gamma$ denote the
affine-action images of the two projective stabilizers in the chart $g$, then
$A_\Gamma\subset A_\Lambda$ and $[A_\Lambda:A_\Gamma]\leq d_\Lambda$.
Their translation subgroups are respectively $\mathfrak b_\Lambda$ and
$\mathfrak b_\Gamma$, hence
\[
m:=[\mathfrak b_\Lambda:\mathfrak b_\Gamma]\leq d_\Lambda.
\]

Write $U_\Lambda:=U_{\mathfrak b_\Lambda}^+$ and
$U_\Gamma:=U_{\mathfrak b_\Gamma}^+$. The group $U_\Lambda$ acts on the set of
index-$m$ sublattices of $\mathfrak b_\Lambda$. The stabilizer of
$\mathfrak b_\Gamma$ is $W:=U_\Lambda\cap U_\Gamma$, so
$[U_\Lambda:W]\leq s_n(m)$, where $s_n(m)$ denotes the number of index-$m$
sublattices of $\bZ^n$. Since $V_\Gamma\subset V_\Lambda\cap W$,
\[
[V_\Lambda W:V_\Lambda]\leq [W:V_\Gamma]\leq [U_\Gamma:V_\Gamma]
=a_{\widetilde \xi}(\Gamma).
\]
Therefore
\[
a_{\widetilde \xi}(\Lambda)
=[U_\Lambda:V_\Lambda]
\leq [U_\Lambda:W]\,[V_\Lambda W:V_\Lambda]
\leq s_n(m)a_{\widetilde \xi}(\Gamma).
\]

It remains to estimate \(s_n(m)\). Choose a \(\mathbb Z\)-basis of
\(\mathfrak b_\Lambda\). An index-\(m\) sublattice \(L\subset\mathfrak b_\Lambda\)
is represented by an integral matrix whose columns form a basis of \(L\).
Changing the basis of \(L\) amounts exactly to right multiplication by an
element of \(\GL_n(\mathbb Z)\). Thus sublattices of index \(m\) are the same
thing as right \(\GL_n(\mathbb Z)\)-equivalence classes of integral matrices
of determinant \(\pm m\).

By column Hermite normal form, each such class has a unique triangular
representative \(R=(r_{ij})\), with positive diagonal entries
\(d_i=r_{ii}\), determinant \(\prod_i d_i=m\), and with the remaining
triangular entries reduced modulo the corresponding diagonal entries. In
particular, for fixed \(d_1,\ldots,d_n\), the number of possible reduced
triangular entries is at most
\[
\prod_{i=1}^n d_i^{\,n-1}\leq m^{n-1}.
\]
The number of possible diagonal tuples is at most \(m^{n-1}\), since
\(d_1,\ldots,d_{n-1}\) determine \(d_n\), and each \(d_i\) is a positive
divisor of \(m\). Hence \(s_n(m)\leq m^{2(n-1)}\).
\end{proof}

We shall also need a bound for \(d_\Lambda\). Recall that, in
Subsection~\ref{sec:global-arithmetic-setting}, the group \(\Lambda\) was
chosen inside the projective normalizer \(\PN(\cE)\) of an Eichler order
\(\cE\subset M_2(\bK)\). The maximal order \(\cM\supset\cE\) was then used to
define \(\cG=\SL(\cM)\) and \(\Lambda_\cM=\Lambda\cap\P\cG\). Thus
\(d_\Lambda=[\Lambda:\Lambda_\cM]\) is controlled by the relative position of
the projective normalizer of \(\cE\) with respect to \(\P\cG\).

\begin{prop}\label{prop:d-lambda-level-prime-bound}
There exists a constant \(C_{\mathrm{AL}}(\bK)>0\), depending only on \(\bK\),
such that
\[
d_\Lambda\leq C_{\mathrm{AL}}(\bK)\,2^{\omega(\fn)}.
\]
\end{prop}

\begin{proof}
Let \(S_\cE\) denote the set of prime ideals \(\fp\) for which
\(\cE_\fp\neq\cM_\fp\). We first prove that
\(
S_\cE\subset\{\fp\mid \fp\vert\fn\}
\).
Suppose that \(\fp\nmid\fn\). By
Lemma~\ref{lem:congruence-product-maximal-order},
\(\cG(\fn)\cG(\fp)=\cG\), so the image of \(\cG(\fn)\) in
\(\cG/\cG(\fp)\simeq \SL(L/\fp L)\) is the whole group. Since
\[
\cG(\fn)\subset\Gamma\subset\operatorname N(\cE)\cap\cG,
\]
the reduction modulo \(\fp\) of \(\operatorname N(\cE)\cap\cG\) is all of
\(\SL(L/\fp L)\). This is impossible if \(\cE_\fp\) is non-maximal. Indeed,
after localizing at \(\fp\) and choosing a basis, \(\cE_\fp\) is a standard
Eichler order of positive level, and the reduction of
\(\operatorname N(\cE_\fp)\cap\SL(\cM_\fp)\) preserves the corresponding line
in \(L/\fp L\). Hence its image is contained in a Borel subgroup of
\(\SL(L/\fp L)\), a contradiction. Thus \(\cE_\fp=\cM_\fp\), proving the
inclusion.

Since \(\Lambda\subset\PN(\cE)\), the natural map
\[
\Lambda/(\Lambda\cap\P\cG)
\longrightarrow
\PN(\cE)/(\PN(\cE)\cap\P\cG)
\]
is injective. It is enough to bound the quotient on the right.

Set
\(
Q:=\PN(\cE)/(\PN(\cE)\cap\P\cG)
\), and for each \(\fp\in S_\cE\),
\[
\mathcal Q_\fp:=
\operatorname N(\cE_\fp)\big/
\bigl(\operatorname N(\cE_\fp)\cap \bK_\fp^*\cM_\fp^*\bigr),
\]
viewed as a finite set of right cosets. Localization gives a well-defined map
\[
\ell:Q\longrightarrow \prod_{\fp\in S_\cE}\mathcal Q_\fp .
\]
Indeed, multiplying a representative by an element of \(\bK^*\) does not
change its local cosets, and an element of \(\PN(\cE)\cap\P\cG\) is locally
represented by an element of
\(\operatorname N(\cE_\fp)\cap\bK_\fp^*\cM_\fp^*\).

By the local form of Eichler orders and the local normalizer theorem
\cite[Definition~23.4.11 and Proposition~23.4.14]{Voight},
the quotient \(\operatorname N(\cE_\fp)/\bK_\fp^*\cE_\fp^*\) has order \(2\)
when \(\cE_\fp\) is non-maximal, and its non-trivial class is the local
Atkin--Lehner involution. Since
\[
\bK_\fp^*\cE_\fp^*
\subset
\operatorname N(\cE_\fp)\cap\bK_\fp^*\cM_\fp^*,
\]
each \(\mathcal Q_\fp\) has at most two elements. Hence
\begin{equation}\label{eq:AL-image-l-bound}
\lvert\operatorname{im}\ell\rvert\leq 2^{|S_\cE|}.
\end{equation}

We now bound the fibers of \(\ell\).  Fix an element
\(q_0\in\operatorname{im}\ell\), and choose \(h_0\in\operatorname N(\cE)\)
representing a class in \(Q\) whose image is \(q_0\).  Let
\(h\in\operatorname N(\cE)\) represent another class in the same fiber.  We
claim first that
\[
  h_0^{-1}h\in \bK_\fp^*\cM_\fp^*
  \qquad\text{for every finite prime }\fp .
\]
For \(\fp\in S_\cE\), this follows from the equality of the local right
cosets of \(h_0\) and \(h\) in \(\mathcal Q_\fp\).  For
\(\fp\notin S_\cE\), one has \(\cE_\fp=\cM_\fp\), and the normalizer of a
maximal order in \(M_2(\bK_\fp)\) is \(\bK_\fp^*\cM_\fp^*\).  This proves the
claim.

For such an \(h\), write locally
\[
  h_0^{-1}h=\lambda_\fp m_\fp,
  \qquad
  \lambda_\fp\in\bK_\fp^*,\quad m_\fp\in\cM_\fp^* .
\]
Since \(\det(m_\fp)\in\cO_{\bK,\fp}^*\), we get
\[
  v_\fp(\det(h_0^{-1}h))=2v_\fp(\lambda_\fp).
\]
Thus \(v_\fp(\det(h_0^{-1}h))\) is even for every finite prime \(\fp\).

We shall use the determinant only modulo squares.  Let
\[
D_\bK:=
\left\{
\alpha(\bK^*)^2\in \bK^*/(\bK^*)^2
\ \middle|\ 
v_\fp(\alpha)\equiv 0 \pmod 2,\;
\forall \fp
\right\}.
\]
This condition is independent of the representative \(\alpha\), because
multiplying \(\alpha\) by a square changes each valuation by an even integer.
The preceding paragraph shows that \(\det(h_0^{-1}h)\) defines an element of
\(D_\bK\).

For the fixed fiber over \(q_0\), we therefore have a map
\begin{equation}\label{eq:AL-fiber-l-bound}
\ell^{-1}(q_0)\longrightarrow D_\bK,\qquad
\bar h\longmapsto \det(h_0^{-1}h)(\bK^*)^2 .
\end{equation}
This map is well-defined.  Multiplying \(h\) by a scalar changes the
determinant by a square.  Moreover, if we multiply the projective class of
\(h\) on the right by an element of \(\PN(\cE)\cap\P\cG\), then this element
may be represented by \(\lambda g\), with \(\lambda\in\bK^*\) and
\(g\in\cG=\SL(\cM)\), so its determinant is \(\lambda^2\).

We next prove that this map is injective.  Let
\(h_1,h_2\in\operatorname N(\cE)\) represent two classes in the fiber over
\(q_0\), and suppose that
\[
  \det(h_0^{-1}h_1)(\bK^*)^2
  =
  \det(h_0^{-1}h_2)(\bK^*)^2 .
\]
Then
\[
  \det(h_2^{-1}h_1)=a^2
\]
for some \(a\in\bK^*\).  Put \(h'=h_2^{-1}h_1\).  Since both
\(h_0^{-1}h_1\) and \(h_0^{-1}h_2\) belong locally to
\(\bK_\fp^*\cM_\fp^*\), so does \(h'\).  Thus, for every finite \(\fp\), we
may write
\[
  h'=\lambda_\fp m_\fp,
  \qquad
  \lambda_\fp\in\bK_\fp^*,\quad m_\fp\in\cM_\fp^* .
\]
From \(\det h'=a^2\), we get
\(v_\fp(a)=v_\fp(\lambda_\fp)\).  Hence
\(a^{-1}\lambda_\fp\in\cO_{\bK,\fp}^*\), and therefore
\[
  a^{-1}h'=(a^{-1}\lambda_\fp)m_\fp\in\cM_\fp^*
  \qquad\text{for every finite }\fp .
\]
It follows that \(a^{-1}h'\in\cM^*\).  Since
\(\det(a^{-1}h')=1\), we have \(a^{-1}h'\in\SL(\cM)=\cG\).  Also
\(h'\in\operatorname N(\cE)\), so
\(a^{-1}h'\in\operatorname N(\cE)\cap\cG\).  Therefore the projective classes
of \(h_1\) and \(h_2\) differ by an element of
\(\PN(\cE)\cap\P\cG\), and they define the same element of \(Q\).  This proves
the injectivity.

It remains to count \(D_\bK\). For \(\alpha(\bK^*)^2\in D_\bK\), all valuations \(v_\fp(\alpha)\) are even.
Hence the principal fractional ideal \((\alpha)\) is the square of a fractional
ideal \(\fa\). Since \((\alpha)\) is principal, the class of \(\fa\) belongs to the \(2\)-torsion subgroup of the
ideal class group. Conversely, once this ideal class is
fixed, the remaining ambiguity is multiplication by a unit, modulo squares.
There is a natural homomorphism
\(
D_\bK\to \Cl(\bK)[2]
\)
whose kernel is \(\cO_\bK^*/(\cO_\bK^*)^2\). In particular,
\begin{equation}\label{eq:AL-fiber-l-bound-2}
\lvert D_\bK\rvert\leq
\lvert\cO_\bK^*/(\cO_\bK^*)^2\rvert\,\lvert\Cl(\bK)[2]\rvert.
\end{equation}
Consequently each fiber of \(\ell\) has cardinality at most this constant. 

Combining the bound \eqref{eq:AL-image-l-bound} for the image of \(\ell\) with the fiber bound \eqref{eq:AL-fiber-l-bound}, \eqref{eq:AL-fiber-l-bound-2} gives
\[
|Q|
\leq
|\cO_\bK^*/(\cO_\bK^*)^2|\,|\Cl(\bK)[2]|\,2^{|S_\cE|}.
\]
Since \(S_\cE\subset\{\fp\mid\fp\vert\fn\}\), we have $|S_\cE|\leq \omega(\fn)$. Thus
\[
d_\Lambda\leq |Q|
\leq
|\cO_\bK^*/(\cO_\bK^*)^2|\,|\Cl(\bK)[2]|\,2^{\omega(\fn)}.
\]
Thus the desired inequlity is proved with 
\[
C_{\mathrm{AL}}(\bK)
:=
\lvert\cO_\bK^*/(\cO_\bK^*)^2\rvert\,\lvert\Cl(\bK)[2]\rvert.
\]
\end{proof}

\subsection{Examples: some classical congruence subgroups}\label{sec:examples-cusp-multiplicity}
For a non-zero ideal \(\fn\), consider the following classical congruence subgroups
\begin{align*}
\Gamma(\fn)
&=
\{g\in\SL_2(\cO_\bK)\mid g\equiv I\pmod{\fn}\},\\
\Gamma_1(\fn)
&=
\left\{
\begin{pmatrix}a&b\\ c&d\end{pmatrix}\in\SL_2(\cO_\bK)\mid 
a\equiv d\equiv 1,\ c\equiv 0\pmod{\fn}
\right\},\\
\Gamma_0(\fn)
&=
\left\{
\begin{pmatrix}a&b\\ c&d\end{pmatrix}\in\SL_2(\cO_\bK)\mid
c\equiv 0\pmod{\fn}
\right\}.
\end{align*}
These three congruence subgroups of level \(\fn\) are often studied in the literature. The quantity \(\sum a_\sigma\) can sometimes be computed directly for these special subgroups. We gather a few such formulas for illustration. We will see that for these examples the ratio $\sum a_\sigma/[\cG:\Gamma]$ goes to zero as the congruence ideal grows. Later in this section we will prove this phenomenon for arbitrary congruence subgroups.  

\subsubsection{Rational numbers}
Let \(\bK=\mathbb Q\). Then \(\cO_\bK^*\cap\bK_+^*=\{1\}\). Hence the
diagonal unit group which enters the definition of \(a_\sigma\) is trivial,
and \(a_\sigma=1\) for every cusp. Thus the sum of cusp multiplicities is the
number of cusps. Denote \(\varphi(m)=|(\mathbb Z/m\mathbb Z)^*|\) the Euler function. We have (see \cite[\S 3]{DiamondShurman}), for \(N>4\), 
\begin{align*}
\sum_\sigma a_\sigma(\Gamma(N))
&=
\frac12 N^2\prod_{p\mid N}(1-p^{-2}), 
& \, 
\frac{\sum_\sigma a_\sigma(\Gamma(N))}
{[\SL_2(\mathbb Z):\Gamma(N)]}
&=
\frac1N,
\\
\sum_\sigma a_\sigma(\Gamma_1(N))
&=
\frac12\sum_{d\mid N}\varphi(d)\varphi(N/d),
&\,
\frac{\sum_\sigma a_\sigma(\Gamma_1(N))}
{[\SL_2(\mathbb Z):\Gamma_1(N)]}
&=
\frac{\sum_{d\vert N}\varphi(d)\varphi(N/d)}
{N^2\prod_{p\vert N}(1-p^{-2})},
\\
\sum_\sigma a_\sigma(\Gamma_0(N))
&=
\sum_{d\mid N}\varphi\bigl(\operatorname{gcd}(d,N/d)\bigr),
&\,
\frac{\sum_\sigma a_\sigma(\Gamma_0(N))}
{[\SL_2(\mathbb Z):\Gamma_0(N)]}
&=
\frac{\sum_{d\vert N}\varphi(\operatorname{gcd}(d,N/d))}
{N\prod_{p\vert N}(1+p^{-1})}.
\end{align*}

The fact that the three ratios tend to \(0\) as \(N\to\infty\) follows directly from the formulas. Indeed, use
\(\sum_{d\vert N}\varphi(d)\varphi(N/d)\le N\,\#\{d\mid d\vert N\}\) and
\(\sum_{d\vert N}\varphi(\operatorname{gcd}(d,N/d))\le
N^{1/2}\#\{d\mid d\vert N\}\).

For arbitrary congruence subgroups over \(\bQ\), the corresponding
sublinearity of the cusp contribution is  controlled by the parabolic
estimates of Cox--Parry~\cite{CoxParry1984}, which is a special case of our main result. 

\subsubsection{Imaginary quadratic fields}
Let \(\bK\) be imaginary quadratic. Since the unit group is finite, the
multiplicities \(a_\sigma\) are bounded above by a constant depending only on
\(\bK\). Hence \(\sum_\sigma a_\sigma\) and the number of cusps are comparable
up to a constant depending only on \(\bK\). For $\PSL_2(\cO_\bK)$ the number of cusps is exactly the class number \(h_\bK\) of $\bK$.

Assume now that \(h_\bK=1\). For the principal congruence subgroup one has
\[
\sum_\sigma 1
=
\frac{|\{u\in\cO_\bK^*\mid u\equiv 1\pmod{\fn}\}|}{|\cO_\bK^*|}
N(\fn)^2
\prod_{\fp\mid\fn}\bigl(1-N(\fp)^{-2}\bigr),
\]
see \cite[Theorem~2.6]{BakerReidPrincipal}. For \(\Gamma_0(\fn)\), 
\[
\sum_\sigma 1
=
\sum_{\fd\vert\fn}
\left| \bigl(\cO_\bK/(\fd,\fn\fd^{-1})\bigr)^* \right|,
\]
see \cite[Theorem~7]{CremonaAranes}. In particular, for square-free \(\fn\),
\(
\sum_\sigma 1=2^{\omega(\fn)}
\), so
\[
\frac{\sum_\sigma 1}
{[\PSL_2(\cO_\bK):\Gamma_0(\fn)]}
=
\prod_{\fp\mid\fn}\frac{2}{N(\fp)+1},
\]
which tends to \(0\) as \(N(\fn)\to\infty\).

\subsubsection{Totally real fields}
Let \(\bK\) be totally real of degree \(2\). Put \(U^+=\cO_\bK^*\cap\bK_+^*\) and \(U_\fn=\{u\in \cO_\bK^*\mid u\equiv 1\pmod{\fn}\}\). For the principal congruence subgroup \(\Gamma(\fn)\), by \cite[Lemma~IV.5.2]{VDG} the number of $\Gamma(\fn)$-cusps above each cusp of the full Hilbert modular group $\SL_2(\cO_\bK)$ is the same number
\[
\frac{N(\fn)^2\prod_{\fp\mid\fn}(1-N(\fp)^{-2})}{[\cO_\bK^*:U_\fn]}.
\]
Moreover
\(
a_\sigma(\Gamma(\fn))=[U^+:U_\fn^2]
\)
for every cusp \(\sigma\) of \(\Gamma(\fn)\). 
Since
\begin{align*}
[U^+:U_\fn^2]
&=
[U^+:(\cO_\bK^*)^2]\,[(\cO_\bK^*)^2:U_\fn^2],\\
[(\cO_\bK^*)^2:U_\fn^2]
&= a_1[\cO_\bK^*:U_\fn], \quad a_1=1 \,\text{or}\, \frac{1}{2},\\
[\SL_2(\cO_\bK):\Gamma(\fn)]
&=a_2 N(\fn)^3\prod_{\fp\mid\fn}(1-N(\fp)^{-2}),\; a_2 =2,1 \,\text{or}\, \frac{1}{2},
\end{align*}
we obtain, for some uniformly bounded constant $a$ (actually $a=1,2$),
\[
\frac{\sum_\sigma a_\sigma(\Gamma(\fn))}
{[\SL_2(\cO_\bK):\Gamma(\fn)]}
=
a\frac{h_\bK[U^+:(\cO_\bK^*)^2]}{N(\fn)}.
\]
In particular this ratio tends to \(0\) as \(N(\fn)\to\infty\).

\subsection{Stabilizer-index decomposition above an ambient cusp}
\label{sec:stabilizer-index-decomposition}

Fix a $\cG$-cusp $\tau \in \bP^1(\bK)$, and denote by
\[
\cC_\tau(\Gamma) := \Gamma \backslash (\cG \cdot \tau)
\]
the set of $\Gamma$-cusps lying above $\cG \cdot \tau$.

\begin{lem}\label{lem:stabilizer-index-sum}
\[
[\cG : \Gamma]
=
\sum_{\sigma \in \cC_\tau(\Gamma)}
[\Stab_\cG(\sigma) : \Stab_\Gamma(\sigma)],
\]
where each summand is computed using any representative point of the $\Gamma$-cusp $\sigma$.
\end{lem}

\begin{proof}
The finite set $X := \Gamma \backslash \cG$ has size $[\cG : \Gamma]$. The group $\Stab_\cG(\tau)$ acts on $X$ on the right by $(\Gamma r) \cdot p := \Gamma r p$. We claim that the orbits of this action are in bijection with $\cC_\tau(\Gamma)$ via $(\Gamma r) \Stab_\cG(\tau) \mapsto [r\tau]_\Gamma$. Surjectivity is clear. For injectivity, suppose $[r_1 \tau]_\Gamma = [r_2 \tau]_\Gamma$. Then $\gamma r_1 \tau = r_2 \tau$ for some $\gamma \in \Gamma$, so $r_2^{-1} \gamma r_1$ fixes $\tau$ and therefore belongs to $\Stab_\cG(\tau)$. Hence $r_2 \in \Gamma r_1 \Stab_\cG(\tau)$, so $\Gamma r_1$ and $\Gamma r_2$ lie in the same $\Stab_\cG(\tau)$-orbit.

The stabilizer of $\Gamma r$ for this right action is $\Stab_\cG(\tau) \cap r^{-1} \Gamma r$, so the orbit through $\Gamma r$ has size
\begin{align*}
[\Stab_\cG(\tau) : \Stab_\cG(\tau) \cap r^{-1} \Gamma r]
&= [r \Stab_\cG(\tau) r^{-1} : \Gamma \cap r \Stab_\cG(\tau) r^{-1}] \\
&= [\Stab_\cG(r\tau) : \Stab_\Gamma(r\tau)].
\end{align*}
Summing over orbits gives the formula.
\end{proof}

Choose now a chart $g_\tau \in \SL_2(\bK)$ with $g_\tau(\infty) = \tau$. For $\sigma = [r\tau]_\Gamma \in \cC_\tau(\Gamma)$, set $g_\sigma := r g_\tau$, so that $g_\sigma(\infty) = r\tau$. For $\cH \in \{\cG, \Gamma\}$, let
\[
P_\cH(g_\sigma)
:=
\bigl(g_\sigma^{-1} \cH g_\sigma \cap B(\bK)\bigr) \big/ \{\pm I\}
\]
be the projective cusp stabilizer of $\cH$ in this chart.

\begin{lem}\label{lem:affine-cusp-sequence}
For $\cH \in \{\cG, \Gamma\}$, there is an exact sequence
\[
1 \longrightarrow b(\cH, g_\sigma) \longrightarrow P_\cH(g_\sigma) \longrightarrow V(\cH, g_\sigma) \longrightarrow 1.
\]
\end{lem}

\begin{proof}
Let $H_\sigma:=g_\sigma^{-1}\cH g_\sigma\cap B(\bK)$. Since $-I\in\cH$, we have $-I\in H_\sigma$. Define
\[
\iota:b(\cH,g_\sigma)\longrightarrow P_\cH(g_\sigma),
\qquad
x\longmapsto [u(x)].
\]
This is an injective homomorphism, because if $[u(x)]$ is trivial in $P_\cH(g_\sigma)$, then $u(x)\in\{\pm I\}$, hence $x=0$.
Next define
\[
\mu:P_\cH(g_\sigma)\longrightarrow V(\cH,g_\sigma),
\qquad
\left[
\begin{pmatrix}a&t\\0&a^{-1}\end{pmatrix}
\right]
\longmapsto a^2.
\]
This is well defined, because replacing a representative by its negative does not change $a^2$. It is a homomorphism, and it is surjective by the definition of $V(\cH,g_\sigma)$.

It remains to identify the kernel. If an element of $P_\cH(g_\sigma)$ represented by $\begin{pmatrix}a&t\\0&a^{-1}\end{pmatrix}$ lies in $\ker\mu$, then $a^2=1$, so $a=\pm1$. If $a=1$, the representative is $u(t)$. If $a=-1$, then
\[
\begin{pmatrix}-1&t\\0&-1\end{pmatrix}
=
-I\cdot u(-t),
\]
so it has the same class in $P_\cH(g_\sigma)$ as $u(-t)$. In both cases the class belongs to the image of $\iota$. Conversely every $u(x)$ maps to $1$ under $\mu$. Hence $\ker\mu=\operatorname{im}\iota$, proving exactness.
\end{proof}

\begin{lem}\label{lem:stabilizer-index-product}
For every $\sigma = [r\tau]_\Gamma \in \cC_\tau(\Gamma)$,
\[
[\Stab_\cG(r\tau) : \Stab_\Gamma(r\tau)]
=
[b(\cG, g_\sigma) : b(\Gamma, g_\sigma)] \cdot [V(\cG, g_\sigma) : V(\Gamma, g_\sigma)].
\]
\end{lem}

\begin{proof}
Conjugating by $g_\sigma$ and quotienting by $\{\pm I\} \subset \Gamma$ gives
\[
[\Stab_\cG(r\tau) : \Stab_\Gamma(r\tau)] = [P_\cG(g_\sigma) : P_\Gamma(g_\sigma)].
\]
The lemma then follows from Lemma~\ref{lem:affine-cusp-sequence} applied to $\cG$ and $\Gamma$.
\end{proof}

Set
\[
w_{\sigma|\tau} := [b(\cG, g_\sigma) : b(\Gamma, g_\sigma)],
\qquad
u_{\sigma|\tau} := [V(\cG, g_\sigma) : V(\Gamma, g_\sigma)].
\]
Both indices depend only on $\sigma$ and $\tau$, by an argument analogous to the one in \S\ref{sec:cusps-hirzebruch-multiplicity}.

\begin{prop}\label{prop:cusp-width-multiplier-index-sum}
\[
[\cG : \Gamma]
=
\sum_{\sigma \in \cC_\tau(\Gamma)} w_{\sigma|\tau} u_{\sigma|\tau}.
\]
\end{prop}

\begin{proof}
Combine Lemmas~\ref{lem:stabilizer-index-sum} and~\ref{lem:stabilizer-index-product}.
\end{proof}

\subsection{Cusp multiplicities and relative multiplier indices}
\label{sec:cusp-multiplicities-and-relative-multiplier-indices}

In general $u_{\sigma|\tau} \neq a_\sigma$, but $a_\sigma$ is controlled by $u_{\sigma|\tau}$ up to a constant depending only on the ambient cusp $\tau$.

Fix a $\cG$-cusp $\tau \in \bP^1(\bK)$ and keep the notation of \S\ref{sec:stabilizer-index-decomposition}; in particular, for $\sigma = [r\tau]_\Gamma \in \cC_\tau(\Gamma)$ we use the chart $g_\sigma = r g_\tau$. Remark that $U_{b(\cG, g_\sigma)}^+ = \cO_\bK^* \cap \bK_+^*$ since $b(\cG, g_\sigma)$ is an $\cO_\bK$-fractional ideal by Lemma~\ref{lem:ambient-quasiamplitude-ideal}. Recall that $V(\cG, g_\sigma)=(\cO_\bK^*)^2$ by Lemma~\ref{lem:ambient-multiplier-square}. Set
\[
C_{\rm UV}:= [U_{b(\cG, g_\sigma)}^+ : V(\cG, g_\sigma)]=[\cO_\bK^* \cap \bK_+^* : (\cO_\bK^*)^2].
\]
This is a finite constant depending only on $\bK$. The inclusion \(U_{M_\sigma}^+ \subset U_{b(\cG,g_\sigma)}^+\) holds as follows. If \(\varepsilon\in U_{M_\sigma}^+\), then \(\varepsilon M=M\) for some representative \(M\) of \(M_\sigma\).  Hence
both \(\varepsilon\) and \(\varepsilon^{-1}\) preserve the finitely generated
\(\bZ\)-module \(M\).  They are therefore integral over \(\bZ\), and so
\(\varepsilon\in\cO_\bK^*\cap\bK_+^*=U_{b(\cG,g_\sigma)}^+\).

\begin{lem}\label{lem:cusp-multiplicity-bounded-by-relative-multiplier}
For every $\sigma \in \cC_\tau(\Gamma)$,
\(
a_\sigma \leq C_{\rm UV} \cdot u_{\sigma|\tau}
\).
\end{lem}

\begin{proof}
Write $M_0=b(\cG, g_\sigma)$, $V_0=V(\cG, g_\sigma)$. Set $W := U_{M_\sigma}^+ \cap V_0$, so that $V_\sigma \subset W$ since $V_\sigma \subset U_{M_\sigma}^+$ and $V_\sigma \subset V_0$. The inclusion $U_{M_\sigma}^+ \subset U_{M_0}^+$ induces an injection
\[
U_{M_\sigma}^+ / W \hookrightarrow U_{M_0}^+ / V_0,
\]
hence $[U_{M_\sigma}^+ : W] \leq [U_{M_0}^+ : V_0] = C_{\rm UV}$. Multiplicativity of indices gives
\[
a_\sigma = [U_{M_\sigma}^+ : V_\sigma] = [U_{M_\sigma}^+ : W] \cdot [W : V_\sigma] \leq C_{\rm UV} \cdot [V_0 : V_\sigma] = C_{\rm UV} \cdot u_{\sigma|\tau}.
\]
\end{proof}

The following is well known (see \cite{VDG} for real quadratic fields):
\begin{lem}\label{lem:ambient-cusps-maximal-order}
The number of $\cG$-cusps is the class number $h_\bK$.
\end{lem}

Choose a finite set $\cT \subset \bP^1(\bK)$ of representatives for the $\cG$-cusps. By Lemma~\ref{lem:ambient-cusps-maximal-order}, one has $|\cT|= h_\bK$. 
\begin{lem}\label{lem:global-reduction-to-weighted-width}
We have
\[
\frac{\sum_{\sigma \in \Gamma \backslash \bP^1(\bK)} a_\sigma}{[\cG:\Gamma]}
\leq
C_{\rm UV}
\sum_{\tau \in \cT}
\frac{\sum_{\sigma \in \cC_\tau(\Gamma)} u_{\sigma|\tau}}{\sum_{\sigma \in \cC_\tau(\Gamma)} w_{\sigma|\tau} u_{\sigma|\tau}}.
\]
\end{lem}

\begin{proof}
The sets $\cC_\tau(\Gamma)$ for $\tau \in \cT$ partition $\Gamma \backslash \bP^1(\bK)$. By Lemma~\ref{lem:cusp-multiplicity-bounded-by-relative-multiplier},
\[
\sum_{\sigma \in \Gamma \backslash \bP^1(\bK)} a_\sigma
=
\sum_{\tau \in \cT} \sum_{\sigma \in \cC_\tau(\Gamma)} a_\sigma
\leq
C_{\rm UV} \sum_{\tau \in \cT} \sum_{\sigma \in \cC_\tau(\Gamma)} u_{\sigma|\tau}.
\]
Dividing by $[\cG : \Gamma]$ and using $[\cG:\Gamma] = \sum_{\sigma \in \cC_\tau(\Gamma)} w_{\sigma|\tau} u_{\sigma|\tau}$ for each $\tau \in \cT$ (see Proposition~\ref{prop:cusp-width-multiplier-index-sum}) gives the bound.
\end{proof}

\subsection{Reduction to the finite quotient}
\label{sec:finite-quotient-weighted-width}

Recall that $\fn \subset \cO_\bK$ denotes the congruence level of $\Gamma$ relative to $\cM$, so that $\cG(\fn) \subset \Gamma$. Write
\[
\overline \cG := \cG/\cG(\fn),
\qquad
\overline\Gamma := \Gamma/\cG(\fn).
\]
We denote by a bar the image in $\overline\cG$ of any subgroup or element of $\cG$.

Fix a $\cG$-cusp $\tau \in \bP^1(\bK)$ and write $P_\tau := \Stab_\cG(\tau)$. Choose a chart $g_\tau \in \SL_2(\bK)$ with $g_\tau(\infty)=\tau$, and set
\[
U_\tau := g_\tau\{u(x)\mid x\in\bK\}g_\tau^{-1}\cap\cG=\{g_\tau u(x)g_\tau^{-1}\mid x\in b(\cG,g_\tau)\}.
\]
Remark that $U_\tau$ is normal in $P_\tau$ and is independent of the choice of $g_\tau$.

For $r\in\cG$, set
\[
S(r,\tau):=P_\tau\cap r^{-1}\Gamma r.
\]
Its image in $\overline P_\tau$ is denoted by $\overline S(r,\tau)$. Since $\cG(\fn)\subset r^{-1}\Gamma r$, one has
\[
\overline S(r,\tau)=\overline P_\tau\cap \bar r^{-1}\overline\Gamma\bar r.
\]

\begin{lem}\label{lem:finite-quotient-cusps}
There are bijections
\[
\cC_\tau(\Gamma) \cong \Gamma\backslash \cG/P_\tau \cong \overline\Gamma\backslash\overline \cG/\overline P_\tau,
\]
the first induced by $r\mapsto [r\tau]_\Gamma$ and the second induced by the quotient map $\cG\to\overline\cG$.
\end{lem}

\begin{proof}
The first bijection is the one used in the proof of Lemma~\ref{lem:stabilizer-index-sum}. For the second, the map sends $\Gamma rP_\tau$ to $\overline\Gamma\bar r\overline P_\tau$. It is well defined and surjective. For injectivity, suppose $\bar r_2=\bar\gamma\bar r_1\bar p$ with $\gamma\in\Gamma$ and $p\in P_\tau$. Then $r_2=\gamma r_1pn$ for some $n\in\cG(\fn)$. Since $\cG(\fn)$ is normal in $\cG$, the element $r_1pnp^{-1}r_1^{-1}$ lies in $\cG(\fn)\subset\Gamma$, and
\[
r_2=\gamma(r_1pnp^{-1}r_1^{-1})r_1p.
\]
Thus $r_2\in\Gamma r_1P_\tau$.
\end{proof}

\begin{lem}\label{lem:finite-stabilizer-indices}
Let $\sigma=[r\tau]_\Gamma\in\cC_\tau(\Gamma)$. Then
\[
w_{\sigma|\tau}u_{\sigma|\tau}=[\overline P_\tau:\overline S(r,\tau)]
\qquad\text{and}\qquad
u_{\sigma|\tau}=[\overline P_\tau:\overline U_\tau\overline S(r,\tau)].
\]
\end{lem}

\begin{proof}
Conjugation by $r$ identifies $P_\tau$ with $\Stab_\cG(r\tau)$ and $S(r,\tau)$ with $\Stab_\Gamma(r\tau)$. Hence Lemma~\ref{lem:stabilizer-index-product} gives
\[
[P_\tau:S(r,\tau)]=w_{\sigma|\tau}u_{\sigma|\tau}.
\]
Since $P_\tau\cap\cG(\fn)\subset S(r,\tau)$, quotienting by $\cG(\fn)$ gives
\[
[P_\tau:S(r,\tau)]=[\overline P_\tau:\overline S(r,\tau)].
\]
This proves the first identity.

For the second identity, the map $x\mapsto g_\tau u(x)g_\tau^{-1}$ identifies $b(\cG,g_\tau)$ with $U_\tau$. Since $g_\sigma=rg_\tau$ and $r\in\cG$, one has $b(\cG,g_\sigma)=b(\cG,g_\tau)$. Under this identification, $b(\Gamma,g_\sigma)$ corresponds to $U_\tau\cap S(r,\tau)$. Therefore
\[
[U_\tau:U_\tau\cap S(r,\tau)]=w_{\sigma|\tau}.
\]
Again $U_\tau\cap\cG(\fn)\subset U_\tau\cap S(r,\tau)$. Moreover
\[
\overline U_\tau\cap \overline S(r,\tau)
=
\overline{\,U_\tau\cap S(r,\tau)\,}.
\]
Indeed, if an element of the left-hand side is represented by
$u\in U_\tau$ and by $s\in S(r,\tau)$, then $u=s n$ for some
$n\in\cG(\fn)$. Since $\cG(\fn)\subset r^{-1}\Gamma r$, this implies
$u\in r^{-1}\Gamma r$; as $u\in U_\tau\subset P_\tau$, we get
$u\in S(r,\tau)$. Hence
\[
[\overline U_\tau:\overline U_\tau\cap\overline S(r,\tau)]
=
[U_\tau:U_\tau\cap S(r,\tau)]
=
w_{\sigma|\tau}.
\]
Since $\overline U_\tau$ is normal in $\overline P_\tau$, the product
$\overline U_\tau\overline S(r,\tau)$ is a subgroup of $\overline P_\tau$,
and
\[
[\overline U_\tau\overline S(r,\tau):\overline S(r,\tau)]
=
[\overline U_\tau:\overline U_\tau\cap\overline S(r,\tau)]
=
w_{\sigma|\tau}.
\]
Combining this with the first identity gives
\[
[\overline P_\tau:\overline U_\tau\overline S(r,\tau)]
=
u_{\sigma|\tau}.
\]
\end{proof}

\begin{prop}\label{prop:finite-quotient-counts}
For every $\cG$-cusp $\tau$,
\begin{gather*}
\sum_{\sigma\in\cC_\tau(\Gamma)}w_{\sigma|\tau}u_{\sigma|\tau}
=
|\overline\Gamma\backslash\overline\cG|
=
[\cG:\Gamma],
\quad 
\sum_{\sigma\in\cC_\tau(\Gamma)}u_{\sigma|\tau}
=
|\overline\Gamma\backslash\overline\cG/\overline U_\tau|,
\\
\frac{\sum_{\sigma\in\cC_\tau(\Gamma)}u_{\sigma|\tau}}
{\sum_{\sigma\in\cC_\tau(\Gamma)}w_{\sigma|\tau}u_{\sigma|\tau}}
=
\frac{|\overline\Gamma\backslash\overline\cG/\overline U_\tau|}
{|\overline\Gamma\backslash\overline\cG|}.
\end{gather*}
\end{prop}

\begin{proof}
Let $\overline X:=\overline\Gamma\backslash\overline\cG$. The group $\overline P_\tau$ acts on $\overline X$ on the right. Its orbits are parametrized by $\overline\Gamma\backslash\overline\cG/\overline P_\tau$, hence by $\cC_\tau(\Gamma)$, by Lemma~\ref{lem:finite-quotient-cusps}. The stabilizer of $\overline\Gamma\bar r$ is $\overline S(r,\tau)$, so the corresponding $\overline P_\tau$-orbit has cardinality
\[
[\overline P_\tau:\overline S(r,\tau)]
=
w_{\sigma|\tau}u_{\sigma|\tau}
\]
by Lemma~\ref{lem:finite-stabilizer-indices}. Summing over the $\overline P_\tau$-orbits gives the first identity.

For the second identity, count the $\overline U_\tau$-orbits inside each $\overline P_\tau$-orbit. The $\overline P_\tau$-orbit through $\overline\Gamma\bar r$ is identified with $\overline S(r,\tau)\backslash\overline P_\tau$. Hence the number of $\overline U_\tau$-orbits in it is
\[
|\overline S(r,\tau)\backslash\overline P_\tau/\overline U_\tau|.
\]
Since $\overline U_\tau$ is normal in $\overline P_\tau$, this number is
\[
[\overline P_\tau:\overline S(r,\tau)\overline U_\tau]
=
[\overline P_\tau:\overline U_\tau\overline S(r,\tau)]
=
u_{\sigma|\tau}
\]
by Lemma~\ref{lem:finite-stabilizer-indices}. Summing over all $\sigma\in\cC_\tau(\Gamma)$ gives the second identity. The ratio formula follows.
\end{proof}

\subsection{Product bounds and localization}
\label{sec:product-bounds-localization}

For a finite group $Q$ with subgroups $H,U\subset Q$, set
\[
\chi_{Q,U}(H):=\frac{|H\backslash Q/U|}{[Q:H]}.
\]

\begin{rem}\label{rem:chi-bounded-by-one}
$\chi_{Q,U}(H)\leq 1$, since $|H\backslash Q/U|\leq |H\backslash Q|=[Q:H]$.
\end{rem}

\begin{lem}\label{lem:burnside-expression-for-finite-ratio}
$\chi_{Q,U}(H)=\frac{1}{|Q|}\sum_{h\in H}|X^h|$, where $X:=Q/U$.
\end{lem}

\begin{proof}
By Burnside's lemma, $|H\backslash X|=\frac{1}{|H|}\sum_{h\in H}|X^h|$. Divide by $[Q:H]=|Q|/|H|$.
\end{proof}

\begin{lem}\label{prop:product-bound-finite-ratio}
Assume that $Q=\prod_{i\in I}Q_i$ is a direct product of finite groups, and that $U=\prod_{i\in I}U_i$ with each $U_i\subset Q_i$ a subgroup. Let $H\subset Q$ be a subgroup, and let $I=I_1\sqcup\cdots\sqcup I_m$ be a partition. For each $j$, set
\[
Q_{I_j}:=\prod_{i\in I_j}Q_i,
\qquad
U_{I_j}:=\prod_{i\in I_j}U_i\subset Q_{I_j},
\]
and let $H_{I_j}$ denote the image of $H$ under the projection $Q\twoheadrightarrow Q_{I_j}$. Then
\[
\chi_{Q,U}(H)\leq \prod_{j=1}^m \chi_{Q_{I_j},U_{I_j}}(H_{I_j}).
\]
\end{lem}

\begin{proof}
Set $X:=Q/U$ and $X_{I_j}:=Q_{I_j}/U_{I_j}$. The product decompositions identify $X=\prod_j X_{I_j}$ as $Q$-sets, with $Q=\prod_jQ_{I_j}$ acting componentwise. Decomposing $h\in Q$ as $h=(h_j)_j$ with $h_j\in Q_{I_j}$, this gives $X^h=\prod_j X_{I_j}^{h_j}$. By definition of $H_{I_j}$, $h\in H$ implies $h_j\in H_{I_j}$ for all $j$, so $H\subseteq\prod_j H_{I_j}$. Therefore
\[
\sum_{h\in H}|X^h|
=
\sum_{h\in H}\prod_j |X_{I_j}^{h_j}|
\leq
\sum_{(h_j)\in\prod_j H_{I_j}}\prod_j |X_{I_j}^{h_j}|
=
\prod_{j=1}^m\sum_{h_j\in H_{I_j}}|X_{I_j}^{h_j}|.
\]
Divide by $|Q|=\prod_j |Q_{I_j}|$ and apply Lemma~\ref{lem:burnside-expression-for-finite-ratio} to each factor.
\end{proof}

\begin{lem}\label{lem:fibre-chi-inequality}
Assume that \(Q=A\times B\), \(U=U_A\times U_B\), and let \(H\subset Q\). Let
\(H_A:=H\cap A\), and \(H^B:=\operatorname{pr}_B(H)\) be the projection of \(H\) onto \(B\). Then
\[
  \chi_{Q,U}(H)
  \le
  \chi_{A,U_A}(H_A)\chi_{B,U_B}(H^B).
\]
\end{lem}
\begin{proof}
Let \(X_A=A/U_A\) and \(X_B=B/U_B\).
Fix \(b\in H^B\), and choose \(a_b\in A\) with \((a_b,b)\in H\).  Then
the fibre of \(H\to H^B\) over \(b\) is \(H_Aa_b\).  Moreover \(a_b\)
normalizes \(H_A\), since conjugation by \((a_b,b)\) preserves the kernel
\(H_A\) of \(H\to H^B\).

We claim that
\[
  \sum_{a\in H_Aa_b}|X_A^a|
  \le |H_A|\,|H_A\backslash X_A|.
\]
Indeed, more generally, if a finite group \(K\) acts on a finite set \(X\)
and \(t\) normalizes $K$, then
\[
  \sum_{k\in K}|X^{kt}|
  =
  |K|\, |(K\backslash X)^t|
  \le
  |K|\,|K\backslash X|.
\]
This follows by counting pairs \((k,x)\in K\times X\) such that \(ktx=x\).
A \(K\)-orbit contributes \(|K|\) such pairs if it is preserved by \(t\),
and contributes \(0\) otherwise.  Applying this with \(K=H_A\), \(X=X_A\),
and \(t=a_b\), gives the claim. Therefore
\begin{align*}
  \chi_{Q,U}(H)
  &=
  \frac{1}{|A||B|}
  \sum_{b\in H^B}
  |X_B^b|
  \sum_{a\in H_Aa_b}|X_A^a|                                      \\
  &\le
  \frac{|H_A|\,|H_A\backslash X_A|}{|A|}
  \cdot
  \frac{1}{|B|}\sum_{b\in H^B}|X_B^b|                              \\
  &=
  \chi_{A,U_A}(H_A)\,\chi_{B,U_B}(H^B).
\end{align*}
\end{proof}

Choose a projective $\cO_\bK$-module $L$ of rank two such that $\cM=\operatorname{End}_{\cO_\bK}(L)$. Write the factorization of \(\fn\) into prime ideals:
\[
\fn=\prod_{\fp\mid\fn}\fp^{e_\fp}.
\]
The proof of Lemma~\ref{lem:congruence-product-maximal-order} shows that the reduction map $\SL(L)\to\SL(L/\fn L)$ is surjective, with kernel $\cG(\fn)$. Hence
\[
\overline\cG\simeq \SL(L/\fn L).
\]
The Chinese remainder theorem gives
\[
\cO_\bK/\fn\simeq \prod_{\fp\mid\fn}\cO_\bK/\fp^{e_\fp}.
\]
Tensoring with \(L\) gives
\[
L/\fn L
\simeq
L\otimes_{\cO_\bK}(\cO_\bK/\fn)
\simeq
\prod_{\fp\mid\fn} L\otimes_{\cO_\bK}(\cO_\bK/\fp^{e_\fp})
\simeq
\prod_{\fp\mid\fn} L/\fp^{e_\fp}L.
\]
Thus an $\cO_\bK/\fn$-linear automorphism of $L/\fn L$ is equivalently a
tuple of $\cO_\bK/\fp^{e_\fp}$-linear automorphisms of the modules
$L/\fp^{e_\fp}L$. Since the determinant is computed componentwise under
this product decomposition, we obtain
\[
\overline\cG\simeq
\prod_{\fp\mid\fn}G_{e_\fp}(\fp),
\qquad
G_e(\fp):=\SL(L/\fp^eL).
\]
Since $L/\fp^eL$ is a free module of rank two over $\cO_\bK/\fp^e$, each choice of a basis identifies $G_e(\fp)$ with $\SL_2(\cO_\bK/\fp^e)$. The local computations below are independent of this choice of basis.

Fix a $\cG$-cusp $\tau$ and a chart $g_\tau\in\SL_2(\bK)$ with $g_\tau(\infty)=\tau$. Write
\[
\fb_\tau:=b(\cG,g_\tau).
\]
By \S\ref{sec:cusps-hirzebruch-multiplicity}, this is an
$\cO_\bK$-fractional ideal. Let
\[
N_\tau:=g_\tau
\begin{pmatrix}0&1\\0&0\end{pmatrix}
g_\tau^{-1}.
\]
Then $\fb_\tau=\{x\in\bK\mid xN_\tau\in\cM\}$, and $x\mapsto I+xN_\tau$ is a bijection between $\fb_\tau$ and $U_\tau$.

\begin{lem}\label{lem:cusp-quotient-iso}
For every non-zero ideal $\fa\subseteq\cO_\bK$, the map
$x\mapsto I+xN_\tau$ induces a group isomorphism
\[
\fb_\tau/\fa\fb_\tau \xrightarrow{\sim} U_\tau/(U_\tau\cap\cG(\fa)).
\]
\end{lem}

\begin{proof}
By the discussion in \S\ref{sec:cusps-hirzebruch-multiplicity},
$\fb_\tau=b(\cG,g_\tau)=\{x\in\bK\mid xN_\tau\in\cM\}$ is an
$\cO_\bK$-fractional ideal. The map $x\mapsto I+xN_\tau$ defines an isomorphism from the additive group $\fb_\tau$ onto $U_\tau$.

The kernel of the composite $\fb_\tau\to U_\tau\to
U_\tau/(U_\tau\cap\cG(\fa))$ consists of those $x\in\fb_\tau$ such that
$xN_\tau\in\fa\cM$. Since $\cO_\bK$ is a Dedekind domain, $\fa$ is
invertible. Hence
\[
xN_\tau\in\fa\cM
\,\Longleftrightarrow\,
\fa^{-1}xN_\tau\subset\cM
\,\Longleftrightarrow\,
\fa^{-1}x\subset\fb_\tau
\,\Longleftrightarrow\,
x\in\fa\fb_\tau.
\]
\end{proof}

For each prime $\fp\mid\fn$, write $U_{\tau,e_\fp}(\fp)\subset G_{e_\fp}(\fp)$ for the image of $U_\tau$ under the reduction map $\cG\to G_{e_\fp}(\fp)$. Equivalently,
\[
U_{\tau,e_\fp}(\fp)=U_\tau/(U_\tau\cap\cG(\fp^{e_\fp})).
\]
Lemma~\ref{lem:cusp-quotient-iso}, applied to $\fa=\fn$ and to $\fa=\fp^{e_\fp}$, gives
\[
\fb_\tau/\fn\fb_\tau \simeq \overline U_\tau,
\qquad
\fb_\tau/\fp^{e_\fp}\fb_\tau \simeq U_{\tau,e_\fp}(\fp).
\]
Since $\fb_\tau$ is a fractional $\cO_\bK$-ideal, the Chinese remainder theorem gives
\[
\fb_\tau/\fn\fb_\tau\simeq
\prod_{\fp\mid\fn}\fb_\tau/\fp^{e_\fp}\fb_\tau.
\]
The following statement follows from our discussion.
\begin{lem}\label{lem:U-bar-product-decomposition}
The isomorphism $\overline\cG\simeq \prod_{\fp\mid\fn}G_{e_\fp}(\fp)$ restricts to an isomorphism
$\overline U_\tau\simeq \prod_{\fp\mid\fn}U_{\tau,e_\fp}(\fp)$.
\end{lem}

Moreover, we have:

\begin{lem}\label{lem:local-cusp-unipotent-standard}
For every \(e\ge1\), there is an \(\cO_\bK/\fp^e\)-basis of \(L/\fp^eL\) such that the induced identification
\[
  G_e(\fp)=\SL(L/\fp^eL)\simeq \SL_2(\cO_\bK/\fp^e)
\]
carries \(U_{\tau,e}(\fp)\) onto the standard upper-triangular unipotent subgroup \(U_e\). 
\end{lem}
See \S\ref{sec:setup} for the notation \(U_e\).
\begin{proof}
Let \(\ell_\tau\subset L_\bK\) be the line corresponding to \(\tau\), and put
\(I_\tau=L\cap\ell_\tau\).  By the proof of
Lemma~\ref{lem:ambient-multiplier-square}, both \(I_\tau\) and
\(L/I_\tau\) are rank-one projective \(\cO_\bK\)-modules.  After localizing
at \(\fp\), the sequence
\[
  0\to I_{\tau,\fp}\to L_\fp\to L_\fp/I_{\tau,\fp}\to0
\]
splits, and both end terms are free of rank one.  Choose a generator \(s\) of
\(I_{\tau,\fp}\), and choose \(t\in L_\fp\) whose image generates
\(L_\fp/I_{\tau,\fp}\).  Then \(L_\fp=\cO_{\bK,\fp}s\oplus
\cO_{\bK,\fp}t\), and the reductions of \(s,t\) give an
\(\cO_\bK/\fp^e\)-basis of
\(L/\fp^eL\simeq L_\fp/\fp^eL_\fp\).

Choose \(g_\tau\in\SL_2(\bK)\) with \(g_\tau(\infty)=\tau\), and write
\[
  N_\tau=g_\tau
  \begin{pmatrix}0&1\\0&0\end{pmatrix}
  g_\tau^{-1}.
\]
Then \(N_\tau\) annihilates \(\ell_\tau\) and has image contained in
\(\ell_\tau\). Put \(\cM_\fp=\cM\otimes_{\cO_\bK}\cO_{\bK,\fp}\).  Since \(L\) is projective over \(\cO_\bK\), localization commutes with endomorphism rings:
\[
  \operatorname{End}_{\cO_\bK}(L)\otimes_{\cO_\bK}\cO_{\bK,\fp}
  \simeq
  \operatorname{End}_{\cO_{\bK,\fp}}(L_\fp).
\]
If \(xN_\tau\in\cM_\fp\), then \(sxN_\tau\in\cM\) for some
\(s\notin\fp\), hence \(sx\in\fb_\tau\), so \(x\in\fb_{\tau,\fp}\). Thus, \(\fb_{\tau,\fp}:=\fb_\tau\cO_{\bK,\fp} = \{x\in\bK_\fp\mid xN_\tau\in\cM_\fp\} \). We claim that
\begin{equation}\label{eq:localize-unipotent-identification-formula}
  \fb_{\tau,\fp}N_\tau
  =
  \operatorname{Hom}_{\cO_{\bK,\fp}}(L_\fp/I_{\tau,\fp},I_{\tau,\fp}),
\end{equation}
viewed as endomorphisms of \(L_\fp\) annihilating \(I_{\tau,\fp}\).  Indeed, the
inclusion \(\subset\) is immediate. Conversely, let \(T\in \operatorname{Hom}_{\cO_{\bK,\fp}}(L_\fp/I_{\tau,\fp},I_{\tau,\fp})\), viewed as an endomorphism of \(L_\fp\). After tensoring with \(\bK\), the endomorphism \(T_\bK\) annihilates \(\ell_\tau\) and has image contained in \(\ell_\tau\). The \(\bK\)-vector space of such endomorphisms is generated by \(N_\tau\). Hence \(T_\bK=xN_\tau\) for some \(x\in\bK\). Since \(T\) preserves \(L_\fp\), we have \(xN_\tau\in
\operatorname{End}_{\cO_{\bK,\fp}}(L_\fp)=\cM_\fp\).  By the equality \(\fb_{\tau,\fp}=\{x\in\bK\mid xN_\tau\in\cM_\fp\}\), it follows that \(x\in\fb_{\tau,\fp}\).

With respect to the decomposition
\(L_\fp=\cO_{\bK,\fp}s\oplus\cO_{\bK,\fp}t\), the last Hom-module is
generated by the endomorphism \(t\mapsto s\), \(s\mapsto0\), namely by the elementary matrix \(e_{12}\). By Lemma~\ref{lem:cusp-quotient-iso}, the image \(U_{\tau,e}(\fp)\) is the image of \(\fb_\tau/\fp^e\fb_\tau\) under \(x\mapsto I+xN_\tau\). By \eqref{eq:localize-unipotent-identification-formula}, \(\fb_{\tau,\fp}N_\tau=\cO_{\bK,\fp}e_{12}\). Thus, the reduction of \(\fb_{\tau,\fp}N_\tau\) in
\(\operatorname{End}_{\cO_\bK/\fp^e}(L/\fp^eL)\) is \((\cO_\bK/\fp^e)e_{12}\). Hence, under the induced identification
\(G_e(\fp)\simeq\SL_2(\cO_\bK/\fp^e)\),
\[
  U_{\tau,e}(\fp)
  =
  \left\{
  I+ae_{12}\mid a\in\cO_\bK/\fp^e
  \right\},
\]
which is the subgroup \(U_e\) of \S\ref{sec:setup}.
\end{proof}

We will group the local factors according to their residue characteristic. For
each rational prime \(p\mid N(\fn)\), define
\[
  G(p):=\prod_{\substack{\fp\vert\fn,\; \fp\vert p}}G_{e_\fp}(\fp),
  \qquad
  U_{\tau}(p):=\prod_{\substack{\fp\vert\fn,\; \fp\vert p}}
  U_{\tau,e_\fp}(\fp).
\]
We view \(G(p)\) and \(U_{\tau}(p)\) as subgroups of \(\overline\cG\) by
putting the identity in all factors not lying above \(p\). We have isomorphisms
\[
  \overline\cG\cong\prod_{p\vert N(\fn)}G(p),
  \qquad
  \overline{U}_\tau\cong\prod_{p\vert N(\fn)}U_{\tau}(p).
\]

\begin{lem}\label{lem:primary-product-chi-Gamma}
\[
  \chi_{\overline\cG,\overline U_\tau}(\overline\Gamma)
  =
  \prod_{p\mid N(\fn)}
  \chi_{G(p),U_{\tau}(p)}(\overline\Gamma\cap G(p)).
\]
\end{lem}

\begin{proof}
Set \(X:=\overline\cG/\overline U_\tau\) and
\(X(p):=G(p)/U_{\tau}(p)\). The product decompositions of
\(\overline\cG\) and \(\overline U_\tau\) give
\(X\simeq\prod_{p\mid N(\fn)}X(p)\). Hence, if
\(\gamma=(\gamma_p)_p\in\overline\cG=\prod_{p\mid N(\fn)}G(p)\), then
\[
  |X^\gamma|=\prod_{p\mid N(\fn)}|X(p)^{\gamma_p}|.
\]

If \(g\in G(p)\) fixes a point \(aU_{\tau}(p)\in X(p)\), then
\(g aU_{\tau}(p)=aU_{\tau}(p)\), hence \(a^{-1}ga\in U_{\tau}(p)\).
Thus \(g\) is conjugate in \(G(p)\) to an element of \(U_{\tau}(p)\). By
Lemma~\ref{lem:cusp-quotient-iso}, \(U_{\tau}(p)\) is a \(p\)-group.

Now let \(\gamma=(\gamma_p)_p\in\overline\Gamma\) and suppose that
\(X^\gamma\neq \emptyset\). Then each \(\gamma_p\) fixes a point of \(X(p)\), hence
has \(p\)-power order. The orders of the non-trivial components
\(\gamma_p\) are therefore powers of distinct rational primes. In particular, 
we can raise \(\gamma\) to a power to kill a chosen coordinate. In other words each \(\gamma_p\) is a power of \(\gamma\). Thus
\[
  \gamma_p\in\overline\Gamma\cap G(p)
  \qquad\text{for every }p\mid N(\fn).
\]
Therefore
\[
\begin{aligned}
\sum_{\gamma\in\overline\Gamma}|X^\gamma|
&=
\sum_{(\gamma_p)\in\prod_p(\overline\Gamma\cap G(p))}
\prod_{p\mid N(\fn)}|X(p)^{\gamma_p}|  \\
&=
\prod_{p\mid N(\fn)}
\sum_{\gamma_p\in\overline\Gamma\cap G(p)}|X(p)^{\gamma_p}|.
\end{aligned}
\]
Dividing by \(|\overline\cG|=\prod_{p\mid N(\fn)}|G(p)|\) and using
Lemma~\ref{lem:burnside-expression-for-finite-ratio} for each factor gives
the result.
\end{proof}

\subsection{Global decay}\label{sec:quantitative-global-decay}

Recall that \(\Lambda\subset \PN(\cE)\) for an Eichler order
\(\cE\subset M_2(\bK)\), a maximal order \(\cM\supset\cE\) has been chosen, and we set
\(\cG=\SL(\cM)\), \(\Lambda_\cM=\Lambda\cap\P\cG\), and $\Gamma=\{g\in\cG\mid \Proj(g)\in\Lambda_\cM\}$. The congruence level of \(\Gamma\), relative to \(\cM\), is denoted by \(\fn\).  For a non-zero ideal \(\fa\subset\cO_\bK\), write
\[
N(\fa)=|\cO_\bK/\fa|,\quad
\omega(\fa)=\#\{\fp\mid\fp\vert\fa\}.
\] 
We will assume that the factorization of $\fn$ is
\[
\fn:=\prod_{{\fp\vert\fn}}\fp^{e_\fp}. 
\]
We shall use the notations of Subsection~\ref{sec:product-bounds-localization}.
Thus, if \(\fp\vert\fn\), then
\[
G_{e_\fp}(\fp)=\SL(L/\fp^{e_\fp}L)\cong \SL_2(\cO_\bK/\fp^{e_\fp}),
\]
where \(\cM=\operatorname{End}_{\cO_\bK}(L)\) is the maximal order used to define $\cG$. If \(\tau\) is a \(\cG\)-cusp, then \(U_{\tau,e_\fp}(\fp)\subset G_{e_\fp}(\fp)\) is the image of the unipotent group \(U_\tau\). 
For every prime ideal \(\fp\vert\fn\), view \(G_{e_\fp}(\fp)\) as a factor of \(\overline\cG\simeq\prod_{\fq\vert\fn}G_{e_\fq}(\fq)\), and set
\[
  H_\fp:=\overline\Gamma\cap G_{e_\fp}(\fp).
\]
Recall that the notion of "exact level" is introduced in \S\ref{sec:setup}. The minimality of the congruence level ideal $\fn$ implies:
\begin{lem}\label{lem:local-intersections-exact-level}
The subgroup \(H_\fp\subset G_{e_\fp}(\fp)\) has exact level \(e_\fp\).
\end{lem}

We set
\[
\mathcal R(\Gamma):=
\frac{\sum_{\sigma\in\Gamma\backslash\bP^1(\bK)}a_\sigma(\Gamma)}
{[\cG:\Gamma]}, \qquad 
\mathcal R(\Lambda):=
\frac{\sum_{\kappa\in\Lambda\backslash\bP^1(\bK)}a_\kappa(\Lambda)}
{[\cG:\Gamma]}.
\]

\begin{thm}\label{thm:quantitative-global-cusp-decay}
There are constants \(C>0\) and \(c>0\), depending only on \(\bK\), such
that every \(\Lambda\) as above satisfies
\[
  \mathcal R(\Lambda)\le C\,N(\fn)^{-c}.
\]
\end{thm}

\begin{proof}
Write \(n=[\bK:\mathbb Q]\). 
Fix a cusp \(\tau\in\cT\). By Lemma~\ref{lem:primary-product-chi-Gamma},
\[
  \chi_{\overline\cG,\overline U_\tau}(\overline\Gamma)
  =
  \prod_{p\mid N(\fn)}
  \chi_{G(p),U_\tau(p)}(\overline\Gamma\cap G(p)).
\]

For each rational prime \(p\vert N(\fn)\), set
\[
  \fn(p):=
  \prod_{\substack{\fp\vert\fn,\; \fp\vert p}}\fp^{e_\fp}.
\]
Choose a prime ideal \(\fp(p)\) such that \(\fp(p)\vert p\), \(\fp(p)\vert\fn\), satisfying
\[
  e_{\fp(p)}\log N(\fp(p))
  =
  \max_{\substack{\fp\vert\fn,\; \fp\vert p}}
  e_\fp\log N(\fp).
\]
Since at most \(n\) prime ideals of \(\cO_\bK\) lie above \(p\), we have
\begin{equation}\label{eq:global-exponential-decay-formula-1}
  e_{\fp(p)}\log N(\fp(p))
  \ge
  \frac{1}{n}\log N(\fn(p)).   
\end{equation}

Write
\[
  G(p)=G_{e_{\fp(p)}}(\fp(p))\times G'(p),
  \quad
  U_\tau(p)=U_{\tau,e_{\fp(p)}}(\fp(p))\times U'_\tau(p),
\]
where \(G'(p)\) and \(U'_\tau(p)\) are the products of the remaining factors
above \(p\). Applying Lemma~\ref{lem:fibre-chi-inequality} to
\(H=\overline\Gamma\cap G(p)\), and using
Remark~\ref{rem:chi-bounded-by-one} for the second factor, we get
\[
\begin{aligned}
  \chi_{G(p),U_\tau(p)}(\overline\Gamma\cap G(p))
  &\le
  \chi_{G_{e_{\fp(p)}}(\fp(p)),\,U_{\tau,e_{\fp(p)}}(\fp(p))}
  \bigl(\overline\Gamma\cap G_{e_{\fp(p)}}(\fp(p))\bigr).
\end{aligned}
\]
Therefore
\[
  \chi_{\overline\cG,\overline U_\tau}(\overline\Gamma)
  \le
  \prod_{p\mid N(\fn)}
  \chi_{G_{e_{\fp(p)}}(\fp(p)),\,U_{\tau,e_{\fp(p)}}(\fp(p))}
  \bigl(\overline\Gamma\cap G_{e_{\fp(p)}}(\fp(p))\bigr).
\]

For ease of notation, set \(e_p:=e_{\fp(p)}\), \(q_p:=N(\fp(p))\), and
\[
  H_p:=\overline\Gamma\cap G_{e_p}(\fp(p)).
\]
By Lemma~\ref{lem:local-intersections-exact-level}, each \(H_p\) has exact level \(e_p\). After choosing a local basis as in Lemma~\ref{lem:local-cusp-unipotent-standard}, a local factor $\chi$ in the above formulas is identified with the corresponding \(\chi_{e_p}(H_p)\), written in concrete matrix notations from \S\ref{sec:setup}.

We now divide the rational primes \(p\vert N(\fn)\) into different types. Let \(S\) be the subset of rational primes which are ramified in \(\bK\) or are \(<59\). From
Theorem~\ref{thm:uniform-exponential-local-chi}, there are constants
\(B>0\) and \(\alpha>0\), depending only on \(\bK\), such that, for every
\(p\vert N(\fn)\),
\[
  \chi_{G_{e_p}(\fp(p)),\,U_{\tau,e_p}(\fp(p))}(H_p)
  \le q_p^{B-\alpha e_p}.
\]
If \(e_p>2B/\alpha\), then
\[
  \chi_{G_{e_p}(\fp(p)),\,U_{\tau,e_p}(\fp(p))}(H_p)
  \le q_p^{B-\alpha e_p}
  \le q_p^{-\alpha e_p/2}
  \le q_p^{-\alpha e_p/(2B)}.
\]
If \(e_p\le 2B/\alpha\) and \(p\notin S\), then \(p\) is unramified,
\(p\ge59\), and \(q_p\ge59\). Since \(H_p\) has exact level \(e_p\),
Remark~\ref{rem:the-H1G1-condition-unramified} says that the reduction of $H_p$ modulo $\fp(p)$ is not all of $\SL_2(\cO_\bK/\fp(p))$, i.e.\ \(\rho_{e_p,1}(H_p)\ne \SL_2(\cO_\bK/\fp(p))\) with the notation of \S\ref{sec:setup}. Hence
Theorem~\ref{thm:main_local_chi_B} gives
\[
  \chi_{G_{e_p}(\fp(p)),\,U_{\tau,e_p}(\fp(p))}(H_p)
  <\frac{2}{q_p-1}\le 3q_p^{-1}\le 3q_p^{-\alpha e_p/(2B)}.
\]

Finally, if \(e_p\le 2B/\alpha\) and \(p\in S\), we use only the trivial
estimate \(\chi\le1\). For these exceptional factors, by \eqref{eq:global-exponential-decay-formula-1} we have
\begin{equation}\label{eq:global-exponential-decay-formula-2}
 \log N(\fn(p))
  \le
  \frac{2nB}{\alpha}\log N(\fp(p)).
\end{equation}
Since \(S\) is finite of size depending only on $\bK$, the product of such \(N(\fn(p))\)'s is bounded by a constant depending only on \(\bK\). 

Combining the preceding estimates with \eqref{eq:global-exponential-decay-formula-1}, we obtain constants \(C_1>0\) and \(A_1>0\), depending only on \(\bK\), such that
\[
  \chi_{\overline\cG,\overline U_\tau}(\overline\Gamma)
  \le
  C_1 A_1^{\omega(\fn)}N(\fn)^{-\alpha/(2Bn)}.
\]
By Lemma~\ref{lem:global-reduction-to-weighted-width} and
Proposition~\ref{prop:finite-quotient-counts},
\[
  \mathcal R(\Gamma)
  \le
  C_{\rm UV}
  \sum_{\tau\in\cT}
  \chi_{\overline\cG,\overline U_\tau}(\overline\Gamma).
\]
Therefore
\[
  \mathcal R(\Gamma)
  \le
  C_{\rm UV}|\cT|\,C_1 A_1^{\omega(\fn)}N(\fn)^{-\alpha/(2Bn)}.
\]
By Proposition~\ref{prop:lambda-cusp-multiplicity-controlled-by-gamma},
\[
  \mathcal R(\Lambda)
  \le
  d_\Lambda^{2(n-1)}\mathcal R(\Gamma).
\]
Using Proposition~\ref{prop:d-lambda-level-prime-bound}, we get
\[
  \mathcal R(\Lambda)
  \le
  C_2
  \bigl(2^{2(n-1)}A_1\bigr)^{\omega(\fn)}
  N(\fn)^{-\alpha/(2Bn)}
\]
for some constant \(C_2>0\) depending only on \(\bK\).

It remains to absorb the factor depending on \(\omega(\fn)\). We use only the
elementary fact that, for every fixed \(A>0\) and every \(\eta>0\),
\(A^{\omega(\fn)}N(\fn)^{-\eta}\) is bounded as \(\fn\) varies. Indeed, after
removing the finitely many prime ideals with \(N(\fp)^\eta<A\), each remaining
prime divisor of \(\fn\) is paid for by the factor \(N(\fn)^\eta\).

Applying this with \(A=2^{2(n-1)}A_1\) and
\(\eta=\alpha/(4Bn)\), and enlarging the constant, gives
\[
  \mathcal R(\Lambda)
  \le
  C\,N(\fn)^{-\alpha/(4Bn)}
\]
for some constant \(C>0\) depending only on \(\bK\). This proves the theorem.

\end{proof}

For \(\Delta=\iota_\infty(\Lambda)\), we set
\[
  \mathcal R(\Delta):=\mathcal R(\Lambda).
\]
The discussion in Subsection~\ref{sec:lattices-arithmetic-groups} identifies the
following arithmetic and Lie-theoretic formulations.

\begin{thm}\label{thm:global-field-finiteness-arithmetic}
Fix a number field \(\bK\), and let \(\epsilon>0\). Consider a subgroup
\[
\Lambda\in \cC_\bK^{\mathrm{rat}}\subset \PGL_2(\bK),
\]
commensurable with \(\PSL_2(\cO_\bK)\), congruent with respect to some maximal
order, and satisfying
\[
\mathcal R(\Lambda)>\epsilon.
\]
Then the set of such \(\Lambda\)'s is contained in the union of finitely many
\(\PGL_2(\bK)\)-conjugacy classes.
\end{thm}

\begin{thm}\label{thm:global-field-finiteness-lie}
Fix a number field \(\bK\), and let \(\epsilon>0\). Let
\[
G_\infty=\prod_{v\mid\infty}\PGL_2(\bK_v),
\qquad
\Delta_\bK=\iota_\infty(\PSL_2(\cO_\bK)).
\]
Consider a lattice
\(
\Delta\in \cC_\bK^{\infty}\subset G_\infty
\)
commensurable with \(\Delta_\bK\), congruent with respect to some maximal
order, and satisfying
\[
\mathcal R(\Delta)>\epsilon .
\]
Then the set of such \(\Delta\)'s is contained in the union of finitely many
\(G_\infty\)-conjugacy classes.
\end{thm}

\begin{proof}
It is enough to prove Theorem~\ref{thm:global-field-finiteness-arithmetic}.
Let \(\Lambda\) be as in that theorem. Choose a maximal order with respect to
which \(\Lambda\) is congruent, and let \(\fn\) be the corresponding
congruence level. By Theorem~\ref{thm:quantitative-global-cusp-decay},
\[
  \epsilon<\mathcal R(\Lambda)\le C N(\fn)^{-c}.
\]
Thus \(N(\fn)\) is bounded in terms of \(\bK\) and \(\epsilon\). Hence only
finitely many ideals \(\fn\) can occur. For each fixed \(\fn\),
Proposition~\ref{prop:fixed-level-finitely-many-lie-classes} gives only
finitely many \(\PGL_2(\bK)\)-conjugacy classes. This proves the arithmetic
statement, and the Lie-theoretic statement follows from
Subsection~\ref{sec:lattices-arithmetic-groups}.
\end{proof}

\begin{rem}\label{rem:arithmeticity-and-csp}
Assume that \(\bK\ne\mathbb Q\) and that \(\bK\) is not imaginary quadratic.
Then \(G_\infty\) has real rank at least \(2\). By Margulis' arithmeticity
theorem, every irreducible lattice in \(G_\infty\) is arithmetic
\cite[Theorem~IX.1.11]{Margulis}. If such a lattice is non-cocompact, the
corresponding \(\bK\)-group is \(\bK\)-isotropic. For a quaternion algebra over
\(\bK\), this is equivalent to being split. Thus every non-cocompact
irreducible lattice in \(G_\infty\) is, up to commensurability and
\(G_\infty\)-conjugacy, in the class \(\cC_\bK^\infty\).

Assume moreover that \(\bK\) has a real place. Then Serre's theorem says that every arithmetic subgroup
commensurable with \(\SL_2(\cO_\bK)\) is congruence, see \cite[Théorème~2 and Corollaire~3]{SerreSL2CSP} (cf.\ \cite{BassMilnorSerreCSP}). The same assertion holds
relative to any maximal order. Indeed, let
\(\cM\subset M_2(\bK)\) be maximal. The two \(\cO_\bK\)-lattices
\(\cM\) and \(M_2(\cO_\bK)\) are commensurable. Hence, for every non-zero
ideal \(\fa\), there is a non-zero ideal \(\fb\) such that
\[
\fb\cM\subset \fa M_2(\cO_\bK).
\]
It follows that
\[
\SL(\cM)(\fb)\subset \SL_2(\cO_\bK)(\fa).
\]
Thus the congruence property for \(\SL_2(\cO_\bK)\) implies the congruence
property for \(\SL(\cM)\).

Therefore if $\bK$ is not rational or totally imaginary, then Theorem~\ref{thm:global-field-finiteness-lie} works for arbitrary non-cocompact irreducible lattices. 
\end{rem}

\appendix
\section{Relation with the work of Finis--Lapid}
\label{sec:finis-lapid}
For each fixed number field $\bK$, a qualitative variant of
Theorem~\ref{mainthmB}, with constants allowed to depend on
$\bK$, follows from
\cite[Corollary~5.28]{FinisLapid}.

Let $\mathbb A_{\mathrm{fin}}$ denote the ring of finite adeles of
$\bQ$, namely the restricted product of the fields $\bQ_\ell$ over
all rational primes $\ell$, with respect to the subrings $\bZ_\ell$.
Set $\widehat{\bZ}:=\prod_\ell\bZ_\ell$. For every finite prime
$\mathfrak q$ of $\bK$, let $\bK_{\mathfrak q}$ be the completion of
$\bK$ at $\mathfrak q$, with ring of integers
$\cO_{\bK_{\mathfrak q}}$.

Let $\mathbf G=\operatorname{Res}_{\bK/\bQ}\SL_2$. There is a
canonical identification
\[
\mathbf G(\mathbb A_{\mathrm{fin}})
\simeq
\prod_{\mathfrak q}'\SL_2(\bK_{\mathfrak q}),
\]
where the restricted product is taken with respect to the subgroups
$\SL_2(\cO_{\bK_{\mathfrak q}})$. Choose a $\bZ$-basis of
$\cO_{\bK}$ and let
$\rho_0:\mathbf G\to\GL_{2[\bK:\bQ]}$ be the faithful
$\bQ$-rational representation induced by the action of $\SL_2(\bK)$
on $\bK^2$. The compact open subgroup determined by this
representation is
\[
\mathbf K
:=
\rho_0^{-1}\bigl(\GL_{2[\bK:\bQ]}(\widehat{\bZ})\bigr)
=
\prod_{\mathfrak q}\SL_2(\cO_{\bK_{\mathfrak q}})
\subseteq\mathbf G(\mathbb A_{\mathrm{fin}}).
\]

Let $\mathbf P\subseteq\mathbf G$ be the restriction of scalars of
the standard upper triangular Borel subgroup of $\SL_{2,\bK}$, and
let $\mathbf U$ be its unipotent radical. Under the preceding
identification,
\[
\mathbf U(\mathbb A_{\mathrm{fin}})\cap\mathbf K
=
\prod_{\mathfrak q}
\left\{
\begin{pmatrix}
1&x\\
0&1
\end{pmatrix}
:x\in\cO_{\bK_{\mathfrak q}}
\right\}.
\]

Let $\mathfrak p$ be a prime of $\bK$ above the rational prime $p$,
and let $H\subseteq G_e$. The canonical isomorphism
$\cO_{\bK}/\mathfrak p^e
\simeq
\cO_{\bK_{\mathfrak p}}/
\mathfrak p^e\cO_{\bK_{\mathfrak p}}$
identifies
\[
G_e
\simeq
\SL_2\bigl(
\cO_{\bK_{\mathfrak p}}/
\mathfrak p^e\cO_{\bK_{\mathfrak p}}
\bigr).
\]
Let $\widetilde H_{\mathfrak p}$ be the inverse image of $H$ under
the reduction map
\[
\SL_2(\cO_{\bK_{\mathfrak p}})
\longrightarrow G_e,
\]
and define the open subgroup
\[
\mathcal K_H
:=
\widetilde H_{\mathfrak p}
\times
\prod_{\mathfrak q\neq\mathfrak p}
\SL_2(\cO_{\bK_{\mathfrak q}})
\subseteq\mathbf K.
\]

Give $\mathbf K$ Haar measure of total mass one. On
$\mathbf U(\mathbb A_{\mathrm{fin}})$, take the product Haar measure
for which each subgroup
$\mathbf U(\bK_{\mathfrak q})
\cap\SL_2(\cO_{\bK_{\mathfrak q}})$ has measure one. Consider
\[
I_H
:=
\int_{\mathbf U(\mathbb A_{\mathrm{fin}})}
\int_{\mathbf K}
\mathbf 1_{\mathcal K_H}(k^{-1}uk)\,dk\,du.
\]
Since $\mathcal K_H$ agrees with $\mathbf K$ at every
$\mathfrak q\neq\mathfrak p$, all factors at these primes are equal
to one. At $\mathfrak p$, reduction modulo $\mathfrak p^e$ gives
\[
I_H
=
\frac{1}{|U_e||G_e|}
\sum_{u\in U_e}\sum_{g\in G_e}
\mathbf 1_H(g^{-1}ug).
\]
Burnside's lemma, applied to the action of $U_e$ on the coset space
$G_e/H$, yields
\[
I_H
=
\frac{|U_e\backslash G_e/H|}{[G_e:H]}
=
\frac{|H\backslash G_e/U_e|}{[G_e:H]}
=
\chi_e(H),
\]
where the middle equality follows by inversion in $G_e$.

Assume that $H$ has exact level $e$. Then
$\widetilde H_{\mathfrak p}$ contains
\[
\ker\bigl(
\SL_2(\cO_{\bK_{\mathfrak p}})
\longrightarrow
\SL_2(\cO_{\bK_{\mathfrak p}}/
\mathfrak p^e\cO_{\bK_{\mathfrak p}})
\bigr),
\]
but does not contain the analogous principal congruence subgroup of
depth $e-1$. Let $e_0$ be the ramification index of
$\mathfrak p$ over $p$. Since
$p^m\cO_{\bK_{\mathfrak p}}
=\mathfrak p^{e_0m}\cO_{\bK_{\mathfrak p}}$,
the $p$-part of the rational congruence level of $\mathcal K_H$,
with respect to $\rho_0$, is exactly
$p^{\lceil e/e_0\rceil}$.

The group $\mathbf G$ is semisimple and simply connected, and the
smallest normal $\bQ$-subgroup of $\mathbf G$ containing
$\mathbf U$ is $\mathbf G$. Hence the relative level appearing in
\cite[Corollary~5.28]{FinisLapid} is, in this case, the ordinary
congruence level of $\mathcal K_H$. The cited corollary therefore
gives constants $C_{\bK},\varepsilon_{\bK}>0$, independent of
$\mathfrak p$, $e$, and $H$, such that
\[
\chi_e(H)
=
I_H
\leq
C_{\bK}p^{-\varepsilon_{\bK}\lceil e/e_0\rceil}.
\]
Writing $|\cO_{\bK}/\mathfrak p|=p^f$ and using
$e_0f=[\bK_{\mathfrak p}:\bQ_p]\leq[\bK:\bQ]$, we obtain
\[
p^{-\varepsilon_{\bK}\lceil e/e_0\rceil}
\leq
|\cO_{\bK}/\mathfrak p|^{
-\varepsilon_{\bK}e/[\bK:\bQ]}.
\]
Consequently, there exist constants $C_{\bK},c_{\bK}>0$ such that
\[
\chi_e(H)
\leq
C_{\bK}
|\cO_{\bK}/\mathfrak p|^{-c_{\bK}e}.
\]

\section{Minimal models of Hilbert Modular Surfaces}\label{sec:Hilbert-modular-surfaces}
Let \(\bK\) be a real quadratic field. By a generalized full Hilbert
modular group over \(\bK\) we mean a group of the form
\[
  \Omega_L
  =
  \operatorname{im}\bigl(\operatorname{SL}(L)\to\operatorname{PGL}_2(\bK)\bigr),
\]
where \(L\) is a rank-two projective \(\cO_\bK\)-module and
\[
  \operatorname{SL}(L)
  =
  \{g\in\operatorname{Aut}_{\cO_\bK}(L)\mid\det(g)=1\}.
\]
Equivalently, after writing \(L\simeq\cO_\bK\oplus I\), this is the
projective image of
\[
  \left\{
  \begin{pmatrix}
  a&b\\ c&d
  \end{pmatrix}
  \mid 
  a,d\in\cO_\bK,\ b\in I^{-1},\ c\in I,\ ad-bc=1
  \right\}.
\]

Let $\Gamma$ be a lattice contained in a generalized Hilbert modular group. Then it is necessarily a congruence subgroup, see Remark~\ref{rem:arithmeticity-and-csp}. 
Bailey--Borel's theorem~\cite{BB1966} implies that $X(\Gamma):=(\bH\times \bH)/\Gamma$ is a quasi-projective surface which is projective when $\Gamma$ is cocompact, where  $\bH\subset \bC$ is the upper half plane. 
Hirzebruch~\cite{Hirzebruch1973} discovered how to resolve $X(\Gamma)$ to get a smooth projective surface. There is a projective surface $\overline{X}(\Gamma)$ compactifying $X(\Gamma)$ such that the boundary $\overline{X}(\Gamma)\setminus X(\Gamma)$ is a finite union of isolated points in bijection with cusps of $\Gamma$. The surface $\overline{X}(\Gamma)$ may have two types of isolated singularities: cusps, and cyclic quotient singularities caused by torsion in $\Gamma$. After resolving these singularities of $\overline{X}$ in a minimal way by cycles of rational curves, we obtain a smooth complex projective surface $Y(\Gamma)$ whose isomorphism class only depends on the conjugacy class of $\Gamma$ in $\PSL(2,\bR)^2$.

Assume that \(\Gamma\) is non-cocompact. We use the notation of
Section~\ref{sec:cusps-hirzebruch-multiplicity}; the cusps are cosets
\(\sigma\in\Gamma\backslash\bP^1(\bK)\), and each cusp has a cusp type
\((M_\sigma,V_\sigma)\) and a multiplicity
\[
a_\sigma(\Gamma)=[U_{M_\sigma}^+:V_\sigma]. 
\]

Two complete $\bZ$-modules $M,M'$ in $\bK$ are said to be \emph{strictly equivalent} if $M=\lambda M'$ for some totally positive $\lambda\in \bK_+^*$. Hirzebruch~\cite{Hirzebruch1973} associates to the strict class of the module \(M_\sigma\) a primitive cycle of integers
\[
C_\sigma^\circ=(b_{\sigma,1}^\circ,\ldots,b_{\sigma,s_\sigma}^\circ),
\qquad b_{\sigma,i}^\circ\ge 2,
\]
with at least one \(b_{\sigma,i}^\circ\ge3\). This is the period of the
continued-fraction cycle attached to the convex hull of the totally positive
elements of \(M_\sigma\). By \cite[\S2.5]{Hirzebruch1973}, the minimal
resolution of the cusp singularity corresponding to \(\sigma\) has cyclic
type
\[
C_\sigma=(C_\sigma^\circ)^{a_\sigma}.
\]
Geometrically this means the following. For \(r_\sigma\ge2\), the exceptional divisor over \(\sigma\) is a cycle
\[
E_{\sigma,1}+\cdots+E_{\sigma,r_\sigma}
\]
of smooth rational curves satisfying
\[
E_{\sigma,i}^2=-b_{\sigma,i},\qquad
E_{\sigma,i}E_{\sigma,i+1}=1
\]
cyclically, all other intersections being zero.

\begin{lem}\label{lem:R-stable-lattices-finite}
Let \(R\subset \cO_\bK\) be a subring containing \(1\) such that \(\cO_\bK/R\) is finite. Consider the set of complete \(\mathbb Z\)-lattices \(M\subset \bK\) satisfying $RM\subset M$. There are only finitely many strict equivalence classes of such lattices.
\end{lem}

\begin{proof}
Since \(\cO_\bK/R\) is finite, there exists an integer \(m\ge 1\) such that \(m\cO_\bK\subset R\). We fix fractional ideals \( J_1,\ldots,J_h\) representing the narrow ideal classes of \(\bK\).

Let \(M\subset\bK\) be a complete \(\mathbb Z\)-submodule with \(RM\subset M\). Set $J:=\cO_\bK M$. Then \(J\) is a fractional ideal. 
Thus there exist \(i\) and \(\lambda\in\bK_+^*\) such that \(\lambda J=J_i\). 

We have \(M\subset J\) and 
\[
  mJ=m\cO_\bK M\subset RM\subset M.
\]
Multiplying the inclusions \(mJ\subset M\subset J\) by \(\lambda\), we get
\[
  mJ_i\subset \lambda M\subset J_i.
\]
Therefore every strict equivalence class has a representative \(M'\) satisfying 
\(mJ_i\subset M'\subset J_i\) for one of the finitely many \(J_i\).

The possible subgroups \(M'\) with \(mJ_i\subset M'\subset J_i\) correspond to subgroups of the finite abelian group \(J_i/mJ_i\). Hence there are only finitely many such \(M'\) for each \(i\), and therefore only finitely many strict equivalence classes.
\end{proof}

\begin{proof}[Proof of Theorem~\ref{mainthm:large-level-normal-minimal}]
Let \(\Delta_\Gamma\) be the reduced exceptional divisor on \(Y_\Gamma\) obtained from the cusp resolution.

\smallskip\noindent
\emph{Step 1.} 
Choose a non-rational unit \(v\in(\cO_\bK^*)^2\), which is possible by Dirichlet's unit theorem. Set \(R=\bZ[v]\subset\cO_\bK\). Then \(\cO_\bK/R\) is finite. Let
\(\sigma\) be a \(\Gamma\)-cusp, choose \(g\in\SL_2(\bK)\) with
\(g(\infty)=\sigma\). By Lemma~\ref{lem:ambient-multiplier-square} the stabilizer of \(\sigma\) in \(\Omega\) contains an element $ghg^{-1}$ where $h$ has the form
\[
h= \begin{pmatrix}a&b\\0&a^{-1}\end{pmatrix}, \; a^2=v.
\]
Note that \(h u(x)h^{-1}=u(vx)\). Since \(\Gamma\triangleleft\Omega\), if
\(x\in b(\Gamma,g)\), then
\[
  g u(vx)g^{-1}=gh u(x)(gh)^{-1}\in\Gamma.
\]
Hence \(Rb(\Gamma,g)\subset b(\Gamma,g)\).

By Lemma~\ref{lem:R-stable-lattices-finite}, as \(\Omega\), \(\Gamma\), and \(\sigma\) vary, only finitely many strict equivalence classes of the $\bZ$-modules \(b(\Gamma,g)\) can occur; recall that strict equivalence means multiplication by an element of \(\bK_+^*\). The primitive cycle of Hirzebruch's resolution of one cusp depends only on this strict equivalence class. Therefore only finitely many
primitive cusp cycles can occur.

Let these primitive cycles be
\[
  (-b_{1,1},\ldots,-b_{1,\ell_1}),\ldots,
  (-b_{s,1},\ldots,-b_{s,\ell_s}).
\]
Set
\[
  B_\bK
  =
  \max\left\{
  1,\,
  \sum_{j=1}^{\ell_i}(b_{i,j}-2)\mid 1\le i\le s
  \right\}.
\]
For each cusp \(\sigma\), the cusp cycle is obtained by repeating its primitive cycle \(a_\sigma(\Gamma)\) times. If \(C\) is an irreducible component with \(C^2=-b\), then adjunction formula gives \(K_{Y_\Gamma}C=b-2\). Hence \(K_{Y_\Gamma}\Delta_{\Gamma,\sigma} \le B_\bK a_\sigma(\Gamma)\). Summing over the cusps gives
\begin{equation}\label{eq:appendix-hilbert-modular-formula-1}
K_{Y_\Gamma}\Delta_\Gamma   \le   B_\bK \sum_{\sigma\in\Gamma\backslash\bP^1(\bK)} a_\sigma(\Gamma). 
\end{equation}

\smallskip\noindent
\emph{Step 2.} 
The second part of the proof is the same as van der Geer's, see \cite[Theorem~VII.7.19]{VDG}. Assume that \(E\subset Y_\Gamma\) is a $(-1)$-curve. As $\Gamma$ is normal, the quotient group \(G=\Omega/\Gamma\) acts on \(Y_\Gamma\) and the other components of the $G$-orbit of $E$ are all disjoint $(-1)$-curves (this is the place where we use normality). The proofs of \cite[Proposition~7.18]{VDG} and \cite[Theorem~VII.7.19]{VDG} show that the $G$-orbit \(E_1,\ldots,E_r\) satisfies
\[
\sum_{i=1}^r E_i\Delta_\Gamma\ge r+\frac{1}{3} [\Omega:\Gamma].
\]
Then the paragraph after \cite[Equation~VII.(6)]{VDG} shows that \begin{equation}\label{eq:appendix-hilbert-modular-formula-2}
K_{Y_\Gamma}\Delta_\Gamma\ge \frac{1}{3}[\Omega:\Gamma].
\end{equation}
This step uses the fact that the canonical bundle is non negative because $Y_\Gamma$ is not rational.

\smallskip\noindent
\emph{Step 3.} 
Equations \eqref{eq:appendix-hilbert-modular-formula-1} and \eqref{eq:appendix-hilbert-modular-formula-2} are not compatible for large $\fn$ by Theorem~\ref{thm:quantitative-global-cusp-decay}.
\end{proof}

\addtocontents{toc}{\protect\setcounter{tocdepth}{-1}}
\section*{Acknowledgements}
\addtocontents{toc}{\protect\setcounter{tocdepth}{2}}

I thank Song-Yan Xie for related discussions and comments on a previous version. I thank Luanyin Shen for suggestions.

I acknowledge partial support from the French National Research Agency under the projects GAG (ANR-24-CE40-3526-01) and DynAtrois (ANR-24-CE40-1163).

\medskip
\bibliographystyle{amsplain}
\bibliography{references_hilbert} 
\end{document}